\documentclass[a4paper]{amsart}
\usepackage{etex}
\usepackage{stackrel}
\usepackage[utf8]{inputenc}
\usepackage[english]{babel}
\usepackage{amsmath}
\usepackage{amssymb}
\usepackage{amsopn}                 % defines new operators
\usepackage{amsbsy}
\usepackage{amsfonts}
\usepackage{MnSymbol}
\usepackage{amsthm}
\usepackage{newlfont}
\usepackage{amscd}
\usepackage{mathtools}              % \MoveEqLeft, \intertext
\usepackage{mathrsfs}
\usepackage{relsize}
\usepackage{tensor}
\usepackage{enumerate}
\usepackage{enumitem}
\usepackage{comment}
\usepackage{float}
\usepackage[abbreviations]{foreign} % i.e., e.g.
\usepackage{eucal}
\usepackage{xy}
\xyoption{arrow}
\xyoption{matrix}
\SilentMatrices
\xyoption{tips}
\SelectTips{cm}{10}
\xyoption{2cell}
\UseAllTwocells
\usepackage{tikz}
\usetikzlibrary{cd}
\usetikzlibrary{calc}
\usetikzlibrary{math}
\usetikzlibrary{arrows}
\usetikzlibrary{fit}
\usetikzlibrary{positioning}

\usetikzlibrary{decorations.pathmorphing, arrows.meta}

\tikzset{Rightarrow/.style={double equal sign distance,>={Implies},->},
triple/.style={-,preaction={draw,Rightarrow}},
quadruple/.style={preaction={draw,Rightarrow,shorten >=0pt},shorten >=1pt,-,double,double
distance=0.2pt}}
\usepackage{xcolor}
\definecolor{darkblue}{rgb}{0,0,0.3}
\usepackage[colorlinks=true, citecolor=darkblue, filecolor=darkblue, linkcolor=darkblue, urlcolor=darkblue]{hyperref}
\newtheorem{thm}{Theorem}%[subsection]

\newtheorem{cor}[thm]{Corollary}

\newtheorem{lemma}[thm]{Lemma}
\newtheorem{prop}[thm]{Proposition}

\theoremstyle{definition}
\newtheorem{define}[thm]{Definition}
\newtheorem{notate}[thm]{Notation}

\theoremstyle{remark}
\newtheorem{rem}[thm]{Remark}
\newtheorem{example}[thm]{Example}
\newtheorem{const}[thm]{Construction}

\newtheorem{warning}[thm]{Warning}

{\parskip=0pt\par\nopagebreak\centering}
{\par\noindent\ignorespacesafterend}        % center tikzpictures
\newcommand\nbd\nobreakdash
\newcommand{\ndef}{\emph}

\def\lrar{\longrightarrow}
\newcommand{\ZZ}{\mathbb{Z}}

\newcommand{\bS}{\mathbf{S}}        % generating set
\newcommand{\op}{\mathrm{op}}

\newcommand{\Cat}{{\mathcal{C}\mspace{-2.mu}\mathit{at}}}
\newcommand{\nCat}[1]{{#1}\hbox{\protect\nbd-}\kern1pt\Cat}	    % n-categories
\newcommand{\inn}{\mathrm{inn}}
\newcommand{\out}{\mathrm{out}}
\newcommand{\s}{\mathcal{S}\mspace{-2.mu}\text{et}_{\Delta}}
\newcommand{\sca}{\mathrm{sc}}
\newcommand{\Ss}{\mathcal{S}\mspace{-2.mu}\text{et}_{\Delta}^{\,\mathrm{sc}}}
\newcommand{\Sms}{\mathcal{S}\mspace{-2.mu}\text{et}_{\Delta}^{+, \mathrm{sc}}}

\newcommand{\St}{\mathcal{S}\mspace{-2.mu}\text{t}} % stratified
\newcommand{\Nsc}{\mathrm{N}^{\mathrm{sc}}}			% scaled nerve
\newcommand{\ho}{\mathrm{ho}}
\newcommand{\fC}{\mathfrak{C}}						% coherent realization
\def\Beta{\mathfrak{P}}                             % pattern

\newcommand{\Catoo}{\Cat_{\infty}}
\newcommand{\BiCat}{\mathrm{BiCat}_{\infty}}

\newcommand{\Fun}{\mathrm{Fun}}

\newcommand\pdfoo{\texorpdfstring{$\infty$}{oo}}

\DeclareMathOperator{\Hom}{Hom}

\DeclareMathOperator{\Map}{Map}

\DeclareMathOperator{\Id}{id}
\DeclareMathOperator{\Ob}{Ob}

\DeclareMathOperator{\co}{co}
\DeclareMathOperator{\coop}{coop}
\DeclareMathOperator{\ad}{ad}

\DeclareMathOperator{\fib}{fib}

\DeclareMathOperator{\cosk}{cosk}

\DeclareMathOperator{\Set}{Set}
\newcommand{\gr}{\mathrm{gr}}               % Gray
\newcommand{\mgr}{\otimes}
\newcommand{\hmgr}{\Box_{\gr}}
\newcommand{\opgr}{\mathrm{opgr}}
\newcommand{\LMap}{\Fun^{\opgr}}
\newcommand{\RMap}{\Fun^{\gr}}
\newcommand{\LNat}{\Nat^{\opgr}}
\newcommand{\RNat}{\Nat^{\gr}}

\DeclareMathOperator{\Nat}{Nat}
\DeclareMathOperator{\Str}{St}
\DeclareMathOperator{\Un}{Un}

\DeclareMathOperator{\coc}{coc}

\DeclareMathOperator{\thi}{th}
\def\alp{{\alpha}}
\def\bet{{\beta}}
\def\gam{{\gamma}}

\def\eps{{\varepsilon}}

\def\sig{{\sigma}}
\def\vphi{{\varphi}}

\def\Del{{\Delta}}

\def\Lam{{\Lambda}}

\def\vphi{{\varphi}}

\def\hrar{\hookrightarrow}

\def\x{\stackrel}

\def\ovl{\overline}

\def\wtl{\widetilde}

\newcommand{\tr}[2]{\mathchoice
	{#1\raise -1.8pt\vbox{\hbox{$\kern -.8pt/\mathsmaller{#2} $}}}
	{#1\raise -1.8pt\vbox{\hbox{$\kern -.8pt/#2$}}\kern .8pt}
	{#1\raise -1.8pt\vbox{\hbox{$\scriptstyle\kern -.8pt /#2$}}}
	{#1\raise -1.8pt\vbox{\hbox{$\scriptscriptstyle\kern -.8pt /#2$}}}}

\newcommand{\trbis}[2]{\mathchoice
	{#1\raise -1.8pt\vbox{\hbox{$\kern -.8pt\mathsmaller{/#2} $}}}
	{#1\raise -1.8pt\vbox{\hbox{$\kern -.8pt\mathsmaller{/#2}$}}\kern .8pt}
	{#1\raise -1.8pt\vbox{\hbox{$\scriptstyle\kern -.8pt /#2$}}}
	{#1\raise -1.8pt\vbox{\hbox{$\scriptscriptstyle\kern -.8pt /#2$}}}}
\newcommand{\overslice}[2]{\mathchoice
	{#1\raise -1.8pt\vbox{\hbox{$\kern -.8pt\mathsmaller{#2/} $}}}
	{#1\raise -1.8pt\vbox{\hbox{$\kern -.8pt\mathsmaller{#2/}$}}\kern .8pt}
	{#1\raise -1.8pt\vbox{\hbox{$\scriptstyle\kern -.8pt #2/$}}}
	{#1\raise -1.8pt\vbox{\hbox{$\scriptscriptstyle\kern -.8pt #2/$}}}}

	\makeatletter

\def\labelstylecode@triangle#1{%
	\pgfkeys@split@path%
	\edef\label@key{/triangle/label/\pgfkeyscurrentname}%
	\edef\style@key{\pgfkeyscurrentkey/.@val}%
	\def\temp@a{#1}%
	\def\temp@b{\pgfkeysnovalue}%
	\ifx\temp@a\temp@b
	\pgfkeysgetvalue{\label@key}\temp@a
	\ifx\temp@a\temp@b\else
	\pgfkeysalso{commutative diagrams/.cd, \style@key}%
	\fi
	\else
	\pgfkeys{\style@key/.code = \pgfkeysalso{#1}}%
	\fi}
\def\arrowstylecode@triangle#1{%
	\edef\style@key{\pgfkeyscurrentkey/.@val}%
	\def\temp@a{#1}%
	\def\temp@b{\pgfkeysnovalue}%
	\ifx\temp@a\temp@b
	\pgfkeysalso{commutative diagrams/.cd, \style@key}%
	\else
	\pgfkeys{\style@key/.code = \pgfkeysalso{#1}}%
	\fi}

\pgfkeys{
	/triangle/label/.cd,
	0/.initial = {$\bullet$}, 1/.initial = {$\bullet$}, 2/.initial = {$\bullet$},
	01/.initial, 12/.initial, 02/.initial,
	012/.initial,
	/triangle/labelstyle/.cd,
	01/.@val/.initial=\pgfkeysalso{swap}, 12/.@val/.initial=\pgfkeysalso{swap},
	02/.@val/.initial,
	012/.@val/.initial=\pgfkeysalso{near start},
	01/.code=\labelstylecode@triangle{#1}, 02/.code=\labelstylecode@triangle{#1},
	12/.code=\labelstylecode@triangle{#1}, 012/.code=\labelstylecode@triangle{#1},
	/triangle/arrowstyle/.cd,
	01/.@val/.initial, 12/.@val/.initial,
	02/.@val/.initial,
	012/.@val/.initial=\pgfkeysalso{Rightarrow},
	01/.code=\arrowstylecode@triangle{#1}, 02/.code=\arrowstylecode@triangle{#1},
	12/.code=\arrowstylecode@triangle{#1}, 012/.code=\arrowstylecode@triangle{#1},
}

\def\tr@abc{%
	\draw [/triangle/arrowstyle/012] (90:0.20) --
	node [/triangle/labelstyle/012] {
		\pgfkeysvalueof{/triangle/label/012}} (270:0.10);
}

\def\tr@#1#2{
	\begin{scope}[shift=#2, commutative diagrams/every diagram]
		
		\node (n{#1}0) at (150:1) {
			\pgfkeysvalueof{/triangle/label/0}};
		\node (n{#1}1) at (270:0.6) {
			\pgfkeysvalueof{/triangle/label/1}};
		\node (n{#1}2) at (30:1) {
			\pgfkeysvalueof{/triangle/label/2}};			
		
		\node (s#1) at (0,0) [circle, inner sep = 0pt,
		fit = (n{#1}0.center)(n{#1}1.center)(n{#1}2.center)] {};
		
		\begin{scope}[commutative diagrams/.cd, every arrow, every label]
			\ifcase #1
			\def\list{0/1, 1/2, 0/2}\or
			\def\list{0/1, 1/2, 0/2}\else
			\def\list{}\fi
			
			\foreach \s / \e in \list {
				\draw [/triangle/arrowstyle/\s\e] (n{#1}\s) --
				node [/triangle/labelstyle/\s\e] {
					\pgfkeysvalueof{/triangle/label/\s\e}} (n{#1}\e);
			}
			
			\ifcase #1
			\tr@abc\or
			\tr@abc
			\else\fi
			
		\end{scope}
	\end{scope}
}

\def\triangle#1{
	\pgfkeys{#1}
	\tr@{0}{(0:0)}
}

\makeatother

\makeatletter

\def\labelstylecode@square#1{%
	\pgfkeys@split@path%
	\edef\label@key{/square/label/\pgfkeyscurrentname}%
	\edef\style@key{\pgfkeyscurrentkey/.@val}%
	\def\temp@a{#1}%
	\def\temp@b{\pgfkeysnovalue}%
	\ifx\temp@a\temp@b
	\pgfkeysgetvalue{\label@key}\temp@a
	\ifx\temp@a\temp@b\else
	\pgfkeysalso{commutative diagrams/.cd, \style@key}%
	\fi
	\else
	\pgfkeys{\style@key/.code = \pgfkeysalso{#1}}%
	\fi}
\def\arrowstylecode@square#1{%
	\edef\style@key{\pgfkeyscurrentkey/.@val}%
	\def\temp@a{#1}%
	\def\temp@b{\pgfkeysnovalue}%
	\ifx\temp@a\temp@b
	\pgfkeysalso{commutative diagrams/.cd, \style@key}%
	\else
	\pgfkeys{\style@key/.code = \pgfkeysalso{#1}}%
	\fi}

\pgfkeys{
	/square/label/.cd,
	0/.initial = {$\bullet$}, 1/.initial = {$\bullet$}, 2/.initial = {$\bullet$},
	3/.initial = {$\bullet$},
	01/.initial, 12/.initial, 23/.initial,
	02/.initial, 03/.initial, 13/.initial,
	012/.initial, 013/.initial, 023/.initial, 123/.initial,
	0123/.initial,
	/square/labelstyle/.cd,
	01/.@val/.initial=\pgfkeysalso{swap},
	12/.@val/.initial=\pgfkeysalso{swap},
	23/.@val/.initial=\pgfkeysalso{swap},
	03/.@val/.initial,
	02/.@val/.initial=\pgfkeysalso{description},
	13/.@val/.initial=\pgfkeysalso{description},
	012/.@val/.initial=\pgfkeysalso{above left = -1pt and -1pt},
	013/.@val/.initial=\pgfkeysalso{swap, right},
	023/.@val/.initial=\pgfkeysalso{left},
	123/.@val/.initial=\pgfkeysalso{swap, above right = -1pt and -1pt},
	0123/.@val/.initial,
	01/.code=\labelstylecode@square{#1}, 02/.code=\labelstylecode@square{#1},
	03/.code=\labelstylecode@square{#1}, 12/.code=\labelstylecode@square{#1},
	13/.code=\labelstylecode@square{#1}, 23/.code=\labelstylecode@square{#1},
	012/.code=\labelstylecode@square{#1}, 013/.code=\labelstylecode@square{#1},
	023/.code=\labelstylecode@square{#1}, 123/.code=\labelstylecode@square{#1},
	0123/.code=\labelstylecode@square{#1},
	/square/arrowstyle/.cd,
	01/.@val/.initial, 12/.@val/.initial, 23/.@val/.initial,
	02/.@val/.initial, 03/.@val/.initial, 13/.@val/.initial,
	012/.@val/.initial=\pgfkeysalso{Rightarrow},
	013/.@val/.initial=\pgfkeysalso{Rightarrow},
	023/.@val/.initial=\pgfkeysalso{Rightarrow},
	123/.@val/.initial=\pgfkeysalso{Rightarrow},
	0123/.@val/.initial=\pgfkeysalso{Rightarrow},
	01/.code=\arrowstylecode@square{#1}, 02/.code=\arrowstylecode@square{#1},
	03/.code=\arrowstylecode@square{#1}, 12/.code=\arrowstylecode@square{#1},
	13/.code=\arrowstylecode@square{#1}, 23/.code=\arrowstylecode@square{#1},
	012/.code=\arrowstylecode@square{#1}, 013/.code=\arrowstylecode@square{#1},
	023/.code=\arrowstylecode@square{#1}, 123/.code=\arrowstylecode@square{#1},
	0123/.code=\arrowstylecode@square{#1}
}

\def\sq@abc{%
	\draw [/square/arrowstyle/012] (235:0.25) --
	node [/square/labelstyle/012] {
		\pgfkeysvalueof{/square/label/012}} (235:0.6);
}
\def\sq@bcd{%
	\draw [/square/arrowstyle/123] (-54:0.25) --
	node [/square/labelstyle/123] {
		\pgfkeysvalueof{/square/label/123}} (-54:0.6);
}
\def\sq@acd{%
	\draw [/square/arrowstyle/023] (55:0.55) --
	node [/square/labelstyle/023] {
		\pgfkeysvalueof{/square/label/023}} (15:0.45);
}
\def\sq@abd{%
	\draw [/square/arrowstyle/013] (125:0.55) --
	node [/square/labelstyle/013] {
		\pgfkeysvalueof{/square/label/013}} (165:0.45);
}

\def\sq@#1#2{
	\begin{scope}[shift=#2, commutative diagrams/every diagram]
		
		\foreach \i in {0,1,2,3} {
			\tikzmath{\a = 135 + (90 * \i);}
			\node (n{#1}\i) at (\a:1) {
				\pgfkeysvalueof{/square/label/\i}};
		}
		
		\node (s#1) at (0,0) [circle, inner sep = 0pt,
		fit = (n{#1}0.center)(n{#1}1.center)(n{#1}2.center)
		(n{#1}3.center)] {};
		
		\begin{scope}[commutative diagrams/.cd, every arrow, every label]
			\ifcase #1
			\def\list{0/1, 1/2, 2/3, 0/2, 0/3}\or
			\def\list{0/1, 1/2, 2/3, 1/3, 0/3}\else
			\def\list{}\fi
			
			\foreach \s / \e in \list {
				\draw [/square/arrowstyle/\s\e] (n{#1}\s) --
				node [/square/labelstyle/\s\e] {
					\pgfkeysvalueof{/square/label/\s\e}} (n{#1}\e);
			}
			
			\ifcase #1
			\sq@abc\sq@acd\or
			\sq@abd\sq@bcd
			\else\fi
			
		\end{scope}
	\end{scope}
}

\def\square#1{
	\pgfkeys{#1}
	\sq@{0}{(180:2)}\sq@{1}{(0:2)}
	
	\begin{scope}[commutative diagrams/.cd, every arrow, every label]
		\draw[->] [shorten >=10pt, shorten <=10pt, /square/arrowstyle/0123] (s0) --
		node [/square/labelstyle/0123] {%
			\pgfkeysvalueof{/square/label/0123}} (s1);
		
	\end{scope}
}

\makeatother

\makeatletter

\def\labelstylecode@pent#1{%
	\pgfkeys@split@path%
	\edef\label@key{/pentagon/label/\pgfkeyscurrentname}%
	\edef\style@key{\pgfkeyscurrentkey/.@val}%
	\def\temp@a{#1}%
	\def\temp@b{\pgfkeysnovalue}%
	\ifx\temp@a\temp@b
	\pgfkeysgetvalue{\label@key}\temp@a
	\ifx\temp@a\temp@b\else
	\pgfkeysalso{commutative diagrams/.cd, \style@key}%
	\fi
	\else
	\pgfkeys{\style@key/.code = \pgfkeysalso{#1}}%
	\fi}
\def\arrowstylecode@pent#1{%
	\edef\style@key{\pgfkeyscurrentkey/.@val}%
	\def\temp@a{#1}%
	\def\temp@b{\pgfkeysnovalue}%
	\ifx\temp@a\temp@b
	\pgfkeysalso{commutative diagrams/.cd, \style@key}%
	\else
	\pgfkeys{\style@key/.code = \pgfkeysalso{#1}}%
	\fi}

\pgfkeys{
	/pentagon/label/.cd,
	0/.initial = {$\bullet$}, 1/.initial = {$\bullet$}, 2/.initial = {$\bullet$},
	3/.initial = {$\bullet$}, 4/.initial = {$\bullet$},
	01/.initial, 12/.initial, 23/.initial, 34/.initial, 04/.initial,
	02/.initial, 03/.initial, 13/.initial, 14/.initial, 24/.initial,
	012/.initial, 013/.initial, 014/.initial, 023/.initial, 024/.initial,
	034/.initial, 123/.initial, 124/.initial, 134/.initial, 234/.initial,
	0123/.initial, 0124/.initial, 0134/.initial, 0234/.initial, 1234/.initial,
	01234/.initial,
	/pentagon/labelstyle/.cd,
	01/.@val/.initial, 12/.@val/.initial, 23/.@val/.initial, 34/.@val/.initial,
	04/.@val/.initial=\pgfkeysalso{swap},
	02/.@val/.initial=\pgfkeysalso{description},
	03/.@val/.initial=\pgfkeysalso{description},
	13/.@val/.initial=\pgfkeysalso{description},
	14/.@val/.initial=\pgfkeysalso{description},
	24/.@val/.initial=\pgfkeysalso{description},
	012/.@val/.initial=\pgfkeysalso{swap, above},
	013/.@val/.initial=\pgfkeysalso{below = 1pt},
	014/.@val/.initial=\pgfkeysalso{below},
	023/.@val/.initial=\pgfkeysalso{swap, above = 1pt},
	024/.@val/.initial=\pgfkeysalso{below left = -1pt and -1pt},
	034/.@val/.initial=\pgfkeysalso{swap, right},
	123/.@val/.initial=\pgfkeysalso{below left = -1pt and -1pt},
	124/.@val/.initial=\pgfkeysalso{swap, right},
	134/.@val/.initial=\pgfkeysalso{left},
	234/.@val/.initial=\pgfkeysalso{swap, below right = -1pt and -1pt},
	0123/.@val/.initial,
	0124/.@val/.initial=\pgfkeysalso{swap},
	0134/.@val/.initial,
	0234/.@val/.initial=\pgfkeysalso{swap},
	1234/.@val/.initial,
	01234/.@val/.initial,
	01/.code=\labelstylecode@pent{#1}, 02/.code=\labelstylecode@pent{#1},
	03/.code=\labelstylecode@pent{#1}, 04/.code=\labelstylecode@pent{#1},
	12/.code=\labelstylecode@pent{#1}, 13/.code=\labelstylecode@pent{#1},
	14/.code=\labelstylecode@pent{#1}, 23/.code=\labelstylecode@pent{#1},
	24/.code=\labelstylecode@pent{#1}, 34/.code=\labelstylecode@pent{#1},
	012/.code=\labelstylecode@pent{#1}, 013/.code=\labelstylecode@pent{#1},
	014/.code=\labelstylecode@pent{#1}, 023/.code=\labelstylecode@pent{#1},
	024/.code=\labelstylecode@pent{#1}, 034/.code=\labelstylecode@pent{#1},
	123/.code=\labelstylecode@pent{#1}, 124/.code=\labelstylecode@pent{#1},
	134/.code=\labelstylecode@pent{#1}, 234/.code=\labelstylecode@pent{#1},
	0123/.code=\labelstylecode@pent{#1}, 0124/.code=\labelstylecode@pent{#1},
	0134/.code=\labelstylecode@pent{#1}, 0234/.code=\labelstylecode@pent{#1},
	1234/.code=\labelstylecode@pent{#1}, 01234/.code=\labelstylecode@pent{#1},
	/pentagon/arrowstyle/.cd,
	01/.@val/.initial, 12/.@val/.initial, 23/.@val/.initial, 34/.@val/.initial,
	04/.@val/.initial, 02/.@val/.initial, 03/.@val/.initial, 13/.@val/.initial,
	14/.@val/.initial, 24/.@val/.initial,
	012/.@val/.initial=\pgfkeysalso{Rightarrow},
	013/.@val/.initial=\pgfkeysalso{Rightarrow},
	014/.@val/.initial=\pgfkeysalso{Rightarrow},
	023/.@val/.initial=\pgfkeysalso{Rightarrow},
	024/.@val/.initial=\pgfkeysalso{Rightarrow},
	034/.@val/.initial=\pgfkeysalso{Rightarrow},
	123/.@val/.initial=\pgfkeysalso{Rightarrow},
	124/.@val/.initial=\pgfkeysalso{Rightarrow},
	134/.@val/.initial=\pgfkeysalso{Rightarrow},
	234/.@val/.initial=\pgfkeysalso{Rightarrow},
	0123/.@val/.initial=\pgfkeysalso{triple},
	0124/.@val/.initial=\pgfkeysalso{triple},
	0134/.@val/.initial=\pgfkeysalso{triple},
	0234/.@val/.initial=\pgfkeysalso{triple},
	1234/.@val/.initial=\pgfkeysalso{triple},
	01234/.@val/.initial=\pgfkeysalso{quadruple},
	01/.code=\arrowstylecode@pent{#1}, 02/.code=\arrowstylecode@pent{#1},
	03/.code=\arrowstylecode@pent{#1}, 04/.code=\arrowstylecode@pent{#1},
	12/.code=\arrowstylecode@pent{#1}, 13/.code=\arrowstylecode@pent{#1},
	14/.code=\arrowstylecode@pent{#1}, 23/.code=\arrowstylecode@pent{#1},
	24/.code=\arrowstylecode@pent{#1}, 34/.code=\arrowstylecode@pent{#1},
	012/.code=\arrowstylecode@pent{#1}, 013/.code=\arrowstylecode@pent{#1},
	014/.code=\arrowstylecode@pent{#1}, 023/.code=\arrowstylecode@pent{#1},
	024/.code=\arrowstylecode@pent{#1}, 034/.code=\arrowstylecode@pent{#1},
	123/.code=\arrowstylecode@pent{#1}, 124/.code=\arrowstylecode@pent{#1},
	134/.code=\arrowstylecode@pent{#1}, 234/.code=\arrowstylecode@pent{#1},
	0123/.code=\arrowstylecode@pent{#1}, 0124/.code=\arrowstylecode@pent{#1},
	0134/.code=\arrowstylecode@pent{#1}, 0234/.code=\arrowstylecode@pent{#1},
	1234/.code=\arrowstylecode@pent{#1}, 01234/.code=\arrowstylecode@pent{#1}
}

\def\pent@abc{%
	\draw [/pentagon/arrowstyle/012] (198:0.45) --
	node [/pentagon/labelstyle/012] {
		\pgfkeysvalueof{/pentagon/label/012}} (198:0.8);
}
\def\pent@bcd{%
	\draw [/pentagon/arrowstyle/123] (126:0.45) --
	node [/pentagon/labelstyle/123] {
		\pgfkeysvalueof{/pentagon/label/123}} (126:0.8);
}
\def\pent@cde{%
	\draw [/pentagon/arrowstyle/234] (54:0.45) --
	node [/pentagon/labelstyle/234] {
		\pgfkeysvalueof{/pentagon/label/234}} (54:0.8);
}
\def\pent@ade{%
	\draw [/pentagon/arrowstyle/034] (-40:0.6) --
	node [/pentagon/labelstyle/034] {
		\pgfkeysvalueof{/pentagon/label/034}} (-5:0.5);
}
\def\pent@abe{%014
	\draw [/pentagon/arrowstyle/014] (-70:0.55) --
	node [/pentagon/labelstyle/014] {
		\pgfkeysvalueof{/pentagon/label/014}} (-110:0.55);
}
\def\pent@acd{%
	\draw [/pentagon/arrowstyle/023] (55:0.3) --
	node [/pentagon/labelstyle/023] {
		\pgfkeysvalueof{/pentagon/label/023}} (125:0.3);
}
\def\pent@bde{%
	\draw [/pentagon/arrowstyle/134] (-5:0.4) --
	node [/pentagon/labelstyle/134] {
		\pgfkeysvalueof{/pentagon/label/134}} (35:0.5);
}
\def\pent@ace{%
	\draw [/pentagon/arrowstyle/024] (-45:0.45) --
	node [/pentagon/labelstyle/024] {
		\pgfkeysvalueof{/pentagon/label/024}} (-45:0.1);
}
\def\pent@abd{%
	\draw [/pentagon/arrowstyle/013] (-90:0.22) --
	node [/pentagon/labelstyle/013] {
		\pgfkeysvalueof{/pentagon/label/013}} (-150:0.46);
}
\def\pent@bce{%
	\draw [/pentagon/arrowstyle/124] (188:0.4) --
	node [/pentagon/labelstyle/124] {
		\pgfkeysvalueof{/pentagon/label/124}} (150:0.55);
}

\def\pent@#1#2{
	\begin{scope}[shift=#2, commutative diagrams/every diagram]
		
		\foreach \i in {0,1,2,3,4} {
			\tikzmath{\a = 270 - (72 * \i);}
			\node (n{#1}\i) at (\a:1) {
				\pgfkeysvalueof{/pentagon/label/\i}};
		}
		
		\node (p#1) at (0,0) [circle, inner sep = 0pt,
		fit = (n{#1}0.center)(n{#1}1.center)(n{#1}2.center)
		(n{#1}3.center)(n{#1}4.center)] {};
		
		\begin{scope}[commutative diagrams/.cd, every arrow, every label]
			\ifcase #1
			\def\list{0/1, 1/2, 2/3, 3/4, 0/4, 0/2, 0/3}\or
			\def\list{0/1, 1/2, 2/3, 3/4, 0/4, 1/3, 1/4}\or
			\def\list{0/1, 1/2, 2/3, 3/4, 0/4, 0/2, 2/4}\or
			\def\list{0/1, 1/2, 2/3, 3/4, 0/4, 0/3, 1/3}\or
			\def\list{0/1, 1/2, 2/3, 3/4, 0/4, 1/4, 2/4}\else
			\def\list{}\fi
			
			\foreach \s / \e in \list {
				\draw [/pentagon/arrowstyle/\s\e] (n{#1}\s) --
				node [/pentagon/labelstyle/\s\e] {
					\pgfkeysvalueof{/pentagon/label/\s\e}} (n{#1}\e);
			}
			
			\ifcase #1
			\pent@abc\pent@acd\pent@ade\or
			\pent@bcd\pent@bde\pent@abe\or
			\pent@cde\pent@ace\pent@abc\or
			\pent@ade\pent@abd\pent@bcd\or
			\pent@abe\pent@bce\pent@cde
			\else\fi
			
		\end{scope}
	\end{scope}
}

\makeatother
\renewcommand{\plus}[1]{\mathop{\amalg}\limits_{#1}}
\newcommand{\st}{\underline{\mathrm{Strat}}}
\newcommand{\B}{\mathcal{B}}
\newcommand{\C}{\mathcal{C}}
\newcommand{\D}{\mathcal{D}}
\newcommand{\E}{\mathcal{E}}
\newcommand{\F}{\mathcal{F}}
\newcommand{\G}{\mathcal{G}}

\newcommand{\I}{\mathcal{I}}

\newcommand{\J}{\mathcal{J}}
\newcommand{\M}{\mathcal{M}}

\newcommand{\X}{\mathcal{X}}
\newcommand{\Y}{\mathcal{Y}}
\newcommand{\U}{\mathcal{U}}
\newcommand{\Tw}{\mathrm{Tw}}

\newcommand{\Car}{\mathrm{Car}}
\newcommand{\coCar}{\mathrm{coCar}}

\newcommand{\fCs}{\mathfrak{C}^{\mathrm{sc}}}
\newcommand{\Maptr}{\Hom^{\triangleright}}
\newcommand{\Cone}{\mathrm{Cone}}

\def\vphi{\varphi}
\def\lrar{\rightarrow}
\def\llar{\longleftarrow}

\usepackage[font=small,format=hang,labelfont={sf,bf}]{caption}

\newcommand{\minisimeq}{\scalebox{0.8}{\ensuremath\simeq}}

\newcommand{\car}{\mathrm{car}}
\newcommand{\id}{\Id}
\newcommand{\var}{\mathrm{var}}

\newcommand{\defn}{\emph}
\newcommand{\uni}{\mathrm{uni}}

\setlist[itemize]{leftmargin=*}
\setlist[enumerate]{leftmargin=*}

\makeatletter

\renewcommand{\tocsection}[3]{%
\indentlabel{\@ifnotempty{#2}{\parbox[b]{3ex}{\bfseries\ignorespaces#1 #2}}}\bfseries#3} 

\renewcommand{\tocsubsection}[3]{%
\indentlabel{\@ifnotempty{#2}{\hspace{1.6em}\parbox[b]{5ex}{\ignorespaces#1 #2}}}#3}

\makeatother

\title{Fibrations and lax limits of \((\infty,2)\)-categories}
\author{Andrea Gagna}
\address{Institute of Mathematics, Czech Academy of Sciences\\ \v{Z}itn\'a 25 \\115 67   Praha 1\\ Czech Republic}
\email{gagna@math.cas.cz}
\urladdr{https://sites.google.com/view/andreagagna/home}
\author{Yonatan Harpaz}
\address{Institut Galilée\\ Université Paris 13\\ 99 avenue Jean-Baptiste Clément\\ 93430 Villeta-neuse\\ France}
\email{harpaz@math.univ-paris13.fr}
\urladdr{https://www.math.univ-paris13.fr/~harpaz}
\author{Edoardo Lanari}
\address{Institute of Mathematics, Czech Academy of Sciences\\ \v{Z}itn\'a 25 \\115 67   Praha 1\\ Czech Republic}
\email{edoardo.lanari.el@gmail.com}
\urladdr{https://edolana.github.io/}
\subjclass[2020]{18N65, 55U35, 18N50}
\begin{document}
\begin{abstract}
	We study four types of (co)cartesian fibrations of \(\infty\)-bicategories over a given base \(\B\), and prove that they encode the four variance flavors of \(\B\)-indexed diagrams of \(\infty\)-categories. We then use this machinery to set up a general theory of \(2\)-(co)limits for diagrams valued in an \(\infty\)-bicategory, capable of expressing lax, weighted and pseudo limits. When the \(\infty\)-bicategory at hand arises from a model category tensored over marked simplicial sets, we show that this notion of \(2\)-(co)limit can be calculated as a suitable form of a weighted homotopy limit on the model categorical level, thus showing in particular the existence of these 2-(co)limits in a wide range of examples. We finish by discussing a notion of cofinality appropriate to this setting and use it to deduce the unicity of 2-(co)limits, once exist.
 \end{abstract}
\maketitle
\tableofcontents

\section*{Introduction}

A fundamental idea in category theory is that diagrams of (\(\infty\)-)categories, indexed by an (\(\infty\)-)category \(\B\), can be encoded via a suitable form of fibration \(\E \to \B\). 
This idea, going back to Grothendieck for ordinary categories, has become indispensable in the \(\infty\)-categorical realm, where it was developed notably in the extensive works of Lurie. In effect, a fibration \(\E \to \B\) is often the most efficient option, and sometimes the only practical one, for writing down a diagram of \(\infty\)-categories with all coherence data involved.

In the \(\infty\)-categorical context, such fibrations come in two flavors, called \emph{cartesian} and \emph{cocartesian} fibrations. The latter encodes the data of a \(\Catoo\)-valued \emph{functor} on \(\B\), while the former the data of a \(\Catoo\)-valued \emph{presheaf} on \(\B\), that is, a functor \(\B^{\op} \to \Catoo\). The existence of two dual flavors of this type is prevalent in category theory, and reflects the \(\ZZ/2\)-symmetry of \(\Catoo\) given by the involution \(\C \mapsto \C^{\op}\). In particular, this symmetry sends cartesian fibrations to cocartesian ones, and vice versa.

As in ordinary category theory, the study of \(\infty\)-categories often leads to consider \emph{\((\infty,2)\)-categories}, as, for example, the collection of \(\infty\)-categories is itself best understood when organized into one.
There are currently many models for \((\infty,2)\)-categories, 
developed and compared to one another in the works of Lurie~\cite{LurieGoodwillie}, Verity~\cite{VerityWeakComplicialI}, Rezk--Bergner~\cite{BergnerRezkInftynI, BergnerRezkInftynII}, Ara~\cite{AraQCatvsRezk}, Barwick--Schommer-Pries~\cite{BarwickSchommerPriesUnicity}, and more recently in the authors previous work~\cite{GagnaHarpazLanariEquiv}, where Lurie's bicategorical model structure on scaled simplicial sets was compared with the \(2\)-trivial complicial model structure on stratified sets developed in~\cite{OzornovaRovelliNComplicial}, yielding the last remaining equivalence between the various models. In addition, in~\cite{GagnaHarpazLanariEquiv} we reinterpreted the bicategorical model structure in terms of Cisinski--Olschok's theory of localizers, a consequence of which is an identification of the notion of \(\infty\)-bicategories - the fibrant object in this model structure - with scaled simplicial sets satisfying a suitable extension property.

It is natural to ponder the counterpart of the theory of (co)cartesian fibrations in the case where the base \(\B\) is now an \(\infty\)-bicategory. On the diagram side, one may then consider \(\B\)-indexed diagrams in \(\Catoo\), where \(\Catoo\) is also considered as an \(\infty\)-bicategory. As an immediate difference from the \(\infty\)-categorical case, one observes that the theory of \(\infty\)-bicategories admits not just a \(\ZZ/2\)-symmetry, but a \((\ZZ/2)^2\)-symmetry: we have the involution \(\C \mapsto \C^{\op}\) which reverses the direction of all 1-morphisms (without affecting the direction of 2-morphisms), but we also have the involution \(\C \mapsto \C^{\co}\), which reverses the direction of 2-morphisms, without affecting the direction of 1-morphisms. As a result, one has four variance flavors for \(\B\)-indexed diagrams in \(\Catoo\), corresponding respectively to \(\Catoo\)-valued functors from \(\B,\B^{\op},\B^{\co}\) and \(\B^{\coop} = (\B^{\co})^{\op}\). As a result, we expect to have \emph{four} different types of fibrations \(\E \to \B\) this time, encoding \(\B\)-indexed \(\Catoo\)-valued diagrams corresponding to the above four variance flavors.

In the first part of this paper we identify these four types of fibrations as \emph{inner cocartesian}, \emph{outer cartesian}, \emph{outer cocartesian} and \emph{inner cartesian} fibrations, respectively. The first of these four notions is based on that studied in~\cite{LurieGoodwillie}, for which a straightening-unstraightening Quillen equivalence is constructed. The last one can be obtained from the first by applying the functor \((-)^{\op}\) on the level of scaled simplicial sets (where we note that this acts on the fibers by \((-)^{\op}\) as well). Both of these are in particular inner fibrations on the level of simplicial sets, hence their name. On the other hand, since the scaled simplicial set model is not equipped with a convenient point-set model for \((-)^{\co}\), the two flavors designated by the term \emph{outer} require a more substantial modification of the definition on the simplicial level. A working definition was introduced in~\cite{GagnaHarpazLanariEquiv}, and in the present paper we prove that these indeed fulfill the purpose of encoding functors of the form \(\B^{\co} \to \Catoo\) and \(\B^{\op} \to \Catoo\), respectively. In particular, the main result of the first part affirms that if we organize the collection of inner/outer (co)cartesian fibrations over a fixed base \(\B\), with functors over \(\B\) which preserves (co)cartesian edges as morphisms, then the result is equivalent to the \(\infty\)-bicategory of \(\B\)-indexed \(\Catoo\)-valued diagrams, with the appropriate variance flavor:

\begin{thm}[{See Corollary~\ref{c:straightening-inner} and Corollary~\ref{c:S-U-for-outer-fibs}}]
\label{t:S-U-intro}
For an \(\infty\)-bicategory \(\B \in \BiCat\) there are natural equivalences of \(\infty\)-bicategories
\begin{align*}
 \coCar^{\inn}(\B) 	& \simeq \Fun(\B,\Catoo)
  &,&&
  \Car^{\inn}(\B) 	&\simeq \Fun(\B^{\coop},\Catoo), \\
 \coCar^{\out}(\B) 	&\simeq \Fun(\B^{\co},\Catoo)
  &\text{and}&&
 \Car^{\out}(\B) 	&\simeq \Fun(\B^{\op},\Catoo).
\end{align*}
Here, \(\coCar^{\inn}(\B)\) denotes the \(\infty\)-bicategory of inner cocartesian fibrations over \(\B\) and cocartesian edges preserving functors over \(\B\) between them, and similarly for \(\Car^{\inn}(\B), \coCar^{\out}(\B)\) and \(\Car^{\out}(\B)\).
\end{thm}

In the second part of this paper we use the theory of inner/outer (co)cartesian fibrations in order to setup a well-behaved notion of \(2\)-(co)limits for diagrams taking values in an \(\infty\)-bicategory. Our notion is sufficiently flexible to accommodate both (op)lax and pseudo-(co)limits, as well as a variety of intermediate variants, and can also be used to give a notion of weighted (co)limits. As with ordinary 2-categories, this notion comes in principle in four flavors, spanning lax and oplax, limits and colimits. To enable a systematic treatment we exploit in a crucial manner the identification of the four types of fibrations from the first part of the paper. To keep the notation tractable, and since the notation of lax and oplax is not completely consistent in the literature, we have opted here to call these \emph{inner} and \emph{outer (co)limits}, making each type of (co)limit directly related to the type of fibration that governs it. After giving the definitions and extracting their main properties we proceed to show that our proposed notion of 2-(co)limit is sufficiently flexible to express a suitable notion of a weighted 2-(co)limit, and that furthermore, every type of 2-(co)limit can eventually be viewed as a weighted one for a suitable weight. We then exploit this point of view in order to compare our notion of 2-(co)limits with that of a \emph{weighted homotopy (co)limits} in the case where the ambient \(\infty\)-bicategory comes from a model category tensored over marked simplicial sets. This allows us to exhibit a wide range of examples for \(\infty\)-bicategories in which all small 2-(co)limits exist:

\begin{thm}[{See Corollary~\ref{c:complete}}]\label{t:complete-intro}
Let \(\M\) be an \(\s^+\)-tensored model category such that the projective (resp.~injective) model structure exists on \(\M^{\J}\) for any small \(\s^+\)-enriched category \(\J\) (\eg, \(\M\) is a combinatorial model category). Then
the \(\infty\)-bicategory \(\M_\infty\) admits inner and outer limits (resp.~colimits) indexed by arbitrary small \(\infty\)-bicategories, and these are computed by taking weighted homotopy limits (resp.~colimits) in \(\M\) with respect to a suitable weight.
\end{thm}

Finally, we also obtain an explicit description of inner and outer limits for diagrams valued in \(\Catoo\). In particular, if \(\I\) is an \(\infty\)-bicategory and \(\chi\colon \I \to \Catoo\) is a diagram, then by Theorem~\ref{t:S-U-intro} we may encode \(\chi\) by an inner cocartesian fibration \(\E^{\inn} \to \I\). At the same time, post-composing \(\chi^{\co}\colon \I^{\co} \to \Catoo^{\co}\) with the functor \((-)^{\op}\colon \Catoo^{\co} \to \Catoo\) we obtain a diagram \(\I^{\co}\to \Catoo\), which can then be encoded by an outer cocartesian fibration \(\E^{\out} \to \I\). We then have the following explicit description of the inner and outer (or lax and oplax) limits of \(\chi\), generalizing in particular~\cite[Proposition 7.1]{GepnerHaugsengNikolausLax} and the limits part of~\cite[Theorem 4.4]{berman-lax}:

\begin{thm}[{See Example~\ref{ex:2-limits-cat}}]\label{t:intro-3}
There are natural equivalences
\[ \mathrm{lim}^{\inn}_\I \chi \simeq \Fun_{\I}(\I,\E^{\inn}) \quad\text{and}\quad \mathrm{lim}^{\out}_\I \chi \simeq \Fun_{\I}(\I,\E^{\out})^{\op}\]
identifying the inner limit of \(\chi\) with the \(\infty\)-category of sections of \(\E^{\inn} \to \I\), and the outer limit of \(\chi\) with the opposite of the \(\infty\)-category of sections of \(\E^{\out} \to \I\).
\end{thm}

This paper is organized as follows. In \S\ref{sec:preliminaries} we describe preliminary material, including marked and scaled simplicial sets, and Lurie's scaled straightening-unstraight\-en\-ing equivalence from~\cite{LurieGoodwillie}. In \S\ref{s:fibrations} we develop in some detail the theory of inner and outer (co)cartesian fibrations. In particular, in \S\ref{s:cartesian} we show that inner (resp.~outer) fibrations induce right (resp.~left) fibrations on the level of mapping \(\infty\)-categories, and in \S\ref{sec:car-edges} we give several equivalent characterizations of (co)cartesian edges. In \S\ref{s:slice} we verify that inner cocartesian fibrations
can be identified with the fibrant objects of Lurie's \(\Beta\)-fibered model structure, which is the one featuring in the straightening and unstaightening equivalence, and develop a similar recognition mechanism for outer fibrations in terms of a suitable extension property against anodyne maps. Finally, in \S\ref{s:lift} we show that inner/outer (co)cartesian fibrations admit a convenient lifting property for (op)lax natural transformations.

In \S\ref{sec:correspondence} we prove the main result of the first part of this paper, as described in Theorem~\ref{t:S-U-intro}. The case of inner cocartesian fibrations being essentially a consequence of Lurie's scaled straightening-unstaightening, our strategy involves reducing all four equivalences to that one by showing that the \((\ZZ/2)^2\)-symmetry of the theory of \((\infty,2)\)-categories switches between all four types of fibrations. Unfortunately, the model of scaled simplicial sets does admit a convenient model for \((-)^{\co}\). To circumvent this problem we first show that the notions of inner/outer (co)cartesian fibrations can be defined also in the setting of \emph{marked simplicial categories}, that is, categories enriched in marked simplicial sets. Establishing the equivalence between the scaled and enriched definitions is the main goal of \S\ref{s:enriched}. We then exploit the fact that the model of marked simplicial categories admits point-set models for both \((-)^{\op}\) and \((-)^{\co}\) to reduce the main theorem to the inner cocartesian case.

We dedicate \S\ref{sec:thick-slice} to the study of the \emph{thickened} slice construction, which enables one to construct slice fibrations of all four variance flavors. These play a key role in the definition of \(2\)-(co)limits, and so we take the time to establish all the properties we will need later on. The construction makes use of a variant of the Gray tensor product in the setting of simplicial sets with both marking and scaling, which we define in \S\ref{s:gray} following a similar construction by Verity~\cite{VerityWeakComplicialI} in the setting stratified sets, as well as a construction studied by the authors in~\cite{GagnaHarpazLanariGrayLaxFunctors} for scaled simplicial sets. The simplest type of slice fibrations are obtained by slicing over an object \(x \) in an \(\infty\)-bicategory \(\B\). In \S\ref{s:representable} we identify these slice fibrations as those classified by representable functors via the equivalence of Theorem~\ref{t:S-U-intro}.

In \S\ref{sec:limits} we introduce the notions of 2-(co)limits for diagrams \(f\colon \I \to \C\), together with the auxiliary data consisting of a collection of edges in \(\I\), that is, a \emph{marking}. The marking data roughly encodes along which edges in \(\I\) the limit is to be ``strong'', as apposed to lax. We then prove a characterization of 2-(co)limits in terms of the functors they (co)represent. In \S\ref{s:weighted} we consider a particular case of 2-(co)limits which, following the works of Rovelli~\cite{RovelliWeighted}, Haugseng~\cite{HaugsengCoends} and Gepner--Hauseng--Nikolaus~\cite{GepnerHaugsengNikolausLax}, we call \emph{weighted} (co)limits. We then show that this particular case is in some sense generic: every 2-(co)limit can equivalently be expressed as a weighted (co)limit with respect to a suitable weight. Finally, in \S\ref{s:model-categories} we prove that when \(\C\) arises from a model category \(\M\) tensored over marked simplicial sets, weighted 2-(co)limits, and consequently all 2-(co)limits, can be computed in terms of weighted homotopy (co)limits in \(\M\). We then deduce the main result of the second part, showing that small 2-(co)limits exist for a wide range of ambient \(\infty\)-bicategories, including for example all those associated to combinatorial model categories tensored over marked simplicial sets. Finally, in \S\ref{s:cofinal} we discuss a notion of cofinality appropriate for the \(\infty\)-bicategorical setting and relate it to 2-(co)limits in two different manners, see Theorem~\ref{t:cofinal-characterization} and Proposition~\ref{p:limit-is-cofinal}. We then use this to deduce the unicity of 2-(co)limits, see Corollary~\ref{c:unicity}.

\subsection*{Relation to other work}
Lax and weighted limits in an \((\infty,2)\)-categorical setting were first defined by Gepner--Hauseng--Nikolaus~\cite{GepnerHaugsengNikolausLax} for diagrams indexed by \(\infty\)-categories and taking values in presentable \(\infty\)-categories tensored over \(\Catoo\). The idea of using marking on the indexing \(\infty\)-category to express partial laxness has recently appeared in several independent works, see J. Berman~\cite{berman-lax} for diagrams taking values in \(\infty\)-categories tensored or cotensored over \(\Catoo\), and F. García~\cite{garcia-cofinality} for diagrams taking values in \(\infty\)-bicategories. In all these works the diagram in question is indexed by an \(\infty\)-category, and the definition is eventually made via a weighted limits for a suitable weight. The definition given in the present paper is of a somewhat wider scope, covering diagrams indexed by an arbitrary (marked) \(\infty\)-bicategory and taking values in an arbitrary \(\infty\)-bicategory. In contrast to the previous works, our definition proceeds via a universal cone approach, and not via weighted limits. We do however show that 2-(co)limits in our sense can be computed via suitable weighted (co)limits, and are consequently compatible with all exiting definitions (whenever those apply), see Remark~\ref{r:compare-weighted} and Remark~\ref{r:compare-limit}.

\subsection*{Acknowledgements}
The first and third authors gratefully acknowledge the support of Praemium Academiae of M.~Markl and RVO:67985840.

\numberwithin{thm}{subsection}
\section{Preliminaries}\label{sec:preliminaries}
In this section we establish notation and recall some preliminary definitions and results concerning marked and scaled simplicial sets, and the straightening-unstraight\-en\-ing Quillen equivalence of~\cite{LurieGoodwillie}.

\begin{notate}
We will denote by \(\Delta\) the category of simplices, that is,
the category whose objects are the finite non-empty ordinals \([n] = \{0, 1, 2, \dots, n\}\)
and morphisms are the non-decreasing maps.
We will denote by \(\s\) the category of simplicial sets, that is the category of
presheaves on sets of~\(\Delta\), and will employ the standard notation
\(\Del^n\) for the \(n\)-simplex, \ie the simplicial set representing
the object~\([n]\) of~\(\Delta\).
For any subset \(\emptyset \neq S \subseteq [n]\) we will write \(\Del^S \subseteq \Del^n\)
to denote the \((|S|-1)\)-dimensional face of \(\Del^n\) whose set of vertices is \(S\).
For \(0 \leq i \leq n\) we will denote by \(\Lam^n_i \subseteq \Del^n\)
the \(i\)-th horn in \(\Del^n\), that is, the subsimplicial set of \(\Del^n\) spanned
by all the \((n-1)\)-dimensional faces containing the \(i\)-th vertex. 
For any simplicial set \(X\) and any integer \(p\geq 0\), we will denote by \(\deg_p(X)\)
the set of degenerate~\(p\)-simplices of~\(X\).
\end{notate}

By an \ndef{\(\infty\)\nbd-category} we will always mean a \emph{quasi-category},
\ie a simplicial set \(\C\) which admits extensions for all inclusions 
\(\Lambda^n_i\rightarrow\Delta^n\) 
with \(0 < i < n\)
(also known as \ndef{inner horn inclusions}). Given an \(\infty\)-category \(\C\),
we will denote its homotopy category by \(\ho(\C)\).
This is the ordinary category having as objects the \(0\)-simplices of \(\C\),
and as morphisms \(x \rightarrow y\) the set of equivalence classes of \(1\)-simplices
\(f\colon x \rightarrow y\) of \(\C\) under the equivalence relation generated by identifying
\(f\) and \(f'\) if there is a \(2\)-simplex \(H\) of \(\C\) with
\( H_{|\Del^{\{1,2\}}}=f, \ H_{|\Del^{\{0,2\}}}=f'\) and \(H_{|\Del^{\{0,1\}}}\) degenerate on~\(x\).

\subsection{Marked simplicial sets}

\begin{define}
	A \emph{marked simplicial set} is a pair \((X,E_X)\) where \(X\) is
	simplicial set and \(E_X\) is a subset of the set of \(1\)-simplices of \(X\),
	called \emph{marked \(1\)-simplices} or 
	\ndef{marked edges}, containing the degenerate ones.
	A map of marked simplicial sets \(f\colon (X,E_X)\rightarrow (Y,E_Y)\) is a map of simplicial sets \(f\colon X \rightarrow Y\) satisfying \(f(E_X)\subseteq E_Y\).
\end{define}

	The category of marked simplicial sets will be denoted by \(\s^+\). It is locally presentable and cartesian closed.

\begin{notate}
	Let \(X\) be a simplicial set. We will denote by \(X^{\flat} = (X, \deg_1(X))\)
	the marked simplicial set whose marked edges are the degenerate \(1\)-simplices 
	and by \(X^{\sharp} = (X, X_1)\) the marked simplicial
	set where all the edges of \(X\) are marked.
	The assignments
	\[X \mapsto X^\flat\qquad\text{and}\qquad X \mapsto X^\sharp\]
	are left and right adjoint, respectively, to the forgetful functor
	\(\s^+ \to \s\).
\end{notate}

Marked simplicial sets can be used as a model for the theory of \((\infty,1)\)-categories:

\begin{thm}[\cite{HTT}]\label{thm:marked-categorical}
	There exists a model category structure on the category \(\s^+\)
	of marked simplicial sets in which the cofibrations are the monomorphisms
	and the fibrant objects are the marked simplicial sets \((X, E)\) in which \(X\)
	is an \(\infty\)-category and \(E\) is the set of equivalences of \(X\),
	\ie \(1\)-simplices \(f\colon \Delta^1 \rightarrow X\) which are invertible in \(\ho(X)\). 
\end{thm}

The theorem above is a special case of Proposition 3.1.3.7 in \cite{HTT}, when \(S=\Delta^0\). By~\cite[Proposition~3.1.5.3]{HTT} the forgetful functor \(\s^+ \to \s\) is a right Quillen equivalence, where \(\s\) is endowed with the categorical model structure of Joyal-Lurie. We will refer to the model structure of Theorem~\ref{thm:marked-categorical}
as the \ndef{marked categorical model structure},
and its weak equivalences as \ndef{marked categorical equivalences}.

We will denote by \(\nCat{\s^+}\) the category of categories in enriched in \(\s^+\) with respect to the cartesian product on \(\s^+\). For a \(\s^+\)-enriched category \(\C\) and two objects \(x,y \in \C\) we will denote by \(\C(x,y) \in \s^+\) to associated mapping marked simplicial set. By an arrow \(e\colon x \to y\) in an \(\s^+\)-enriched category \(\C\) we will simply mean a vertex \(e \in \C(x,y)_0\).

We will generally consider \(\nCat{\s^+}\) together with its associated \emph{Dwyer-Kan model structure} (see~\cite[\S A.3.2]{HTT}). In this model structure the weak equivalences are the Dwyer-Kan equivalences, that is, the maps which are essentially surjective on homotopy categories and induce marked categorical equivalences on mapping objects. The fibrant objects are the enriched categories \(\C\) whose mapping objects \(\C(x,y)\) are all fibrant, that is, are all \(\infty\)-categories marked by their equivalences. The model category \(\nCat{\s^+}\) is then a presentation of the theory of \((\infty,2)\)-categories, and is Quillen equivalent to other known models, see \S\ref{sec:scaled} below.

\subsection{Scaled simplicial sets and \pdfoo-bicategories}\label{sec:scaled}

\begin{define}[\cite{LurieGoodwillie}]
	A \emph{scaled simplicial set} is a pair \((X,T_X)\) where \(X\) is simplicial set and \(T_X\) is a subset of the set of 2-simplices of \(X\), called \emph{thin \(2\)-simplices} or \emph{thin triangles}, containing the degenerate ones. A map of scaled simplicial sets \(f\colon (X,T_X)\rightarrow (Y,T_Y)\) is a map of simplicial sets \(f\colon X \rightarrow Y\) satisfying \(f(T_X)\subseteq T_Y\). 
\end{define}

We will denote by \(\Ss\) the category of scaled simplicial sets. It is locally presentable and cartesian closed.

\begin{notate}\label{not:scaled_flat-sharp}
	Let \(X\) be a simplicial set. We will denote by \(X_{\flat} = (X, \deg_2(X))\)
	the scaled simplicial set where the
	thin triangles of \(X\) are the degenerate \(2\)-simplices and by \(X_{\sharp} = (X, X_2)\) the scaled simplicial
	set where all the triangles of \(X\) are thin.
	The assignments
	\[X \mapsto X_\flat\qquad\text{and}\qquad X \mapsto X_\sharp\]
	are left and right adjoint, respectively, to the forgetful functor
	\(\Ss \to \s\).
\end{notate}

\begin{define}
	\label{core defi}
	Given a scaled simplicial set \(X\), we define its \emph{core} to be the simplicial set \(X^{\thi}\) spanned by those \(n\)-simplices of \(X\) whose 2-dimensional faces are thin triangles. The assignment \(X \mapsto X^{\thi}\) is then right adjoint to the functor \((-)_{\sharp}\colon \s \to \Ss\).
\end{define}

\begin{warning}
	In \cite[\href{https://kerodon.net/tag/01XA}{Tag 01XA}]{LurieKerodon}, Lurie uses the term \emph{pith} in place of core, and denotes it by \(\operatorname{Pith}(\operatorname{\mathcal{C}})\).
\end{warning}

\begin{notate}\label{not:scaled_subset}
	We will often speak only of the non-degenerate thin \(2\)-simplices
	when considering a scaled simplicial set. For example, if \(X\) is a simplicial set
	and \(T\) is any set of triangles in \(X\) then we will denote by \((X,T)\)
	the scaled simplicial set whose underlying simplicial set is \(X\) and
	whose thin triangles are \(T\) together with the degenerate triangles.
	If \(L \subseteq K\) is a subsimplicial set then we use \(T|_L := T \cap L_2\)
	to denote the set of triangles in \(L\) whose image in \(K\) is contained in \(T\). 
\end{notate}

\begin{define}
	\label{d:anodyne}
	The set of \emph{generating scaled anodyne maps} \(\bS\) is the set of maps of scaled simplicial sets consisting of:
	\begin{enumerate}
		\item[(i)]\label{item:anodyne-inner} the inner horns inclusions
		\[
		 \bigl(\Lambda^n_i,\{\Delta^{\{i-1,i,i+1\}}\}\bigr)\rightarrow \bigl(\Delta^n,\{\Delta^{\{i-1,i,i+1\}}\}\bigr)
		 \quad , \quad n \geq 2 \quad , \quad 0 < i < n ;
		\]
		\item[(ii)]\label{i:saturation} the map 
		\[
		 (\Delta^4,T)\rightarrow (\Delta^4,T\cup \{\Delta^{\{0,3,4\}}, \ \Delta^{\{0,1,4\}}\}),
		\]
		where we define
		\[
		 T\overset{\text{def}}{=}\{\Delta^{\{0,2,4\}}, \ \Delta^{\{ 1,2,3\}}, \ \Delta^{\{0,1,3\}}, \ \Delta^{\{1,3,4\}}, \ \Delta^{\{0,1,2\}}\};
		\]
		\item[(iii)]\label{item:anodyne_outer} the set of maps
		\[
		\Bigl(\Lambda^n_0\coprod_{\Delta^{\{0,1\}}}\Delta^0,\{\Delta^{\{0,1,n\}}\}\Bigr)\rightarrow \Bigl(\Delta^n\coprod_{\Delta^{\{0,1\}}}\Delta^0,\{\Delta^{\{0,1,n\}}\}\Bigr)
		\quad , \quad n\geq 3.
		\]
	\end{enumerate}
	A general map of scaled simplicial set is said to be \emph{scaled anodyne} if it belongs to the weakly saturated closure of \(\bS\).
\end{define}

\begin{define}\label{d:bicategory}
	An \emph{\(\infty\)-bicategory} is a scaled simplicial set \(\C\) which admits extensions along all maps in \(\bS\). 
\end{define}
\begin{warning}
	The notion of \(\infty\)-bicategory we have just introduced can be proven to be equivalent to that of \((\infty,2)\)-category given in \cite[\href{https://kerodon.net/tag/01W9}{Tag 01W9}]{LurieKerodon}.
\end{warning}
\begin{rem}
If \(\C\) is an \(\infty\)-bicategory then its core \(\C^{\thi}\) is an \(\infty\)-category.
\end{rem}

To avoid confusion we point out that simplicial sets as in Definition~\ref{d:bicategory} are referred to in~\cite{LurieGoodwillie} as \emph{weak \(\infty\)-bicategories}, while the term \(\infty\)-bicatgory was reserved for an a priori stronger notion. However, as we have shown in \cite{GagnaHarpazLanariEquiv}, these two notions in fact coincide. In particular:

\begin{thm}[\cite{LurieGoodwillie},\cite{GagnaHarpazLanariEquiv}]\label{t:bicategorical}
	There exists a model structure on \(\Ss\)
	whose cofibrations are the monomorphisms and whose fibrant objects are the \(\infty\)-bicategories (in the sense of Definition~\ref{d:bicategory}). \end{thm}

We will refer to the model structure of Theorem~\ref{t:bicategorical} as the \emph{bicategorical model structure}. In~\cite{LurieGoodwillie} Lurie constructs a Quillen equivalence
\[
\xymatrixcolsep{1pc}
\vcenter{\hbox{\xymatrix{
			**[l]\Ss \xtwocell[r]{}_{\Nsc}^{\fCs}{'\perp}& **[r] \nCat{\s^+}}}} ,
\]
in which the right functor \(\Nsc\) is also known as the \emph{scaled coherent nerve}. A Quillen equivalence
\begin{equation}
\label{equiv}
\xymatrixcolsep{1pc}
\vcenter{\hbox{\xymatrix{
			**[l]\Ss \xtwocell[r]{}_{U}^{\iota}{'\perp}& **[r] \st_{2}}}} ,
\end{equation}
to Verity's model structure on stratified sets for saturated \(2\)-trivial complicial sets (see~\cite{VerityWeakComplicialI}, \cite{OzornovaRovelliNComplicial}) was also established in~\cite{GagnaHarpazLanariEquiv}.
We consider the bicategorical model structure as a presentation of the theory of \((\infty,2)\)-categories. In addition to the above two comparisons, the bicategorical model structures has been compared in~\cite{LurieGoodwillie} to several other models for \((\infty,2)\)-categories, which, to our knowledge, have been compared to all other known models (see, e.g., \cite{BergnerRezkInftynI}, \cite{BergnerRezkInftynII}, \cite{AraQCatvsRezk}, \cite{BarwickSchommerPriesUnicity}; we refer the reader to~\cite[Figure 1]{GagnaHarpazLanariEquiv} of a diagrammatic depiction of all equivalences known to us).

\begin{define}
We will denote by \(\Catoo\) the scaled coherent nerve of the (fibrant) \(\s^+\)-enriched subcategory \((\s^+)^{\circ} \subseteq \s^+\) spanned by the fibrant marked simplicial sets. We will refer to \(\Catoo\) as the \emph{\(\infty\)-bicategory of \(\infty\)-categories.}
\end{define}

\begin{define}\label{d:equivalence}
	Let \(\C\) be an \(\infty\)-bicategory. We will say that an edge in \(\C\) is \ndef{invertible} if it is invertible when considered in the \(\infty\)-category \(\C^{\thi}\), that is, if its image in the homotopy category of \(\C^{\thi}\) is an isomorphism. We will sometimes refer to invertible edges in \(\C\) as \ndef{equivalences}. We will denote by \(\C^{\minisimeq} \subseteq \C^{\thi}\) the subsimplicial set spanned by the invertible edges. Then \(\C^{\minisimeq}\) is an \(\infty\)-groupoid (that is, a Kan complex), which we call the \emph{core groupoid} of \(\C\). It can be considered as the \(\infty\)-groupoid obtained from \(\C\) by discarding all non-invertible \(1\)-cells and \(2\)-cells. If \(X\) is an arbitrary scaled simplicial set then we will say that an edge in \(X\) is \ndef{invertible} if its image in \(\C\) is invertible for any bicategorical equivalence \(X \to \C\) such that \(\C\) is an \(\infty\)-bicategory. This does not depend on the choice of the \(\infty\)-bicategory replacement \(\C\).
\end{define}

\begin{notate}
	\label{n:mapping}
	Let \(\C\) be an \(\infty\)-bicategory and let \(x,y \in \C\) be two vertices.
	In~\cite[\S 4.2]{LurieGoodwillie}, Lurie gives an explicit
	model for the mapping \(\infty\)\nbd-cat\-egory  from \(x\) to \(y\) in \(\C\) that we now recall.
	Let \(\Hom_{\C}(x,y)\) be the marked simplicial set whose \(n\)-simplices are given by maps \(f\colon\Del^n \times \Del^1 \lrar \C\) such that \(f_{|\Del^n \times \{0\}}\) is constant on \(x\), \(f_{|\Del^n \times \{1\}}\) is constant on \(y\), and the triangle \(f_{|\Del^{\{(i,0),(i,1),(j,1)\}}}\) is thin 
	for every \(0 \leq i\leq j \leq n\). An edge \(f\colon\Del^1 \times \Del^1 \lrar \C\) of \(\Hom_{\C}(x,y)\) is marked exactly when the triangle \(f_{|\Del^{\{(0,0),(1,0),(1,1)\}}}\) is thin. 
	The assumption that \(\C\) is an \(\infty\)-bicategory implies that 
	the marked simplicial set \(\Hom_{\C}(x,y)\) is \ndef{fibrant} in the marked categorical model structure, that is, it is an \(\infty\)-category whose marked edges are exactly the equivalences.	
\end{notate}

\begin{rem}\label{r:dwyer-kan}
	By~\cite[Remark~4.2.1 and Theorem~4.2.2]{LurieGoodwillie}, if \(\D\) is a fibrant \(\s^+\)\nbd-en\-riched category and \(\C\) is an \(\infty\)-bicategory equipped with a bicategorical equivalence \(\vphi\colon\C \simeq \Nsc(\D)\), then the maps 
	\[ \Hom_{\C}(x,y) \longrightarrow \Hom_{\Nsc(\D)}(\vphi(x),\vphi(y)) \llar \D(\vphi(x),\vphi(y)) \]
	are marked categorical equivalences for every pair of vertices \(x, y\) of~\(\C\). It then follows that a map \(\vphi\colon \C \lrar \C'\) of \(\infty\)-bicategories is a bicategorical equivalence if and only if it is essentially surjective (that is, every object in \(\C'\) is equivalent to an object in the image, see Definition~\ref{d:equivalence})
	and the induced map \(\Hom_{\C}(x,y) \lrar \Hom_{\C'}(\vphi(x),\vphi(y))\) is a marked categorical equivalence of (fibrant) marked simplicial sets for every \(x,y \in \C\).
\end{rem}

\begin{rem}\label{r:underlying}
	It follows from Remark~\ref{r:dwyer-kan} that if \(\vphi\colon\C \lrar \C'\) is a bicategorical equivalence of \(\infty\)-bicategories then the induced map \(\vphi^{\thi}\colon\C^{\thi} \lrar (\C')^{\thi}\) is an equivalence of \(\infty\)-categories.
\end{rem}

It is shown in~\cite[Proposition 3.1.8 and Lemma 4.2.6]{LurieGoodwillie} that the cartesian product
\[ \times:\Ss \times \Ss \longrightarrow \Ss \]
is a left Quillen bifunctor with respect to the bicategorical model structure,
\ie \(\Ss\) is a cartesian closed model category. In particular, for every two scaled simplicial sets \(X, Y\) we have a mapping object \(\Fun(X, Y)\) which satisfies (and is determined by) the exponential formula
\[ \Hom_{\Ss}(Z,\Fun(X,Y)) \cong \Hom_{\Ss}(Z \times X,Y) .\]
In addition, when the codomain is an \(\infty\)-bicategory \(\C\) the mapping object \(\Fun(X,\C)\) is an 
\(\infty\)\nbd-bi\-category as well, which we can consider as the \(\infty\)-bicategory of functors from \(X\) to \(\C\). 
In this case we will denote by \(\Fun^{\thi}(X,\C) \subseteq \Fun(X, \C)\) the associated core \(\infty\)-category,
which we consider as the \ndef{\(\infty\)-category of functors} from \(X\) to \(\C\).

\begin{define}
We define \(\BiCat\) to be the scaled coherent nerve of the (large)
\(\s^+\)-enriched category \(\mathrm{BiCat}_{\Del}\) whose objects are the \(\infty\)-bicategories
and whose mapping marked simplicial set, for \(\C,\D \in \mathrm{BiCat}_{\Del}\),
is given by \(\mathrm{BiCat}_{\Del}(\C,\D) := \Fun^{\thi}(\C,\D)^{\natural}\).
Here by \((-)^{\natural}\) we mean that the associated marked simplicial set in which the marked arrows are the equivalences.
We will refer to \(\BiCat\) as the \emph{\(\infty\)-bicategory of \(\infty\)-bicategories.}
\end{define}

Since the scaled coherent nerve functor \(\Nsc\) is a right Quillen equivalence it determines an equivalence 
\begin{equation}\label{e:enriched-cat-model}
(\nCat{\s^+})_{\infty} \xrightarrow{\simeq} \BiCat^{\thi}
\end{equation}
between the \(\infty\)-category associated to the model category \(\nCat{\s^+}\) and the core \(\infty\)-category of \(\BiCat\).
One of the technical advantages of using \(\nCat{\s^+}\) as a model is that the \((\ZZ/2)^2\)-action on the theory of \((\infty,2)\)-categories can be realized by an action of \((\ZZ/2)^2\) on \(\nCat{\s^+}\) via model category isomorphisms. More precisely, the operation \(\C \mapsto \C^{\op}\) which inverts only the direction of 1-morphisms is realized by setting \(\C^{\op}(x,y) = \C(y,x)\), while the operation \(\C \mapsto \C^{\co}\) of inverting only the direction of 2-morphisms is realized by setting \(\C^{\co}(x,y) = \C(x,y)^{\op}\), where the right hand side denotes the operation of taking opposites in marked simplicial sets. We will also denote by \(\C^{\coop} := (\C^{\co})^{\op} = (\C^{\op})^{\co}\) the composition of these operations.

\begin{const}\label{cn:op-action}
The two commuting involutions \((-)^{\op}\) and \((-)^{\co}\) act on \(\nCat{\s^+}\) via equivalences of categories which preserve the Dwyer-Kan model structure. Through the equivalence~\eqref{e:enriched-cat-model} 
these two involutions induce a \((\ZZ/2)^{2}\)-action on the core \(\infty\)-category \(\BiCat^{\thi}\), which we then denote by the same notation. In particular, we have involutions 
\[(-)^{\op}\colon \BiCat^{\thi} \to \BiCat^{\thi} \quad\text{and}\quad (-)^{\co}\colon \BiCat^{\thi} \to \BiCat^{\thi},\] 
the first inverting the direction of 1-morphisms and the second the direction of 2-morphisms. 
\end{const}

\begin{example}\label{ex:op-on-cat}
The \(\op\)-action on the core \(\infty\)-category \(\Catoo^{\thi}\) admits a point-set model via the functor \((-)^{\op}\colon \s^+ \to \s^+\), which is an equivalence of categories which preserves the marked categorical model structure. The functor \((-)^{\op}\) is however not an enriched functor. Instead, it refines to an enriched functor
\[ (-)^{\op}\colon \s^+ \to (\s^+)^{\co} ,\]
and hence induces an equivalence
\[ (-)^{\op}\colon \Catoo \xrightarrow{\simeq} \Catoo^{\co} \]
upon restricting the fibrant objects and taking scaled coherent nerves. 
This can also be phrased by saying that the functor \((-)^{\op}\) endows the \(\infty\)-bicategory \(\Catoo \in \BiCat^{\thi}\) with a fixed  point structure under the action of \((-)^{\co}\colon \BiCat^{\thi} \to \BiCat^{\thi}\), which is also sometimes called a \emph{twisted action}.
\end{example}

\begin{example}\label{ex:co-on-cat}
Every Kan complex \(X\) admits a canonical zig-zag of equivalences \(X \xleftarrow{\simeq} \Tw(X) \xrightarrow{\simeq} X^{\op}\), where \(\Tw(X)\) denotes the \emph{twisted arrow category} of \(X\). In particular, the restriction of \((-)^{\op}\colon \Catoo^{\thi} \to \Catoo^{\thi}\) to the full subcategory spanned by Kan complexes is homotopic to the identity. It then follows that the restriction of the equivalence \((-)^{\co}\colon \BiCat^{\thi} \to \BiCat^{\thi}\) to the full subcategory \(\Catoo^{\thi} \subseteq \BiCat^{\thi}\) (which can be modeled by the full subcategory of \([\nCat{\s^+}]^{\circ}\) spanned by the Kan-complex-enriched categories)
is homotopic to the identity as well.
\end{example}

\begin{rem}\label{r:twisted-op}
The \((\ZZ/2)^2\)-action on \(\BiCat^{\thi}\) does not extend to an action of \(\ZZ/2\) on the \(\infty\)-bicategory \(\BiCat\). Instead, as in Example~\ref{ex:op-on-cat}, it extends to a \emph{twisted} action.
More precisely, note that by construction, the enrichment of \(\BiCat\) in \(\Catoo^{\thi}\) is the one induced by the closed action of \(\Catoo^{\thi}\) on \(\BiCat^{\thi}\) via the inclusion \(\Catoo^{\thi} \subseteq \BiCat^{\thi}\) and the cartesian product in \(\BiCat\) (see~\cite[\S 7]{GeHaEnriched} for the relation between \(\Catoo\)-enrichment and closed actions of \(\Catoo\)). In particular, since the \(\co\)-action on \(\BiCat^{\thi}\) fixes \(\Catoo^{\thi} \subseteq \BiCat^{\thi}\) object-wise (see Example~\ref{ex:co-on-cat}) it follows that it extends to an equivalence of \(\infty\)-bicategories 
\[(-)^{\co}\colon \BiCat \xrightarrow{\simeq} \BiCat.\] 
However, since \((-)^{\op}\) restricts to the usual opposite operation on \(\Catoo\), its action on \(\BiCat\) is contravariant in \(2\)-morphisms, and it hence extends to an equivalence 
\[(-)^{\op}\colon \BiCat \xrightarrow{\simeq} \BiCat^{\co},\]
similarly to Example~\ref{ex:op-on-cat}.
\end{rem}

\subsection{Gray products of scaled simplicial sets}\label{s:scaled-gray}

In this section we recall from~\cite{GagnaHarpazLanariGrayLaxFunctors} the definition of the \defn{Gray product} of two scaled simplicial sets. 
In what follows, when we say that a \(2\)-simplex \(\sig\colon \Del^2 \to X\) \defn{degenerates along} \(\Del^{\{i, i+1\}} \subseteq \Del^2\) (for \(i=0,1\)) we mean that \(\sig\) is degenerate and \(\sig_{|\Del^{\{i,i+1\}}}\) is degenerate. This includes the possibility that \(\sig\) factors through the surjective map \(\Del^2 \to \Del^1\) which collapses \(\Del^{\{i, i+1\}}\) as well as the possibility that \(\sig\) factors through \(\Del^2 \to \Del^0\).

\begin{define}\label{d:gray}
	Let \((X,T_X),(Y,T_Y)\) be two scaled simplicial sets. The \defn{Gray product} \((X,T_X) \otimes (Y,T_Y)\) is the scaled simplicial set whose underlying simplicial set is the cartesian product of \(X \times Y\) and such that a \(2\)-simplex \(\sig\colon \Del^2 \to X \times Y\) is thin if and only if the following conditions hold:
	\begin{enumerate}[leftmargin=*]
		\item
		\(\sig\) belongs to \(T_X \times T_Y\); 
		\item
		either the image of \(\sig\) in \(X\) degenerates along \(\Del^{\{1,2\}}\) or the image of \(\sig\) in \(Y\) degenerates along \(\Del^{\{0,1\}}\).
	\end{enumerate}
\end{define}

\begin{rem}\label{r:associative}
The Gray product of scaled simplicial sets is \emph{associative}~\cite[Proposition 2.2]{GagnaHarpazLanariGrayLaxFunctors}, and in particular can be iterated in a unique manner. Specifically, for scaled simplicial sets \(X_1,...,X_n\), the iterated Gray product \(X_1 \mgr \cdots \mgr X_n\) is given by the scaled simplicial set whose underlying simplicial set is the cartesian product of \(X_1,...,X_n\) and such that a triangle 
\(\sig = (\sig_1,...,\sig_n)\colon\Del^2_{\flat} \to X_1 \otimes \cdots \otimes X_n\) is thin if and only if
the following conditions hold:
\begin{enumerate}[leftmargin=*]
\item
Each \(\sig_i\) is thin in \(X_i\).
\item
There exists an \(j \in \{1,...,n\}\) such that \(\sig_i\) degenerates along \(\Del^{\{1,2\}}\) for \(i < j\) and \(\sig_i\) degenerates along \(\Del^{\{0,1\}}\) for \(i > j\).
\end{enumerate}
The \(0\)-simplex \(\Del^0\) can be considered as a scaled simplicial set in a unique way, and serves as the unit of the Gray product. In particular \(\mgr\) is a \emph{monoidal structure} on \(\Ss\). This monoidal structure is however \emph{not} symmetric. Instead, there is a natural isomorphism
	\[ X \otimes Y \cong (Y^{\op} \otimes X^{\op})^{\op}. \]
\end{rem}

\begin{example}\label{ex:lax-square}
	Consider the Gray product \(X = \Del^1 \otimes \Del^1\). Then \(X\) has exactly two non-degenerate triangles \(\sig_1,\sig_2\colon \Del^2 \to X\), where \(\sig_1\) sends \(\Del^{\{0,1\}}\) to \(\Del^{\{0\}} \times \Del^1\) and \(\Del^{\{1,2\}}\) to \(\Del^1 \times \Del^{\{1\}}\), and \(\sig_2\) sends \(\Del^{\{0,1\}}\) to \(\Del^1 \times \Del^{\{0\}}\) and \(\Del^{\{1,2\}}\) to \(\Del^{\{1\}} \times \Del^1\). By definition we see that \(\sig_2\) is thin in \(X\) but \(\sig_1\) is not. If \(\C\) is an \(\infty\)-bicategory then a map \(p \colon X \to \C\) can be described as a diagram in \(\C\) of the form
	\[
	 \begin{tikzcd}[column sep=3em, row sep=large]
	  x \ar[r, "f_0"] \ar[d, "g_0"'] \ar[rd, "h"{description}, ""{swap, name=diag}] &
	  y \ar[d, "g_1"] \\
	  z \ar[r, "f_1"{swap}] & w
	  \ar[Rightarrow, from=diag, to=2-1]
	  \ar[from=2-1, to=1-2, phantom, "\simeq"{description, pos=0.75}]
	 \end{tikzcd}
	\]
	whose upper right triangle is thin (here \(f_i = p_{|\Del^1 \times \{i\}}\) and \(g_i = p_{|\{i\} \times \Del^1}\)). We thus have an invertible \(2\)-cell \(h \x{\simeq}{\Longrightarrow} g_1 \circ f_0\) and a non-invertible \(2\)-cell \(h \Rightarrow f_1 \circ g_0\). Such data is essentially equivalent to just specifying a single non-invertible \(2\)-cell \(g_1 \circ f_0 \Rightarrow f_1 \circ g_0\). We may hence consider such a square as a \defn{oplax-commutative} square, or a square which commutes up to a prescribed \(2\)-cell.
\end{example}

One of the main results of~\cite{GagnaHarpazLanariGrayLaxFunctors} is the following:
\begin{prop}[{\cite[Theorem 2.16]{GagnaHarpazLanariGrayLaxFunctors}}]\label{p:gray-quillen}
The Gray product 
\[ \mgr \colon \Ss \times \Ss \to \Ss\]
is a left Quillen bifunctor with respect to the bicategorical model structure.
\end{prop}

Proposition~\ref{p:gray-quillen} together with Remark~\ref{r:associative} imply that \(\Ss\) is a \emph{monoidal model category} with respect to the Gray product. In particular, one may associate with \(\otimes\) a right and a left mapping objects, which we shall denote by \(\RMap(X,Y)\) and \(\LMap(X,Y)\) respectively. More explicitly, an \(n\)-simplex of \(\RMap(X,Y)\) is given by a map of scaled simplicial sets
\[ \Del^n_{\flat} \otimes X \longrightarrow Y .\]
A \(2\)-simplex \(\Del^2_{\flat} \otimes X \to Y\) of \(\RMap(X,Y)\) is thin if it factors through \(\Del^2_{\sharp} \otimes X\). Similarly, an \(n\)-simplex of \(\LMap(X,Y)\) is given by a map of scaled simplicial sets
\[ X \otimes \Del^n_{\flat} \longrightarrow Y \]
and the scaling is determined as above. The compatibility of the Gray product of the bicategorical model structure then implies that for a fixed \(X\) the functors \(Y \mapsto \RMap(X,Y)\) and \(Y \mapsto \LMap(X,Y)\) are right Quillen functors. In particular, if \(\C\) is an \(\infty\)-bicategory then \(\RMap(X,\C)\) and \(\LMap(X,\C)\) are \(\infty\)-bicategories as well. The objects of \(\RMap(X,\C)\) correspond to functors \(X \to \C\) and by Example~\ref{ex:lax-square} we may consider morphisms in \(\RMap(X,\C)\) as \defn{lax natural transformations}. If we take \(\LMap(X,\C)\) instead then the objects are again functors \(X \to \C\), but now the edges will correspond to \defn{oplax natural transformations}.

\subsection{Scaled straightening and unstraightening}\label{s:straightening}

In~\cite[\S 3]{LurieGoodwillie} Lurie established a straightening-unstraightening equivalence in the setting of \(\infty\)-bicategories. In this subsection 
we recall the setup of~\cite[\S 3]{LurieGoodwillie} and explain how to obtain from it an equivalence on the level of \(\infty\)-bicategories.

\begin{define}\label{d:fibered}
Let \((S,T_S)\) be a scaled simplicial set. A marked simplicial set \((X,E_X)\) equipped with a map of simplicial sets \(f \colon X \to S\)
is said to be \ndef{\(\Beta_S\)-fibered} if the following conditions hold:
\begin{enumerate}[leftmargin=*]
	\item[(i)]
	The map \(f\) is an inner fibration.
	\item[(ii)]\label{it:locally}
	For every edge \(e\colon\Del^1 \to S\) the map \(e^*f \colon X \times_{S} \Del^1 \to \Del^1\) is a cocartesian fibration
	and the marked edges of \(X\) lying over \(e\) are exactly the \(e^*f\)-cocartesian edges.
	\item[(iii)]
	For every commutative diagram
	\[
	\begin{tikzcd}
		\Del^{\{0,1\}} \ar[r, "e"]\ar[d] & X \ar[d] \\
		\Del^2 \ar[r, "\sigma"] & S
	\end{tikzcd}
	\]
	with \(e \in E_X\) and \(\sig \in T_S\), the edge of \(X \times_S \Del^2\) determined by~\(e\) is \(\sig^*f\)-cocartesian. 
\end{enumerate}
\end{define}
Let \(\trbis{(\s^+)}{S}\) denote the category of marked simplicial sets \((X,E_X)\) equipped with a map of simplicial sets \(f \colon X \to S\). 
In~\cite[\S 3.2]{LurieGoodwillie} Lurie constructs a model structure on \(\trbis{(\s^+)}{S}\) 
whose cofibrations are the monomorphisms and whose fibrant objects are exactly the \(\Beta_S\)-fibered objects. 
Let us refer to this model structure as the \ndef{\(\Beta_S\)-fibered model structure}. Given a weak equivalence of \(\s^+\)-enriched categories \(\phi\colon \fC(S,T_S) \to \C\) he then proceeds to construct a \ndef{straightening-un\-straight\-ening} Quillen equivalence
\[
 \xymatrixcolsep{1pc}
 \vcenter{\hbox{\xymatrix{
			**[l] \trbis{(\s^+)}{S} \xtwocell[r]{}_{\Un^{\sca}_{\phi}}^{\Str^{\sca}_{\phi}}{'\perp}& **[r] (\s^+)^{\C} }
				}
			}
\]
where the right hand side denotes the category of \(\s^+\)-enriched functors \(\C \to \s^+\) equipped with the projective model structure, and the left hand side is equipped with the \(\Beta_S\)-fibered model structure. The straightening functor \(\Str^{\sca}_{\phi}\) is given by the explicit formula  
\[ [\Str^{\sca}_{\phi}(X,E_X)](v) = \Cone_{\phi}(X,E_X)(\ast,v),\]
that is, by the restriction to \(\C\) of the functor \(\Cone_{\phi}(X,E_X) \to \s^+\) represented by the cone point \(\ast\) in the scaled cone of \(X\) over \(\C\), which by is defined by
\[\Cone_{\phi}(X,E_X) := \fCs\Big(\Del^0 \coprod_{\Del^{\{0\}} \times X}( \Del^1\times X,T)\Big) \coprod_{\fCs(\Del^{\{1\}} \times X_{\flat})}\C. \]
Here, the second pushout is along the composed map \(\fCs(X_{\flat}) \to \fCs(S,T_S) \to \C\),
and \(T\) denotes the set of all those triangles
\((\tau,\sig_X)\colon \Del^2 \to \Del^1 \times X\) such that
\(\sig_X\) is degenerate and either \(\tau_{|\Del^{\{1,2\}}}\) 
is degenerate in \(\Del^1\) or \({\sig_X}_{|\Del^{\{0,1\}}}\) belongs to \(E_X\). 
In the case where \(\phi\colon \fCs(S,T_S) \to \fCs(S,T_S)\) is the identity we will generally replace the subscript \(\phi\) in \(\Str^{\sca}_{\phi}\) and \(\Un^{\sca}_{\phi}\) and \(\Cone_{\phi}\) by the subscript \((S,T_S)\).

\begin{rem}\label{r:st-gray-scaled}
When \(X\) has no non-degenerate marked edges
the cone \(\Cone_{\phi}(X^{\flat})\) appearing in the straightening construction above can be described in terms of the Gray product of \S\ref{s:scaled-gray} by
\[ \Cone_{\phi}(X^{\flat}) = \fCs\Big(\Del^0 \coprod_{\Del^{\{0\}} \otimes X_{\flat}}[{}^{\flat}\Del^1 \otimes X_{\flat}]\Big)\coprod_{\fCs(\Del^{\{1\}} \otimes X_{\flat})} \C. \]
This also holds more generally if one uses a Gray product which takes into account marked edges, see Remark~\ref{r:st-gray}. 
\end{rem}

The straightening-unstraightening Quillen equivalence induces an equivalence between the \(\infty\)-categories underlying the two sides of the adjunction \(\Str^{\sca}_{\phi} \dashv \Un^{\sca}_{\phi}\). These two sides are both model categories which are \emph{tensored} over \(\s^+\), that is, they admit a closed action of \(\s^+\) in the form of a left Quillen bifunctor, and in particular both acquire en enrichment in \(\s^+\). In addition, the unstraightening functor is lax-compatible with the action of \(\s^+\), in the sense that one has a natural map
\begin{equation}\label{e:lax-tensored} 
\Un^{\sca}_{\phi}(\F) \times K \to \Un^{\sca}_{\phi}(\F) \times \Un^{\sca}_{\ast}(K) \cong \Un^{\sca}_{\phi}(\F \times K), 
\end{equation}
where \(\Un^{\sca}_{\ast}\) denotes the scaled unstraightening functor with respect to the isomorphism \(\fC(\Del^0)\cong \ast\). This structure promotes \(\Un^{\sca}_{\phi}\) to a \(\s^+\)-enriched functor from \((\s^+)^{\C}\) to \(\trbis{(\s^+)}{S}\). Passing to the full subcategories of fibrant-cofibrant objects (and using the simplifying fact that all objects in \(\trbis{(\s^+)}{S}\) are cofibrant) we then obtain an enriched functor of fibrant \(\s^+\)-enriched categories
\[ [\Un^{\sca}_{\phi}]^{\circ} \colon [(\s^+)^{\C}]^{\circ} \to [\trbis{(\s^+)}{S}]^{\circ} .\]

\begin{lemma}\label{l:unstraightening-simplicial}
The functor \([\Un^{\sca}_{\phi}]^{\circ}\)
is a Dwyer-Kan equivalence.
\end{lemma}

Taking scaled nerves we now obtain a form of the unstraightening construction as an equivalence of \(\infty\)-bicategories 
\[\Nsc[(\s^+)^{\C}]^{\circ} \xrightarrow{\simeq} \Nsc[\trbis{(\s^+)}{S}]^{\circ}.\]

\begin{proof}[Proof of Lemma~\ref{l:unstraightening-simplicial}]
To begin, note that this functor is essentially surjective since \(\Un^{\sca}_{\phi}\) is a right Quillen equivalence. To see that it is also homotopically fully-faithful we use the fact that the map~\eqref{e:lax-tensored} is a weak equivalence whenever \(\F\) and \(K\) are fibrant by~\cite[Proposition 3.6.1]{LurieGoodwillie} and~\cite[Corollary~1.4.4(b)]{HoveyModelCategories}. Then, for every fibrant \(\F,\G \in (\s^+)^{\C}\) and a fibrant \(K \in \s^+\) the functor \(\F \times K\) is again fibrant and the induced map of sets
\begin{align*}
[K,(\s^+)^{\C}(\F,\G)]_{\s^+} \cong [\F \times K,\G]_{(\s^+)^{\C}} \\
\xrightarrow{\cong} [\Un^{\sca}_{\phi}(\F \times K),\Un^{\sca}_{\phi}(\G)]_{\trbis{(\s^+)}{S}}\\ \xrightarrow{\cong} [\Un^{\sca}_{\phi}(\F) \times K,\Un^{\sca}_{\phi}(\G)]_{\trbis{(\s^+)}{S}} \\
[K,\trbis{(\s^+)}{S}(\F,\G)]_{\s^+}
\end{align*}
is a bijection, where \([-,-]\) denotes sets of homotopy classes of maps with respect to the relevant model structure. It then follows that \(\Un^{\sca}_{\phi}\) induces a weak equivalence of marked simplicial sets, that is, a marked categorical equivalence
\[ (\s^+)^{\C}(\F,\G) \xrightarrow{\simeq} \trbis{(\s^+)}{S}(\F,\G),\]
for every \(\F,\G \in [(\s^+)^{\C}]^{\circ}\).
\end{proof}

\subsection{Marked-scaled simplicial sets}

In our treatment of fibrations of \(\infty\)-bicategories, it will be useful to work in a setting where we have both a scaling and a marking.

\begin{define}
	A \ndef{marked-scaled simplicial set} is a triple \((X,E_X,T_X)\) where \(X\) is a simplicial set, \(E_X\) is a collection of edges containing all the degenerate edges and \(T_X\) is collection of \(2\)-simplices containing all the degenerate \(2\)-simplices. In particular, if \((X,E_X,T_X)\) is a marked-scaled simplicial set then \((X,E_X)\) is a marked simplicial set and \((X,T_X)\) is a scaled simplicial set. A map of marked-scaled simplicial sets \((X,E_X,T_X) \to (T,E_Y,T_Y)\) is a map of simplcial sets \(X \to Y\) such that \(f(E_X) \subseteq E_Y\) and \(f(T_X) \subseteq T_Y\).  
\end{define}

We will denote by \(\Sms\) the category of marked-scaled simplicial sets. It is locally presentable and cartesian closed.

\begin{define}
	For \(X \in \s^+\) a marked simplicial set we will denote by \(X_\flat = (X,E_X,\deg_2(X))\) 
	the marked-scaled simplicial set which has the same marking as \(X\) and only the degenerate \(2\)-simplices are thin, 
	and by \(X_\sharp = (X,E_X,X_2)\) the marked-scaled simplicial set which has the same marking as \(X\) and all \(2\)-simplices are thin. 
	For \(Y\) is a scaled simplicial set then we will denote by \(Y^{\flat} = (Y,\deg_1(Y),T_Y)\) 
	the marked-scaled simplicial set which has the same scaling as \(Y\) and only the degenerate edges marked 
	and by \(Y^{\sharp} = (Y,Y_1,T_Y)\) the marked-scaled simplicial which has the same scaling as \(Y\) and all edges are marked. 
	Finally, for \(Z\) a simplicial set we will denote by  \(\prescript{\flat}{}{Z} = (Z,\deg_1(Z),\deg_2(Z))\) and \(\prescript{\sharp}{}{Z} = (Z,Z_1,Z_2)\) 
	the corresponding minimal and maximal marked-scaled simplicial sets as indicated. By abuse of notation we will denote \(\prescript{\sharp}{}{\Del^0} \cong \prescript{\flat}{}{\Del^0}\) simply by \(\Del^0\).
\end{define}

\begin{define}\label{d:underlying}
For a marked-scaled simplicial set \(X\) we will denote by \(\ovl{X}\) the underlying \emph{scaled} simplicial set. 
\end{define}

\begin{define}\label{d:marked-bicategory}
By a \emph{marked \(\infty\)-bicategory} we will simply mean a marked-scaled simplicial set whose underlying scaled simplicial set is a weak \(\infty\)-bicategory.
\end{define}

\section{Inner and outer cartesian fibrations}
\label{s:fibrations}

In this section we will study four types of fibrations between \(\infty\)-categories, which we call inner cocartesian, inner cartesian, outer cocartesian and outer cartesian fibrations. The notion of an inner cocartesian fibration \(\E \to \B\) is essentially equivalent to that of a \(\Beta_{\B}\)-fibered object, as described in \S\ref{s:straightening}. We will make the comparison precise in \S\ref{s:slice}, see Proposition~\ref{p:inner-fibred}. Assuming this for the moment, Lurie's straightening-unstraightening equivalence then implies that 
\(\E \to \B\) can be obtained as the unstraightening of an objectwise fibrant functor \(\fCs(\B) \to \s^+\), which we can also encode as a map \(\chi\colon \B \to \Catoo\). Using the compatibility of straightening-unstraightening with base change we may informally describe \(\chi\)  
by the formula \(b \mapsto \E_b := \E \times_{\B}\{b\}\). 

Dually, a map \(f \colon \E \to \B\) is an inner \emph{cartesian} fibration if \(f^{\op}\colon \E^{\op} \to \B^{\op}\) is an inner cocartesian fibration. Then \(f^{\op}\) encodes the data of a diagram \(\B^{\op} \to \Catoo\), given informally by the formula \(b \mapsto \E_b^{\op}\). Noting (see Example~\ref{ex:op-on-cat}) that the functor \((-)^\op\) yields an equivalence of \((\infty,2)\)-categories \((-)^{\op}\colon \Catoo \to \Catoo^{\co}\), we may consider the association \(b \mapsto \E_b\) as a functor \(\B^{\coop} \to \Catoo\).

In the paper~\cite{GagnaHarpazLanariEquiv} we have introduced two more notions of fibrations, which we call \emph{outer} cartesian and cocartesian fibrations, respectively. Our primary goal, to which we will arrive in \S\ref{sec:correspondence}, is to show that the data of an outer cartesian fibration \(f\colon \E \to \B\) encodes a functor \(\B^{\op} \mapsto \Cat\), while that of an outer cocartesian fibration encodes a functor \(\B^{\co} \mapsto \Cat\). In particular, the four types of fibrations mentioned above correspond exactly to the four variance types a \(\Cat\)-valued diagram can have. For this, we will dedicate the present section to studying the properties of these four types of fibrations, establishing the key results about them we will need later on. In particular, in \S\ref{s:cartesian} we will study the maps induced by these fibrations on the level of mapping \(\infty\)-categories, and show that these are always left and right fibrations. In \S\ref{sec:car-edges} we will give several equivalent characterizations of (co)cartesian edges, and deduce in particular that any invertible edge is (co)cartesian. In \S\ref{s:slice} we will use this in order to prove that outer cartesian and cocartesian fibrations can be characterized via an extension property against a certain class of anodyne maps, in manner similar in principal to that which allows one to identify inner cocartesian fibrations with \(\Beta_{\B}\)-fibered objects. Finally, in \S\ref{s:lift} we will prove that these types of fibrations admit a lifting property for (op)lax transformations.

\subsection{Recollections}\label{s:prelim}

In this section we recall from~\cite{GagnaHarpazLanariEquiv} the main definitions we will need, and recall some of their properties which were already established in \loccit

\begin{define}\label{d:weak}
	We will say that a map of scaled simplicial sets \(X \rightarrow Y\) is a \ndef{weak fibration} if it has the right lifting property with respect to the following types of maps:
	\begin{enumerate}
		\item
		All scaled inner horn inclusions of the form 
		\[ (\Lam^n_i,\{\Del^{\{i,i-1,i\}}\}_{|\Lam^n_i}) \subseteq (\Del^n,\{\Del^{\{i,i-1,i\}}\}) \] 
		for \(n \geq 2\) and \(0 < i < n\).
		\item
		The scaled horn inclusions of the form: \[\Bigl(\Lam^n_0 \plus{\Del^{\{0,1\}}}\Del^0,\{\Del^{\{0,1,n\}}\}_{|\Lam^n_0}\Bigr) \subseteq \Bigl(\Del^n\plus{\Del^{\{0,1\}}}\Del^0,\{\Del^{\{0,1,n\}}\}\Bigr)\] for \(n \geq 2\).
		\item
		The scaled horn inclusions of the form: \[\Bigl(\Lam^n_n \plus{\Del^{\{n-1,n\}}}\Del^0,\{\Del^{\{0,n-1,n\}}\}_{|\Lam^n_n}\Bigr) \subseteq \Bigl(\Del^n\plus{\Del^{\{n-1,n\}}}\Del^0,\{\Del^{\{0,n-1,n\}}\}\Bigr)\] for \(n \geq 2\).
	\end{enumerate}
\end{define}

\begin{rem}\label{r:bicategorical-weak}
	The maps of type (1)-(3) in Definition~\ref{d:weak} are trivial cofibrations with respect to the bicategorical model structure: indeed, the first two are scaled anodyne and the third is the opposite of a scaled anodyne map. It follows that every bicategorical fibration is a weak fibration. 
\end{rem}

\begin{rem}\label{r:weak-scaled}
Let \(f\colon X \to Y\) be a weak fibration and suppose in addition that \(f\) detects thin triangles, that is, a triangle in \(X\) is thin if and only if its image in \(Y\) is thin. Then \(f\) has the right lifting property with respect to the generating scaled anodyne maps of Definition~\ref{d:anodyne}. In particular, if \(Y\) is an \(\infty\)-bicategory then \(X\) is an \(\infty\)-bicategory. In addition, in this case for every \(y \in Y\) the fiber \(X_y\) is an \(\infty\)-bicategory in which every triangle is thin, and can hence be considered as an \(\infty\)-category.
\end{rem}

\begin{define}\label{d:cartesian}
	Let \(f\colon X \rightarrow Y\) be a weak fibration. We will say that an edge \(e\colon \Del^1 \rightarrow X\) is \ndef{\(f\)-cartesian} if the dotted lift exists in any diagram of the form
	\[ \xymatrix{
		(\Lam^n_n,\{\Del^{\{0,n-1,n\}}\}_{|\Lam^n_n}) \ar^-{\sig}[r]\ar[d] & (X,T_X) \ar^f[d] \\
		(\Del^n,\{\Del^{\{0,n-1,n\}}\}) \ar@{-->}[ur]\ar[r] & (Y,T_Y) \\
	}\]
	with \(n \geq 2\) and \(\sig_{|\Del^{n-1,n}} = e\). We will say that \(e\) is \(f\)-cocartesian if \(e^{\op}\colon \Del^1 \to X^{\op}\) is \(f^{\op}\)-cartesian.
\end{define}

\begin{define}
Let \(f\colon X \rightarrow Y\) be a weak fibration. 	
We will say that \(f\) is 
\begin{enumerate}
\item an \ndef{inner fibration} if it detects thin triangles and the underlying map of simplicial sets is in inner fibration, that is, satisfies the right lifting property with respect to inner horn inclusions;
\item an \ndef{outer fibration} if it detects thin triangles and  
the underlying map of simplicial sets 
satisfies the right lifting property with respect to the inclusions 
\[\Lam^n_0 \coprod_{\Del^{\{0,1\}}}\Del^0 \subseteq \Del^n\coprod_{\Del^{\{0,1\}}}\Del^0 \quad \text{and} \quad
\Lam^n_n \coprod_{\Del^{\{n-1,n\}}}\Del^0 \subseteq \Del^n\coprod_{\Del^{\{n-1,n\}}}\Del^0\]
for \(n \geq 2\). 
\end{enumerate}
\end{define}

\begin{warning}
	In \cite[\href{https://kerodon.net/tag/01WF}{Tag 01WF}]{LurieKerodon}, Lurie uses the term \emph{interior fibration} to encode what we just defined as outer fibrations. Our choice already appeared in Definition 2.4 of \cite{GagnaHarpazLanariEquiv}, and it is motivated by the intent of highlighting that \emph{special outer horns} admit fillers against such maps.
\end{warning}

\begin{define}\label{d:car-fibration}
	Let \(f\colon X \rightarrow Y\) be a map of scaled simplicial sets.
	We will say that \(f\) is an \ndef{outer} (resp.~\ndef{inner}) \ndef{cartesian fibration} if the following conditions hold:
	\begin{enumerate}[leftmargin=*]
		\item
		The map \(f\) is an outer (resp.~inner) fibration.
		\item
		For every \(x \in X\) and an edge \(e\colon y \rightarrow f(x)\) in \(Y\) there exists a \(f\)-cartesian edge \(\wtl{e}\colon \wtl{y} \to x\) such that \(f(\wtl{e}) = e\). 
	\end{enumerate}
	Dually, we will say that \(f\colon X \rightarrow Y\) is an \ndef{outer cocartesian fibration} if \(f^{\op}\colon X^{\op} \rightarrow Y^{\op}\) is an outer cartesian fibration.
\end{define}

\begin{rem}\label{r:base-change}
	The classes of weak fibrations, inner/outer fibrations and inner/outer (co)cartesian fibrations are all closed under base change.
\end{rem}

It follows from Remark~\ref{r:weak-scaled} that if \(f\colon X \to Y\) is an inner/outer (co)cartesian fibration and \(X\) is an \(\infty\)-bicategory then \(Y\) is an \(\infty\)-bicategory as well. In this case we will say that \(f\) is an inner/outer (co)cartesian fibration of \(\infty\)-bicategories.

\begin{rem}\label{r:iso}
	Let \(f\colon\E \to \B\) be an inner/outer (co)cartesian fibration of \(\infty\)-bicat\-e\-gor\-ies. 
	Then the base change \(f_{|\B^{\thi}} \colon \E \times_{\B} \B^{\thi} \to \B^{\thi}\)
	(see Definition~\ref{core defi}) is a (co)cartesian fibration of \(\infty\)-categories. 
	In particular, \(f_{|\B^{\thi}}\) is a categorical fibration (see~\cite[Proposition~3.3.1.7]{HTT})
	and so an isofibration. We may hence conclude that \(f\) is an isofibration of \(\infty\)-bicategories. 
\end{rem}

\begin{rem}\label{r:cart-equiv}
	In the setting of Remark~\ref{r:iso}, if \(e\colon x \to y\) is a \(f\)-(co)cartesian edge of \(\E\),
	then it also (co)cartesian with respect to the (co)cartesian fibration of \(\infty\)-categories \(\E^{\thi} \to \B^{\thi}\).
	This implies, in particular, that any \(f\)-(co)cartesian edge which lies over an equivalence in \(\B\) is necessarily an equivalence in \(\E\).
\end{rem}

\begin{rem}
\label{cart fib are fib}
	An inner/outer (co)cartesian fibration \(f\colon \E \to \B\) of \(\infty\)-bicategories is a fibration of scaled simplicial sets. This follows the characterization of the bicategorical model structure established in~\cite{GagnaHarpazLanariEquiv} since \(f\) lifts against scaled anodyne maps by virtue of being a weak fibration and is an isofibration by Remark~\ref{r:iso}.
\end{rem}

\begin{rem}\label{r:fibers}
	It follows from Remarks~\ref{cart fib are fib} and~\ref{r:base-change} that if \(X \to Y\) is an inner/outer fibration then for every \(y \in Y\) the fiber \(X_y\) is an \(\infty\)-bicategory in which every triangle is thin. Forgetting the scaling, we may simply consider these fibers as \(\infty\)-categories.
\end{rem}

An important source of examples of outer (co)cartesian fibrations comes from \emph{slice fibrations}. Let us recall from~\cite{GagnaHarpazLanariEquiv} the relevant definitions.

\begin{define}\label{d:join}
	Let \((X,E_X,T_X)\) and \((Y,E_Y,T_Y)\) be two marked-scaled simplicial sets. We define their \ndef{join} \((X,E_X,T_X) * (Y,E_Y,T_Y) = (X \ast Y,T_{X\ast Y})\) to be the \emph{scaled} simplicial set whose underlying simplicial set is the ordinary join of simplicial sets \(X \ast Y\), and whose thin triangles are given by
	\[
	 T_{X\ast Y} = T_X \coprod (E_X \times Y_0) \coprod (X_0 \times E_Y) \coprod T_Y
	\]
	seen as a subset of 
	\[
	 (X \ast Y)_2 = X_2 \coprod (X_1 \times Y_0) \coprod (X_0 \times Y_1) \coprod Y_2.
	\]
\end{define}

To avoid confusion, we point out that while the input to the join bifunctor above are two marked-scaled simplicial sets, its output is only considered as a scaled simplicial set.

\begin{const}\label{con:slice}
Recall that for a marked-scaled simplicial set \(K\) we denote by \(\ovl{K}\) the underlying \emph{scaled} simplicial set (see Definition~\ref{d:underlying}).  
For a fixed marked-scaled simplicial set \(K\) we may regard the association \(X \mapsto X \ast K\)
as a functor \(\Sms \lrar \overslice{(\Ss)}{\ovl{K}}\). 
As such, it becomes a colimit preserving functor with a right adjoint
\[\overslice{(\Ss)}{\ovl{K}} \lrar \Sms\] 
by the adjoint functor theorem. 
Given a scaled simplicial set \(S\) and a map \(f\colon \ovl{K} \lrar S\), considered as an object of \(\overslice{(\Ss)}{\ovl{K}}\),
we will denote by \(S_{/f}\) the marked-scaled simplicial set obtained by applying the above mentioned right adjoint, and by \(\ovl{S}_{/f}\) the underlying scaled simplicial set of \(S_{/f}\). In particular, the marked-scaled simplicial set \(S_{/f}\) is characterized by the following mapping property
\[
 \Hom_{\Sms}\bigl(X, S_{/f}\bigr) = \Hom_{\overslice{(\Ss)}{\ovl{K}}}\bigl(X \ast K, S\bigr) ,
\]
while \(\ovl{S}_{/f}\) is characterized by the property
\[
 \Hom_{\Ss}\bigl(X, \ovl{S}_{/f}\bigr) = \Hom_{\overslice{(\Ss)}{\ovl{K}}}\bigl(X^{\flat} \ast K, S\bigr) .
\]
\end{const}

\begin{example}
If \(f\colon \emptyset \to S\) is the unique map from the empty scaled simplicial set then \(S_{/f} = S^{\sharp} = (S,S_1,T_S)\) is the marked-scaled simplicial set having the same underlying scaled simplicial set as \(S\) and with all edges marked. In this case we will also use the notation \(S_{/\emptyset}\). In particular, \(\ovl{S}_{/\emptyset} \cong S\) canonically.
\end{example}

\begin{example}
If \(K = \Del^0\) and \(f\colon \Del^0 \to S\) corresponds to a vertex \(x \in S\) then we will denote \(S_{/f}\) also by \(S_{/x}\). This can be considered as a version of the slice construction in the setting of scaled simplicial set. For example, if \(\C\) is an \(\infty\)-bicategory and \(x,y \in \C\) are two objects, then the fiber of \(\C_{/y}\) over \(x \in \C\), which is a marked-scaled simplicial set in which all triangles are thin, is categorically equivalent as a marked simplicial set to the mapping \(\infty\)-category \(\Hom_{\C}(x,y)\) of Notation~\ref{n:mapping}, see~\cite[Proposition 2.24]{GagnaHarpazLanariEquiv}. We may consider this as a ``thinner'' model for the mapping \(\infty\)-category, and will make use of it in \S\ref{s:cartesian} below.
\end{example}

\begin{warning}
The notation \(S_{/f}\) is somewhat abusive - the marked-scaled simplicial set \(S_{/f}\) depends not only on the scaled map \(f\colon \ovl{K} \to S\), but also on the given marking on \(K\). For example, suppose that \(K\) is \(\Del^1\) with some marking \(E \subseteq (\Del^1)_1\) and \(f\) corresponds to a given edge \(e\colon x \to y\) in \(S\). If \(E\) consists only of degenerate edges then the vertices of \(S_{/f}\) are given by arbitrary triangles of the form
\[ \xymatrix{
& z \ar[dl]\ar[dr] & \\
x \ar[rr]^{e} && y \\
}\] 
in \(S\), while if \(E\) contains the non-degenerate edge of \(\Del^1\) then the vertices of \(S_{/f}\) correspond only to those triangles as above which are \emph{thin}. 
\end{warning}

We now recall from~\cite{GagnaHarpazLanariEquiv} the following results concerning the above slice constructions:
 
\begin{prop}[{\cite[Corollary 2.18]{GagnaHarpazLanariEquiv}}]\label{p:slice-fibration}
Let \(K\) be a marked scaled simplicial sets, \(\C\) an \(\infty\)-bicategory and \(f\colon \ovl{K} \to \C\) a scaled map. Then the map of scaled simplicial sets
\[p\colon \ovl{\C}_{/f} \to \ovl{\C}_{/\emptyset} = \C \]
is an outer cartesian fibration such that every marked edge in \(\C_{/f}\) is \(p\)-cartesian
and every edge in \(\C\) admits a marked \(p\)-cartesian lift.
\end{prop}

More generally:
\begin{prop}[{\cite[Corollary 2.20]{GagnaHarpazLanariEquiv}}]\label{p:slice-fibration-2}
Let \(\iota\colon L \subseteq K\) be an inclusion of marked-scaled simplicial sets, \(q\colon X \to S\) is a weak fibration and \(f\colon \ovl{K} \to X\) is a scaled map. Then the map of scaled simplicial sets
\[p\colon \ovl{X}_{/f} \to \ovl{X}_{/f\iota} \times_{\ovl{S}_{/qf\iota}} \times \ovl{S}_{/qf}\]
is an outer fibration such that every marked edge in its domain is \(p\)-cartesian and every marked edge in its codomain admits a marked \(p\)-cartesian lift.
\end{prop}

\begin{rem}
In contrast to the case of Proposition~\ref{p:slice-fibration}, in the situation of Proposition~\ref{p:slice-fibration-2} the map \(p\) is 
an outer fibration which is generally not cartesian: only some of the maps in its codomain admit cartesian lifts.
\end{rem}

The following result is a reformulation in the present language of \cite[Lemma 1.2.8]{GaitsgoryRozenblyumStudy}, which is stated in \loccit without proof.

\begin{prop}
	Let \(\C\) be an \(\infty\)-category and \(f\colon X \to \C_{\sharp}\) be a weak fibration of scaled simplicial sets. 
	Then the following are equivalent:
	\begin{enumerate}
	\item \(f\) is an inner cartesian fibration.
	\item \(f\) is an outer cartesian fibration.
	\item All triangles in \(X\) are thin and the map of simplicial sets underlying \(f\) is a cartesian fibration.
	\end{enumerate}
\end{prop}

\begin{proof}
We first note that both (1) and (2) imply by definition that \(f\) detects thin triangles, and since all triangles in \(\C_{\sharp}\) is thin this is the equivalent to saying that every triangle in \(X\) is thin. Since this is also stated explicitly in (3), we may simply assume that all triangles in \(X\) are thin. In this case, \(f\) is both an inner and an outer fibration as soon as it is a weak fibration, and so (1) and (2) are both equivalent to saying \(f\) is a weak fibration and edges in \(\C\) admits a sufficient supply of \(f\)-cartesian lifts. Using again that all triangles in \(X\) and \(\C_{\sharp}\) are thin we see that this is the same as saying that the map of simplicial sets underlying \(f\) is a cartesian fibration which satisfies the right lifting property with respect to the inclusions
\[\Lam^n_0 \coprod_{\Del^{\{0,1\}}}\Del^0 \subseteq \Del^n\coprod_{\Del^{\{0,1\}}}\Del^0 \quad \text{and} \quad
\Lam^n_n \coprod_{\Del^{\{n-1,n\}}}\Del^0 \subseteq \Del^n\coprod_{\Del^{\{n-1,n\}}}\Del^0\]
for \(n \geq 2\). But this holds for any cartesian fibration of \(\infty\)-categories: indeed, any cartesian fibration is also a categorical fibration and the above inclusions are trivial cofibrations in the categorical model structure.
\end{proof}

\subsection{Local properties of inner and outer fibrations}\label{s:cartesian}

	Let \(\C\) be an \(\infty\)-bicategory and let \(x,y\) be two vertices of \(\C\).
	Recall the explicit model for the mapping \(\infty\)-category \(\Hom_\C(x, y)\) 
	from \(x\) to \(y\) discussed in Notation \ref{n:mapping}. 
	In~\cite[\S 2.3]{GagnaHarpazLanariEquiv}, we have introduced another model \(\Hom^{\triangleright}_\C(x, y)\),
	defined as the underlying marked simplicial set of \((\C_{/y})_x\).
	We have shown the existence of a canonical map (see~\cite[Construction~2.22]{GagnaHarpazLanariEquiv})
	\[
		i \colon  \Hom^{\triangleright}_\C(x, y) \to \Hom_\C(x, y),
	\]
	which is an equivalence of fibrant marked simplicial sets (see~\cite[Proposition~2.24]{GagnaHarpazLanariEquiv}).

In what follows we will use the term \emph{marked left} (resp.~\ndef{right}) \ndef{fibration} to denote a map of marked simplicial sets \(f\colon X \to Y\) which detects marked edges and which is a left (resp.~right) fibration on the level of underlying simplicial sets.

\begin{prop}\label{p:mapping-2}
	Let \(f \colon \E \lrar \B\) be an outer (resp.~inner) fibration of \(\infty\)-bicategories
	and let \(x,y\) be two vertices of \(\E\). 
	Then the map of marked simplicial sets
	\[ f_*\colon \Hom^{\triangleright}_{\E}(x,y) \lrar \Hom^{\triangleright}_{\B}(f(x),f(y)) \]
	is a marked left (resp.~right) fibration. 
	Furthermore, if \(e \colon x' \lrar y\) is a \(f\)-cartesian edge with \(f(x')=f(x)\)
	then the post-composition with \(e\) induces an equivalence between the the mapping space \(\Hom_{\E_{f(x)}}(x,x')\) 
	and the fiber of \(f_*\) over \(f(e)\), where \(\E_{f(x)}\) denotes the fiber of \(f\) over \(f(x)\). 
\end{prop}

\begin{proof}
	Assume first that \(f\) is an outer cartesian fibration. We need to show that the map
	\begin{equation}\label{e:local}
	(\E_{/y})_x \lrar (\B_{/f(y)})_{f(x)} 
	\end{equation}
	is a marked left fibration, where abusing of notation we consider both \((\E_{/y})_x\) and \((\B_{/f(y)})_{f(x)}\)
	as marked simplicial sets, thereby ignoring their scaling. 
	We first observe that since the map \(\E \lrar \B\) has the right lifting property 
	with respect to \(\Del^2 \subseteq \Del^2_{\sharp}\), it follows that \(\E_{/y} \lrar \B_{/f(y)}\) 
	has the right lifting property with respect to \(\Del^1 \subseteq (\Del^1)^{\sharp}\). 
	In particular, an arrow in \((\E_{/y})_x\) is marked if and only if its image in \((\B_{/f(y)})_{f(x)}\) is marked. 
	It will hence suffice to show that the map of simplicial sets underlying~\eqref{e:local} is a left fibration. 
	Unwinding the definitions, we see that in order to show that~\eqref{e:local} has the right lifting property 
	with respect to \(\Lam^n_i \hrar \Del^n\) for \(0 \leq i < n\) we need to prove the existence of a dotted lift in diagrams of the form
	\begin{equation}\label{e:rectangle} 
	 \begin{tikzcd}
		\Del^n \ast \varnothing \coprod_{\Lam^n_i \ast \varnothing} \Lam^n_i \ast \Del^0 \ar[r]\ar[d] & 
		\Del^0 \ast \varnothing \coprod_{\Lam^n_i \ast \varnothing} \Lam^n_i \ast \Del^0 \ar[r, "f"]\ar[d] & 
		\E \ar[d, "f"] \\
		\Del^n \ast \Del^0 \ar[r] & \Del^0 \ast \varnothing \coprod_{\Del^n \ast \varnothing}\Del^n \ast \Del^0
		\ar[r, "g"]\ar[ur, dotted] & \B \\
	 \end{tikzcd}
	\end{equation}
	where \(f\) maps \(\Del^0 \ast \varnothing\) to \(x\), \(g\) maps \(\Del^0 \ast \varnothing\) to \(f(x)\).
	Now since the left square is a pushout square it will suffice to find a lift in the external rectangle of~\eqref{e:rectangle}. 
	When \(0 < i < n\) the left vertical map is isomorphic to the inner horn inclusion 
	\(\Lam^{[n] \ast [0]}_i \subseteq \Del^{[n] \ast [0]}\) 
	and the triangle \(\Del^{\{i-1,i,i+1\}}\) is mapped to a degenerate (and hence thin) triangle of \(\B\). 
	On the other hand, when \(i=0\) the left vertical map is isomorphic to the \(0\)-horn inclusion 
	\(\Lam^{[n] \ast [0]}_0 \subseteq \Del^{[n] \ast [0]}\) and the edge \(\Del^{\{0,1\}}\) 
	is mapped to a degenerate edge of \(\E\). 
	In both cases the desired lift exists by virtue of the assumption that \(f\colon \E \lrar \B\) is an outer fibration.
	
	For the second part of the claim, we need to show that if \(e\colon x' \lrar y\)
	is \(f\)-cartesian edge with \(f(x')=f(x)\) 
	then post-composition with \(e\) induces an equivalence between the the mapping space
	\(\Hom_{\E_{f(x)}}(x,x')\) and the fiber of \(f_*\) over \(f\).
	We note that this statement is local, \ie proving the claim for a given \(e\colon x' \lrar y\) only requires us to consider
	maps in \(\E\) lying over either \(f(e)\) or \(\Id_{f(x)}\). Thus, it is enough to prove the claim for the outer 
	cartesian fibration \(\E \times_\B \Del^1 \lrar \Del^1\) 
	obtained by pulling back \(f \colon \E \to \B\)
	along the map \(\Delta^1 \to \B\) corresponding to \(f(e)\). Since every triangle in \(\Del^1\) is degenerate 
	this pullback is a cartesian fibration of \(\infty\)-categories. 
	By possibly replacing \(\E \times_\B \Del^1\) 
	with an equivalent \(\infty\)-category, we may assume that \(\E \times_\B \Del^1\) is given by the nerve 
	of a map of fibrant simplicial categories \(\C \lrar [1]\), 
	and an application of~\cite[Proposition 2.4.1.10 (2)]{HTT} then finishes the proof.
	
	Finally, the proof in the case where \(f\) is an inner cartesian fibration proceeds verbatim, except that in the first part we do not need to consider the case \(i=0\), but do need to consider the case \(i=n\). In the latter case the left vertical map in~\eqref{e:rectangle} is isomorphic to the inner horn inclusion \(\Lam^{[n]*[0]}_n \hrar \Del^{[n]*[0]}\), and so the lift exists by the assumption that \(f\) is an inner fibration.
\end{proof}

\begin{cor}\label{c:fiberwise}
	Let 
	\[
	 \begin{tikzcd}
		\E \ar[r, "r"] \ar[d, "p"'] & \E' \ar[d, "q"] \\
		\B \ar[r, "f"'] & \B'
	 \end{tikzcd}
	\]
	be a diagram of \(\infty\)-bicategories such that \(p\) and \(q\) 
	are both outer (or both inner) cartesian fibrations and \(r\) maps \(p\)-cartesian edges to \(q\)-cartesian edges. 
	Assume that \(f\) is a bicategorical equivalence.
	Then the following statements are equivalent:
	\begin{enumerate}[leftmargin=*]
		\item
		For every \(x \in \B\) the map
		\( r_x\colon  \E_x \lrar \E'_{p(x)} \)
		is an equivalence of \(\infty\)-categories. 
		\item
		The map \(r\colon \E \lrar \E'\) is an equivalence of \(\infty\)-bi\-cat\-egories.
	\end{enumerate}
\end{cor}

\begin{proof}
	Assume (1) holds. It is clear that \(r \colon \E \lrar \E'\) is essentially surjective.
	Let us prove that the map \(r\) is also fully-faithful.
	For any pair of vertices \(x\) and \(y\) of \(\E\), 
	we wish to show that the map \(r_\ast \colon \Hom_\E(x, y) \to \Hom_{\E'}(rx, ry)\)
	of marked \(\infty\)-categories is an equivalence.
	Proposition~\ref{p:mapping-2} tells us that we have a commutative triangle
	\[
	 \begin{tikzcd}
	 	\Hom^{\triangleright}_\E(x, y) \ar[r, "r_\ast"] \ar[d, "p_\ast"'] &
	 	\Hom^{\triangleright}_{\E'}(rx, ry) \ar[d, "q_\ast"] \\
	 	\Hom^{\triangleright}_{\B'}(px, py) \ar[r,"f_\ast"] & \Hom^{\triangleright}_{\B'}(fpx, fpy) 
	 \end{tikzcd}
	\]
	of (marked) left fibrations 
	(or right fibrations in the inner case), in which the bottom horizontal map is an equivalence of \(\infty\)-categories by Remark~\ref{r:dwyer-kan}. It hence
	suffices to prove that \(r_\ast\) is an equivalence fiber-wise, that is, 	for any \(e \in \Hom^{\triangleright}_{\B}(px, py)\),
	the induced map on the fibers \((p_\ast)^{-1}(e) \to (q_\ast)^{-1}(f(e))\) is an equivalence.
	Choose a \(p\)-cartesian lift \(e' \colon x' \to y\) of \(e\). Notice that by assumption
	the edge \(r(e') \colon r(x') \to r(y)\) of \(\E'\) is \(q\)-cartesian.
	Using again Proposition~\ref{p:mapping-2}, we get a commutative square
	\[
	 \begin{tikzcd}
	 	(p_\ast)^{-1}(e) \ar[r]\ar[d, "\simeq"'] & (q_\ast)^{-1}(f(e)) \ar[d, "\simeq"] \\
	 	\Hom_{\E_{p(x)}}(x, x') \ar[r] & \Hom_{\E_{q(rx)}}(rx, rx')
	 \end{tikzcd}
	\]
	where the two vertical arrows are equivalences. Statement (1) now tells us that
	the bottom horizontal map is also an equivalence, and so the desired result follows from 2-out-of-3. 

	Let us now assume Statement~(2).
	We then obtain a commutative diagram of \(\infty\)-categories
	\begin{equation}\label{e:diag-oo} 
	 \begin{tikzcd}
		\E^{\thi} \ar[r, "r^{\thi}", "\simeq"']\ar[d, "p"'] & (\E')^{\thi} \ar[d, "q"] \\
		\B^{\thi} \ar[r, "f"', "\simeq"] & (\B')^{\thi}
	 \end{tikzcd}
	\end{equation}
	in which the horizontal maps are categorical equivalences (see Remark~\ref{r:underlying}) and the vertical maps are cartesian fibrations (see Remark~\ref{r:iso}). Now consider the extended diagram
	\begin{equation}\label{e:diag-oo-2} 
	 \begin{tikzcd}
		\E^{\thi} \ar[r] \ar[d, "f"'] & \B \times_{(\B')^{\thi}} (\E')^{\thi} \ar[d] \ar[r] & (\E')^{\thi} \ar[d, "q"] \\
		\B^{\thi} \ar[r, equal] & \B^{\thi} \ar[r, "\simeq", "f"'] & (\B')^{\thi}
	 \end{tikzcd}
	\end{equation}
	in which the right square is a pullback square and the composition of the two top horizontal maps is the equivalence \(r^{\thi}\). By~\cite[Corollary 3.3.1.4]{HTT} the right square in~\eqref{e:diag-oo-2} is a homotopy pullback square. 
	But the external rectangle is also homotopy cartesian (since it contains a parallel pair of equivalences) and so the map \(\E^{\thi} \lrar \B \times_{(\B')^{\thi}}(\E')^{\thi}\) is a categorical equivalence.
	By~\cite[Proposition 3.3.1.5]{HTT} we now get that the induced map \(\E_x \lrar \E'_{f(x)}\) is a categorical equivalence for every vertex~\(x\) of~\(\B\).
\end{proof}

\subsection{Cartesian edges}
\label{sec:car-edges}

The notion of a (co)cartesian edge admits equivalent descriptions in various contexts. To fully exploit this it will be convenient to introduce the following two variants:

\begin{define}
	Let \(f \colon X \lrar S\) be a weak fibration of scaled simplicial sets.
	 We will say that an edge \(e\colon x \lrar y\) in \(X\) is \ndef{weakly \(f\)-cartesian}
	  if the dotted lift exists in any square of the form
	\[
	 \begin{tikzcd}
		(\Lam^n_n,T_{|\Lam^n_n}) \ar[r, "\sigma"] \ar[d] & X \ar[d, "f"] \\
		(\Del^n,T) \ar[dotted, ur] \ar[r] & S
	 \end{tikzcd}
	\]
	with \(n \geq 2\) and \(\sig_{|\Del^{n-1,n}} = e\), where \(T\) is the collection of all triangles in \(\Del^n\) which are either degenerate 
	or contain the edge \(\Del^{\{n-1,n\}}\).   
	We will say that \(e\colon x \lrar y\) is \ndef{strongly \(f\)-cartesian} if the dotted lift exists in any square of the form
	\[
	 \begin{tikzcd}
		(\Lam^n_0)_{\flat} \ar[r, "\sigma"]\ar[d] & X \ar[d, "f"] \\
		\Del^n_{\flat} \ar[dotted, ur]\ar[r] & S
	 \end{tikzcd}
	\]
	with \(n \geq 2\) and \(\sig_{|\Del^{n-1,n}} = e\). 
	Similarly, we define the notions of weak and strong cocartesian edges by using the opposite scaled simplicial sets.
\end{define}

\begin{example}\label{ex:joyal}
Let \(\B\) be an \(\infty\)-bicategory and \(f \colon \E \lrar \B\) be a weak fibration. Then every equivalence in \(\E\) is both weakly \(f\)-cartesian and weakly \(f\)-cocartesian. This follows by applying~\cite[Lemma 5.2]{GagnaHarpazLanariEquiv} to \(f\) and \(f^{\op}\).
\end{example}

\begin{prop}
	\label{p:cart_1-simp_via_homs}
	Given a weak fibration of \(\infty\)-bicategories \(f\colon\E \rightarrow \B\) 
	and an edge \(e\colon y \rightarrow z \) in \(\E\), the following statements are equivalent:
	\begin{itemize}
		\item \(e\) is weakly \(f\)-cartesian.
		\item The following square is a homotopy pullback in the marked categorical model structure:
		\begin{equation}
		\label{4}
		\begin{tikzcd}[column sep=large]
		\Maptr_{\E}(x,y)\ar[d,"f_{x,y}"{swap}]\ar[r,"e\circ -"]&\Maptr_{\E}(x,z)\ar[d,"f_{x,z}"]\\
		\Maptr_{\B}(fx,fy)\ar[r,"f(e)\circ-"]&\Maptr_{\B}(fx,fz)
		\end{tikzcd}
		\end{equation}
	\end{itemize}
\end{prop}

For the proof of Proposition~\ref{p:cart_1-simp_via_homs} we introduce the following piece of notation:
\begin{notate}\label{not:slice-over-arrow}
	Given a scaled simplicial set \(X\) and an edge \(e\colon x \to y\) in \(X\), we will denote by \(\trbis{X}{e^{\sharp}}\in \Sms\) the result of the slice Construction~\ref{con:slice} applied to the marked-scaled simplicial set \((\Del^1)^{\sharp}\) and the map \(\Del^1 \to X\) determined by \(e\).
	Explicitly, the set of \(n\)-simplices in \(\trbis{X}{e^{\sharp}}\) is given by:
	\[
	 (\trbis{X}{e^{\sharp}})_n\stackrel{\text{def}}{=} \bigl\{\alpha\colon \prescript{\flat}{}{\Delta^n} \ast \prescript{\sharp}{}{\Delta^1} \to X \ \vert \ \alpha_{\vert \Delta^{\{n+1,n+2\}}}=e \bigr\},
	\] 
	the marked edges are those which factor through 
	\(\prescript{\sharp}{}{\Delta^1} \ast \prescript{\sharp}{}{\Delta^1}\)
	and the thin triangles are those which factor through 
	\((\Delta^2)^{\flat}_{\sharp} \ast \prescript{\sharp}{}{\Delta^1}\). We will write \(\ovl{X}_{/e^{\sharp}}\) for the underlying scaled simplicial set of \(X_{/e^{\sharp}}\)
\end{notate}

\begin{rem}
	We note that for an \(\infty\)-bicategory \(\E\), it follows from Corollary~\ref{p:slice-fibration} that the projection \(\ovl{\E}_{/e^{\sharp}} \to \E\) is an outer cartesian fibration such that every marked edge in \(\E_{/e^{\sharp}}\) is cartesian. On the other hand, from the dual version of~\cite[Lemma 2.17]{GagnaHarpazLanariEquiv} (see also Lemma~\ref{l:pushout-join} below) \(\ovl{\E}_{/e^{\sharp}} \to \ovl{\E}_{/x}\) is a trivial fibration.  
\end{rem}

\begin{lemma}\label{l:weak-criterion}
	Let \(f\colon X \to S\) be a weak fibration of scaled simplicial sets. An edge \(e\colon x \to y\) in \(X\) is weakly \(f\)-cartesian if and only if the map
\begin{equation}\label{e:weakly-car}
\trbis{\ovl{X}}{e^{\sharp}}\rightarrow 
	\trbis{\ovl{X}}{y}\times_{\trbis{\ovl{S}}{fy}} \trbis{\ovl{S}}{fe^{\sharp}}
\end{equation}
	is a trivial fibration of scaled simplicial sets.
\end{lemma} 

\begin{proof}
	For any \(n \geq 0\), we have the following correspondence of lifting problems:
	\[
		\begin{tikzcd}
			\prescript{\flat}{}{\partial \Delta^n} \ast \prescript{\sharp}{}{\Delta^1} \cup \prescript{\flat}{}{\Delta^n} \ast \{1\} \ar[r] \ar[d] &
			X \ar[d] \\
			\prescript{\flat}{}{\Delta^n} \ast \prescript{\sharp}{}{\Delta^1} \ar[r] \ar[ru, dotted] & S
		\end{tikzcd}
		\quad \leftrightsquigarrow \quad
		\begin{tikzcd}
			(\partial \Delta^n)_\flat \ar[r] \ar[d] & \trbis{\ovl{X}}{e^{\sharp}} \ar[d] \\
			\Delta^n_\flat \ar[r] \ar[ru, dotted] & \trbis{\ovl{X}}{y}\times_{\trbis{\ovl{S}}{py}} \trbis{\ovl{S}}{pe^{\sharp}} \ .
		\end{tikzcd}
	\]
	It is clear from the definition that the scaled simplicial set \(\prescript{\flat}{}{\Delta^n} \ast \prescript{\sharp}{}{\Delta^1}\)
	is isomorphic to \((\Delta^{n+2}, T)\), where \(T\) is the scaling of Definition~\ref{d:weak}, for any \(n \geq 0\).
	Identifying \(\Delta^n \ast \Delta^1\) with \(\Delta^{n+2}\), one immediately checks that \(\partial\Delta^n \ast \Delta^1\)
	is the subsimplicial set given by the union of \(\Lambda^{n+2}_{n+1}\) and \(\Lambda^{n+2}_{n+2}\)
	and that \(\Delta^n \ast \{1\}\) is the face \(\Delta^{\{0, 1, \dots, n, n+2\}}\). Hence, the left vertical map in the
	left square of the previous correspondence is isomorphic to the map \((\Lambda^{n+2}_{n+2}, T_{|\Lambda^{n+2}_{n+2}}) \hookrightarrow (\Delta^{n+2}, T)\)
	appearing in Definition~\ref{d:weak}, for any \(n \geq 0\). Finally, the lifting against the map \(\Delta^2_\flat \hookrightarrow \Delta^2_\sharp\) holds without any assumption on \(e\) since thin triangles on both sides of~\eqref{e:weakly-car} are detected by \(f\). 
\end{proof}

\begin{proof}[Proof of Proposition~\ref{p:cart_1-simp_via_homs}]
By Lemma~\ref{l:weak-criterion} we have that an edge \(e\colon y \to z\) is weakly \(f\)-cartesian if and only if 
the map  
\[\trbis{\ovl{\E}}{e^{\sharp}}\rightarrow
	\trbis{\ovl{\E}}{z}\times_{\trbis{\ovl{\B}}{fz}}\trbis{\ovl{\B}}{fe^{\sharp}}\] 
is a trivial fibration.
This can be viewed as a morphism between outer cartesian fibrations over \(\E\), thanks to Proposition~\ref{p:slice-fibration}.
Therefore, according to Corollary~\ref{c:fiberwise} it is an equivalence if and only if it induces an equivalence between the respective fibers
\[
 \trbis{\ovl{\E}}{e^{\sharp}}\times_{\E} \{x\} \to
 	\bigl(\trbis{\ovl{\E}}{z}\times_{\trbis{\ovl{\B}}{fz}}\trbis{\ovl{\B}}{fe^{\sharp}}\bigr)\times_{\E}\{x\} ,
\] 
for any vertex \(x\) of \(\E\).
At the same time, the square in~\eqref{4} can be modeled by the following commutative square of marked-scaled simplicial sets, after ignoring the (maximal) scaling everywhere: 
\[
 \begin{tikzcd}
	\trbis{\E}{e^{\sharp}}\times_{\E}\{x\}\ar[d]\ar[r]&\trbis{\E}{z}\times_{\E}\{x\}\ar[d]\\
	\trbis{\B}{fe^{\sharp}}\times_{\B}\{fx\} \ar[r]&\trbis{\B}{fz}\times_{\B}\{fx\} \ .
 \end{tikzcd}
\]
We note that the underlying marked simplicial sets in the above square are all fibrant (see~\cite[Remark 2.23]{GagnaHarpazLanariEquiv}). To finish the proof it will hence suffice to show that the map of marked simplicial sets underlying the bottom horizontal map is a fibration in the marked categorical model structure. Indeed, it follows from Proposition~\ref{p:slice-fibration-2} that the map in question satisfies the right lifting property with respect to all cartesian anodyne maps (in the sense of ~\cite[Definition 3.1.1.1]{HTT}). Since its domain and target are fibrant this map is a marked categorical fibration by~\cite[Lemma 4.40]{quillen}.
\end{proof}

Clearly any cartesian edge is weakly cartesian, and any strongly cartesian edge is cartesian.
The following proposition offers a partial inverse to this statement:

\begin{prop}\label{p:weak-to-strong}
	Let \(f\colon \E \to \B\) be a weak fibration of \(\infty\)-bicategories.
	Then an edge of \(\E\) is \(f\)-cartesian if and only if it is weakly \(f\)-cartesian. 
	If \(f\) is an outer fibration then an edge in \(\E\) is \(f\)-cartesian if and only if it is strongly \(f\)-cartesian. 
	In particular, in the latter case all three classes coincide.
\end{prop}

The proof of Proposition~\ref{p:weak-to-strong} will rely on the following lemmas, 
which will also give us useful 2-out-of-3 type properties for cartesian edges. To formulate it, let \(\B\) be an \(\infty\)-bicategory and let \(f \colon \mathcal \E \lrar \B\) be a weak fibration,
that we will assume to be an outer fibration when considering statements about
strong \(f\)-cartesian edges. 
Let \(\sig\colon \Del^2 \lrar \E\) be a thin triangle, depicted as a commutative diagram
\begin{equation}\label{e:technical}
\begin{tikzcd}[column sep = small]
& y_0 \ar[dr, "t"] & \\
x\ar[ur, "e_0"] \ar[rr, "e_1"']&& y_1\\
\end{tikzcd}.
\end{equation}

\begin{lemma}\label{l:technical-1}
If \(e_1\) is \(f\)-cartesian (resp.~strongly \(f\)-cartesian) and \(t\) is weakly \(f\)-cartesian then \(e_0\) is \(f\)-cartesian (resp.~strongly \(f\)-cartesian).
\end{lemma}

\begin{lemma}\label{l:technical-2}
If \(e_0\) is weakly \(f\)-cartesian and \(t\) is \(f\)-cartesian (resp.~strongly \(f\)-cartesian) then \(e_1\) is \(f\)-cartesian (resp.~strongly \(f\)-cartesian).
\end{lemma}

\begin{proof}[Proof of Lemma~\ref{l:technical-1}]
We will prove the \(f\)-cartesian case; the proof for strongly \(f\)-cartesian edges is similar (and easier). 
We need to show that the dotted lift exist in any square of the form
\begin{equation}\label{e:lift-f0}
\begin{tikzcd}
(\Lambda^n_n,\{\Delta^{\{0,n-1,n\}}\}_{|\Lambda^n_n}) \ar[r, "\sigma"]\ar[d] & \E \ar[d, "f"] \\
(\Delta^n,\{\Delta^{\{0,n-1,n\}}\}) \ar[ur, dotted]\ar[r,"\tau"] & \B
\end{tikzcd}
\end{equation}
with \(n \geq 2\) and \(\sigma_{|\Delta^{\{n-1,n\}}} = e_0\). Let us view \(\Lambda^n_n\) and \(\Del^n\) as the subsimplicial sets \(\partial \Del^{\{0,...,n-1\}} \ast \Del^{\{n\}}\) and
\(\Del^{\{0,...,n-1\}} \ast \Del^{\{n\}}\) of \(\Del^{n+1} = \Del^{\{0,...,n-1\}} \ast \Del^{\{n,n+1\}}\), and consider in addition the subsimplicial set 
\(Z := \partial \Del^{\{0,...,n-1\}} \ast \Del^{\{n,n+1\}} \subseteq \Del^{n+1}\). 
Let \(T \subseteq T' \subseteq \Del^{n+1}_2\) be the following two sets of triangles: \(T\) consists of the triangle \(\Del^{\{0,n-1,n\}}\) as well as all the triangles which contain the edge \(\Del^{\{n,n+1\}}\), while \(T' = T \cup \{\Del^{\{0,n-1,n+1\}}\}\). The maps \(\sig\) and the triangle~\eqref{e:technical} now fit to form a map of scaled simplicial sets
\begin{equation}\label{e:rho} 
(\Lambda^n_n,T_{|\Lambda^n_n}) \coprod_{\Del^{\{n-1,n\}}_{\flat}}\Del^{\{n-1,n,n+1\}}_{\sharp} \to \E .
\end{equation}
Applying (the dual version of) \cite[Lemma 2.17]{GagnaHarpazLanariEquiv} to the inclusions \(\Del^{\{n-1\}} \subseteq \partial \Del^{\{0,...,n-1\}}_{\flat}\) and \(\Del^{\{n\}} \subseteq {}^{\sharp}\Del^{\{n,n+1\}}\) we deduce that the map
\[ (\Lambda^n_n,\{\Del^{\{0,n-1,n\}}\}_{|\Lambda^n_n}) \coprod_{\Del^{\{n-1,n\}}_{\flat}}\Del^{\{n-1,n,n+1\}}_{\sharp} \hrar (Z,T_{|Z})\]
is scaled anodyne. 
We may hence extend the map~\eqref{e:rho} to a map \(\rho\colon (Z,T_{|Z}) \to \E\). The maps \(\tau\) and \(f\rho\) now combine to form a map
\[ \eta\colon (\Del^{n},\{\Del^{\{0,n-1,n\}}\})\coprod_{(\Lam^n_n,\{\Del^{\{0,n-1,n\}}\}_{|\Lam^n_n})} (Z,T_{|Z}) \to \B .\]
Applying again the dual of~\cite[Lemma 2.17]{GagnaHarpazLanariEquiv}, this time to the inclusions \(\partial\Del^{\{0,...,n-1\}}_{\flat} \subseteq \Del^{\{0,...,n-1\}}_{\flat}\) and \(\Del^{\{n\}} \subseteq {}^{\sharp}\Del^{\{n,n+1\}}\) we deduce that the map
\[ (\Del^{n},\{\Del^{\{0,n-1,n\}}\})\coprod_{(\Lam^n_n,\{\Del^{\{0,n-1,n\}}\}_{|\Lam^n_n})} (Z,T_{|Z}) \to (\Del^{n+1},T) \]
is scaled anodyne. We may hence extend the square~\eqref{e:lift-f0} to a square
\begin{equation}\label{e:lift-f0-2}
\begin{tikzcd}
(Z,T_{|Z}) \ar[r, "\rho"]\ar[d] & \E \ar[d, "f"] \\
(\Delta^{n+1},T) \ar[r,"\tau"] & \B . \
\end{tikzcd}
\end{equation}
We now observe that the 3-simplex \(\Del^{\{0,n-1,n,n+1\}} \subseteq \Del^{n+1}\) has the property that all its faces except \(\Del^{\{0,n-1,n+1\}}\) are contained in \(T\), while the face \(\Del^{\{0,n-1,n+1\}}\) is exactly the one triangle that is in \(T'\) but not in \(T\). We also note that \(Z\) contains \(\Del^{\{0,n-1,n,n+1\}}\) unless \(n=2\), in which case \(Z\) does not contain \(\Del^{\{0,n-1,n+1\}}\) either. Since \(\E\) and \(\B\) are \(\infty\)-bicategories we may now conclude from~\cite[Remark 3.1.4]{LurieGoodwillie} that the square~\eqref{e:lift-f0-2} extends to a square
\begin{equation}\label{e:lift-f0-3}
\begin{tikzcd}
(Z,T'_{|Z}) \ar[r, "\rho'"]\ar[d] & \E \ar[d, "f"] \\
(\Delta^{n+1},T') \ar[ur, dotted]\ar[r,"\tau'"] & \B . \
\end{tikzcd}
\end{equation}
To finish the proof we now produce a lift in~\eqref{e:lift-f0-3}. For this,
notice that
\[
	\Lam^{n+1}_{n+1} = Z \coprod_{\Lam^{\{0,...,n-1,n+1\}}_{n+1}}\Del^{\{0,...,n-1,n+1\}}
\]
so that we can write \(\Del^{n+1}\) as 
\[\Delta^{n+1} = Z \coprod_{\Lam^{\{0,...,n-1,n+1\}}_{n+1}}\Del^{\{0,...,n-1,n+1\}} \coprod_{\Lam^{n+1}_{n+1}}\Del^{n+1}.\] 
We can then construct a lift in two steps, the first by using the assumption that \(\rho'_{|\Del^{\{n-1,n+1\}}} = e_1\) is \(f\)-cartesian (and that \(T'\) contains \(\Del^{\{0,n-1,n+1\}}\)), and the second by using the assumption that \(\rho'_{|\Del^{\{n,n+1\}}} = t\) is weakly \(f\)-cartesian (and that \(T'\) contains all the triangles with the edge \(\Del^{\{n,n+1\}}\)). 
\end{proof}

\begin{proof}[Proof of Lemma~\ref{l:technical-2}]
The proof is very similar to that of Lemma~\ref{l:technical-1}, but we spell out the differences for the convenience of the reader.
As in the case of Lemma~\ref{l:technical-1}, we will prove only the \(f\)-cartesian case, as the proof for strongly \(f\)-cartesian edges proceeds verbatim, with just less details to verify.

Let \(I = \{0,...,n-1,n+1\} \subseteq [n+1]\), and let us identify \(\Lambda^n_n\) and \(\Del^n\) with the subsimplicial sets 
\[\Lam^I_{n+1} = \partial \Del^{\{0,...,n-2\}} \ast \Del^{\{n-1,n+1\}} \coprod_{\partial \Del^{\{0,...,n-2\}} \ast \Del^{\{n+1\}}} \Del^{\{0,...,n-2,n+1\}}\]
and \(\Del^I = \Del^{\{0,...,n-2\}} \ast \Del^{\{n-1,n+1\}}\) of \(\Del^{n+1}\), respectively. In addition, we consider the subsimplicial sets 
\[Z := \partial \Del^{\{0,...,n-2\}} \ast \Del^{\{n-1,n,n+1\}} \quad\text{and}\quad W := \partial \Del^{\{0,...,n-2\}} \ast \Del^{\{n-1,n+1\}} \subseteq Z.\]
Let \(T \subseteq T' \subseteq \Del^{n+1}_2\) be the following two sets of triangles: \(T\) consists of the triangle \(\Del^{\{0,n-1,n+1\}}\), as well as all the triangles which contain the edge \(\Del^{\{n-1,n\}}\), while \(T' = T \cup \{\Del^{\{0,n,n+1\}}\}\). 
We need to show that the dotted lift exist in any square of the form
\begin{equation}\label{e:lift-f1}
\begin{tikzcd}
(\Lambda^I_{n+1},\{\Delta^{\{0,n-1,n+1\}}\}_{|\Lambda^I_{n+1}}) \ar[r, "\sigma"] \ar[d] & \E \ar[d, "f"] \\
(\Delta^I,\{\Delta^{\{0,n-1,n+1\}}\}) \ar[ur, dotted]\ar[r,"\tau"] & \B
\end{tikzcd}
\end{equation}
with \(n \geq 2\) and \(\sigma_{|\Delta^{\{n-1,n+1\}}} = e_1\). Similarly to the proof of Lemma~\ref{l:technical-1}, the maps \(\sig\) and the triangle~\eqref{e:technical} fit to form a map of scaled simplicial sets
\begin{equation}\label{e:rho-2} 
(\Lambda^I_{n+1},T_{|\Lambda^I_{n+1}}) \coprod_{\Del^{\{n-1,n+1\}}_{\flat}}\Del^{\{n-1,n,n+1\}}_{\sharp} \to \E .
\end{equation}
We now apply \cite[Lemma 2.17]{GagnaHarpazLanariEquiv} to the map \(\emptyset \subseteq \partial \Del^{\{0,...,n-2\}}_{\flat}\) and the composed inclusion 
\[\Del^{\{n-1,n+1\}}_{\flat} \subseteq (\Lam^{\{n-1,n,n+1\}}_{n-1},\{\Del^{\{n-1,n\}}\},\emptyset) \subseteq (\Del^{\{n-1,n,n+1\}},\{\Del^{\{n-1,n\}}\},\{\Del^{\{n-1,n,n+1\}}\})\] 
to deduce that the map
\[ (W,T_{|W}) \coprod_{\Del^{\{n-1,n+1\}}_{\flat}}\Del^{\{n-1,n,n+1\}}_{\sharp} \hrar (Z,T_{|Z})\]
is scaled anodyne.
We may hence extend the map~\eqref{e:rho-2} to a map 
\[\rho\colon (\Lambda^I_{n+1},T_{|\Lambda^I_{n+1}}) \coprod_{(W,T_{|W})}(Z,T_{|Z}) \to \E.\] 
The maps \(\tau\) and \(f\rho\) now combine to form a map
\[ \eta\colon (\Del^I,\{\Del^{\{0,n-1,n+1\}}\})\coprod_{(W,T_{|W})} (Z,T_{|Z}) \to \B .\]
Applying again~\cite[Lemma 2.17]{GagnaHarpazLanariEquiv}, now to the inclusions \(\partial\Del^{\{0,...,n-2\}}_{\flat} \subseteq \Del^{\{0,...,n-2\}}_{\flat}\) and \(\Del^{\{n-1,n+1\}}_{\flat} \subseteq (\Del^{\{n-1,n,n+1\}},\{\Del^{\{n-1,n\}}\},\{\Del^{\{n-1,n,n+1\}}\})\),
we deduce that the map
\[ (\Del^I,\{\Del^{\{0,n-1,n+1\}}\})\coprod_{(W,T_{|W})} (Z,T_{|Z}) \to (\Del^{n+1},T) \]
is scaled anodyne.
Setting
\[
	(Y, T_{|Y}) =  (\Del^I,\{\Del^{\{0,n-1,n+1\}}\})\coprod_{(W,T_{|W})} (Z,T_{|Z}),
\]
we may hence extend the square~\eqref{e:lift-f1} to a square
\begin{equation}\label{e:lift-f1-2}
\begin{tikzcd}
(Y,T_{|Y}) \ar[r, "\rho"]\ar[d] & \E \ar[d, "f"] \\
(\Delta^{n+1},T) \ar[r,"\tau"] & \B . \
\end{tikzcd}
\end{equation}
We now observe that the 3-simplex \(\Del^{\{0,n-1,n,n+1\}} \subseteq \Del^{n+1}\) has the property that all its faces except \(\Del^{\{0,n,n+1\}}\) are contained in \(T\), while the face \(\Del^{\{0,n,n+1\}}\) is exactly the one triangle that is in \(T'\) but not in \(T\). As in the proof of Lemma~\ref{l:technical-1}, 
we deduce that the square~\eqref{e:lift-f1-2} extends to a square
\begin{equation}\label{e:lift-f1-3}
\begin{tikzcd}
(Y,T'_{|Y}) \ar[r, "\rho'"]\ar[d] & \E \ar[d, "f"] \\
(\Delta^{n+1},T') \ar[ur, dotted]\ar[r,"\tau'"] & \B . \
\end{tikzcd}
\end{equation}
To finish the proof we now produce a lift in~\eqref{e:lift-f1-3}. For this, we note that
\[
	\Lam^{n+1}_{n} = Y \coprod_{\Lam^{\{0,...,n-2,n,n+1\}}_{n+1}}\Del^{\{0,...,n-2,n,n+1\}} \coprod_{\Lam^{\{0,...,n\}}_{n}}\Del^{\{0,...,n\}}
\]
so that we can write \(\Del^{n+1}\) as 
\[\Delta^{n+1} = Z \coprod_{\Lam^{\{0,...,n-2,n,n+1\}}_{n+1}}\Del^{\{0,...,n-2,n,n+1\}} \coprod_{\Lam^{\{0,...,n\}}_{n}}\Del^{\{0,...,n\}} \coprod_{\Lam^{n+1}_{n}}\Del^{n+1}.\] 
We can then construct a lift in three steps, the first by using the assumption that \(\rho'_{|\Del^{\{n,n+1\}}} = t\) is \(f\)-cartesian (and that \(T'\) contains \(\Del^{\{0,n,n+1\}}\)), the second by using the assumption that \(\rho'_{|\Del^{\{n-1,n\}}} = e_0\) is weakly \(f\)-cartesian (and that \(T'\) contains all the triangles with the edge \(\Del^{\{n-1,n\}}\)), and the third by using the fact that \(f\) is a weak fibration and \(T'\) contains the triangle \(\Del^{\{n-1,n,n+1\}}\). 
\end{proof}

\begin{proof}[Proof of Proposition~\ref{p:weak-to-strong}]
	Apply Lemma~\ref{l:technical-2} to the degenerate triangle
	\[
	 \begin{tikzcd}
		& y_0 \ar[dr, "\Id_{y_0}"] & \\
		x \ar[ur, "e_0"] \ar[rr, "e_0"'] && y_0
	 \end{tikzcd} .\qedhere
	\]
\end{proof}

Combining Proposition~\ref{p:weak-to-strong} with Example~\ref{ex:joyal} we now conclude:

\begin{cor}\label{c:joyal}
Let \(f \colon \E \lrar \B\) be a weak fibration of \(\infty\)-bicategories. Then every equivalence in \(\E\) is both \(f\)-cartesian and \(f\)-cocartesian.
\end{cor}

\begin{cor}\label{c:2-out-of-3}
Let \(p\colon\E \to \B\) be a weak fibration of \(\infty\)-bicategories and let
\[
\begin{tikzcd}
& y \ar[dr, "t"] & \\
x \ar[ur, "e_0"] \ar[rr, "e_1"'] && z
\end{tikzcd} 
\]
be a thin triangle in \(\E\) such that \(t\) is \(p\)-cartesian (e.g., \(p\) is an equivalence, see Corollary~\ref{c:joyal}). Then \(e_0\) is \(p\)-cartesian if and only if \(e_1\) is \(p\)-cartesian.
\end{cor}

\begin{rem}\label{r:car-invariant}
It follows from Corollary~\ref{c:2-out-of-3} that the property of being a cartesian edge is closed under equivalence. More precisely, suppose that \(p\colon \E \to \B\) is a weak fibration and \(e\colon x \to y,e'\colon x'\to y'\) two edges which are equivalent in the \(\infty\)-bicategory \(\Fun(\Del^1,\E)\). We claim that \(e\) is \(p\)-cartesian if and only if \(e'\) is \(p\)-cartesian. Indeed, suppose that \(e'\) is \(p\)-cartesian and let \(\eta\colon \Del^1 \times \Del^1 \to \E\) encode
\[
 \begin{tikzcd}
 	x \ar[r, "\simeq"] \ar[d, "e"'] \ar[rd] & {x'} \ar[d, "{e'}"] \\
 	y \ar[r, "\simeq"] & {y'}
 \end{tikzcd}
\]
an equivalence from \(e=\eta_{|\Del^{\{0\}} \times \Del^1}\) to \(e'=\eta_{|\Del^{\{1\}} \times \Del^1}\). Applying Corollary~\ref{c:2-out-of-3} once we get that the diagonal arrow \(x \to y'\) is \(p\)-cartesian, and applying it a second time gives that \(e\colon x \to y\) is \(p\)-cartesian. Taking the equivalence in the other direction we can similarly deduce that if \(e\) is \(p\)-cartesian then \(e'\) is \(p\)-cartesian.
\end{rem}

\subsection{Marked fibrations and anodyne maps}
\label{s:slice}

As we have alluded to in the beginning of \S\ref{s:fibrations}, the notion of inner cocartesian fibration over a base \((S,T_S)\) 
is essentially a reformulation of the notion of 
\(\mathfrak{P}_{S}\)-fibered objects studied in~\cite[\S 3.2]{LurieGoodwillie}, see Definition~\ref{d:fibered} in \S\ref{s:straightening}. In the present subsection we will make this connection precise, and will then give a similar characterization for outer (co)cartesian fibrations.

\begin{prop}\label{p:inner-fibred}
	Let \((S,T_S)\) be a scaled simplicial set and let \(f: X \to S\) an inner fibration of simplicial sets. Let \(T_X = f^{-1}(T_S)\) to be the set of all triangles in \(X\) whose image in \(S\) belongs to \(T_S\) and let \(E_X\) be the set of edges in \(X\) which are locally \(f\)-cocartesian. Then the following are equivalence:
	\begin{enumerate}[leftmargin=*]
		\item
		\(f\colon (X,T_X) \to (S,T_S)\) is an inner cocartesian fibration.
		\item
		\(f\) exhibits \((X,E_X)\) as \(\Beta_{S}\)-fibered.
	\end{enumerate}
In addition, when these two equivalent conditions hold the set \(E_X\) identifies with the set of \(f\)-cocartesian arrows in \((X,T_X)\) (in the sense of Definition~\ref{d:cartesian}).
\end{prop}

\begin{rem}\label{rem:weaker}
In the situation of Proposition~\ref{p:inner-fibred}, if \((X,E)\) is \(\Beta_S\)-fibered for some set of marked edges \(E\), then \(E\) is necessarily the set of locally \(f\)-cocartesian edges by Condition~\hyperref[it:locally]{(ii)} above. In particular, the condition that \((X,E_X)\) is \(\Beta_S\)-fibered could be replaced by the a-priori weaker condition that \((X,E)\) is \(\Beta_S\)-fibered for \emph{some} set of marked edges \(E\).
\end{rem}

\begin{proof}[Proof of Proposition~\ref{p:inner-fibred}]
	Suppose (1) holds. We need to verify conditions (i)-(iii) of Definition~\ref{d:fibered}. Condition (i) holds since \(f\) is an inner fibration of scaled simplicial sets. To see (ii), note that for every edge \(e\colon \Del^1 \to S\) the base change \(e^*f\colon (X,T_X) \times_{(S,T_S)} \Del^1_{\flat} \to \Del^1_{\flat}\)
	 is an inner cocartesian fibration. Since all the triangles in \(\Del^1_{\flat}\) are thin we get that the same holds for the domain of \(e^*f\) and so \(e^*f\) can be viewed as a cocartesian fibration of \(\infty\)-categories. The \(e^*f\)-cocartesian edges are then by definition the locally cocartesian edges lying over \(e\), and they are also by definition the marked edges lying over \(e\). This shows (ii).
	 
	 To prove (iii), let now \(\sig\colon \Del^2 \to S\) be a thin triangle and let \(e\colon \Del^1 \to X\) be a marked edge lying over
	 \(\ovl{e} := \sig_{|\Del^{\{0,1\}}}\), with domain \(x := e_{|\Del^{\{0\}}}\) and codomain \(y := e_{\Del^{\{1\}}}\). 
	 We wish to show that \(e\) is \(\sig^*f\)-cocartesian. For this, we note that \(\sig^*f\colon (X,T_X) \times_{(S,T_S)}\Del^2_{\sharp} \to \Del^2_{\sharp}\) 
	 is an inner cocartesian fibration, and hence a cocartesian fibration on the level of the underlying simplicial sets, since all the triangles in \(\Del^2_{\sharp}\) are thin. 
	 We may then conclude that there exists a \(\sig^*f\)-cocartesian edge \(e'\colon \Del^1 \to X\) lying over \(\ovl{e}\) 
	 such that \(e'_{|\Del^{\{0\}}} = x\). Then \(e\) and \(e'\) both determine cocartesian edges of 
	 \(X \times_S \Del^1 \to \Del^1\) with the same domain \(x\), and hence there exists a commutative diagram in \(X \times_S \Del^1\) of the form
	\[
	 \begin{tikzcd}
		& y \ar[dr, "u"] & \\
		x \ar[ur, "e"]\ar[rr, "e'"] && y
	 \end{tikzcd}
	\]
	where \(u\) is an equivalence which covers the identity \(\Id_{f(y)}\). 
	Since \(e'\) and \(u\) are \(\sig^*f\)-cocartesian it follows that \(e\) is \(\sig^*f\)-cocartesian, as desired.
	
	Let us now assume that (2) holds. We first show that \(f\colon(X,T_X) \to (S,T_S)\) is a weak fibration. 
	Since \(f\) is already assumed to be an inner fibration and we also assume that
	\(T_X = f^{-1}(T_S)\), it will suffice to show that 
	\(f\colon(X,T_X) \to (S,T_S)\) has the right lifting property with respect to \(f\colon (A,T_A) \to (B,T_B)\), 
	where \(f\) is one of the maps appearing in (2) and (3) of Definition~\ref{d:weak}. Since \(T_X = f^{-1}(T_S)\) it will suffice to check the lifting problem of \(X \to S\) against \(A \to B\) on the level of underlying simplicial sets.
Now by assumption the map 
	\(f\colon X \to S\) exhibits \((X,E_X)\) as \(\Beta_S\)-fibered. 
	In particular, the object \((X,E_X)\) of \((\s^+)_{/S}\) is fibrant with respect to the \(\Beta_S\)-fibered model structure. 
	In order to prove the lifting property against \(A \to B\) 
	it will consequently suffice to prove that the induced map 
	\(A^{\flat} \to B^{\flat}\), when considered as a map in \((\s^+)_{/S}\), 
	is a \(\Beta_S\)-fibered weak equivalence. Consider the straightening functor 
	\[
	 \Str^{\sca}_{(S,T_S)}\colon(\s^+)_{/S} \to (\s^+)^{\fC(S,T_S)}
	\]
	associated to the identity \(\id\colon \fCs(S,T_S) \to \fCs(S,T_S)\). Since \(\Str^{\sca}_{\id}\)
	is a left Quillen equivalence and all the objects of \((\s^+)_{/S}\) 
	are cofibrant it is enough to show that 
	\[\Str^{\sca}_{(S,T_S)}(A^{\flat}) \to \Str^{\sca}_{(S,T_S)}(B^{\flat})\]
	is a weak equivalence in the projective model structure on \((\s^+)^{\fC(S,T_S)}\).
Unwinding the definitions, let \(Z = \Del^1 \times B\) and
let \(T\) denote the set of those triangles \((\tau,\sig_B)\colon \Del^2 \to \Del^1 \times B\) such that \(\sig_B\) is degenerate and either \(\tau_{|\Del^{\{1,2\}}}\) is degenerate in \(\Del^1\) or \({\sig_B}_{|\Del^{\{0,1\}}}\) is degenerate in \(B\),	
	together with the triangles in \(T_B \times \Del^{\{0\}}\) and \(T_B \times \Del^{\{1\}}\). 
	Let
	\[
	 Z_0 = (\Del^1 \times A) \coprod_{\partial \Del^1 \times A} (\partial \Del^1 \times B) \subseteq Z
	\]
	and let
	\(T_0\) be the collection of \(2\)-simplices of \(Z_0\) whose image in \(Z\) belongs to \(T\). Consider the commutative rectangle
	\[
	 \begin{tikzcd}
		\partial \Del^1 \times (B,T_B)\ar[d] \ar[r] & (Z_0,T_0) \ar[r]\ar[d] & (Z,T)  \ar[d] \\
		\Del^{\{0\}} \coprod [\Del^{\{1\}} \times (S,T_S)] \ar[r] & C_A \ar[r] & C_B
	 \end{tikzcd}
	\]
	in which \(C_A\) and \(C_B\) are defined by the condition that the left square and the external rectangle are pushout squares. By the definition of \(\St^{\sca}_{\id}\) (recalled in \S\ref{s:straightening}) it will suffice to show that \(C_A \to C_B\) is a bicategorical weak equivalence. 
	By the pasting lemma the right square is a pushout square as well,
	and so it suffices to show that the top horizontal map in the right square is scaled anodyne.
	Inspecting the set of thin triangles \(T\) we observe that we have a commutative diagram
	\[
	 \begin{tikzcd}
		(Z_0,T_0) \ar[r]\ar[d, "\simeq"'] & (Z,T) \ar[d, "\simeq"] \\
		{\displaystyle\bigl(\Del^1_{\flat} \otimes (A,T_A)\bigr) \mathop{\coprod}_{\partial \Del^1_{\flat} \otimes (A,T_A)} \bigl(\partial \Del^1_{\flat} \otimes (B,T_B)\bigr)} \ar[r] & \Del^1_{\flat} \otimes (B,T_B)
	 \end{tikzcd}
	\]
	where \(\otimes\) denotes the Gray product of scaled simplicial sets, see \S\ref{s:scaled-gray}  
	and where the vertical arrows are scaled anodyne by~\cite[Proposition~2.8]{GagnaHarpazLanariGrayLaxFunctors}. 
	To finish the proof we now just need to show that the lower horizontal map in the last square is a bicategorical equivalence.
	Now in the case where \(f\colon (A,T_A) \to (B,T_B)\) is as in (2) of Definition~\ref{d:weak} 
	this follows from \cite[Proposition~2.16]{GagnaHarpazLanariGrayLaxFunctors} since \(f\) is scaled anodyne. 
	In the case where \(f\colon (A,T_A) \to (B,T_B)\) is as in (3) of Definition~\ref{d:weak} 
	we have that \(f^{\op}\colon(A^{\op},T_A) \to (B^{\op},T_B)\) is scaled anodyne and 
	hence the lower horizontal map is the opposite of a scaled anodyne map, 
	then in particular a bicategorical equivalence, by~\cite[Remark~2.4 and Proposition~2.16]{GagnaHarpazLanariGrayLaxFunctors}.
	
	We thus proved that \(f\colon(X,T_X) \to (S,T_S)\) is a weak fibration. 
	To finish the proof we need to show that every arrow \(f\colon x \to y\) in \(S\) admits
	\(f\)-cocartesian lifts starting from any object \(x'\in X\) lying over \(x\). 
	For this we invoke~\cite[Proposition 3.2.16]{LurieGoodwillie} which implies that the object 
	\((X,E_X)\) of \((\s^+)_{/S}\) satisfies the right lifting property with respect to the 
	\(\Beta_{S}\)-anodyne maps listed in~\cite[Definition 3.2.10]{LurieGoodwillie}. 
	In particular, the right lifting property with respect to the maps of type \((C_0)\) 
	of this list implies that every edge in \(E_X\) is \(f\)-cocartesian, and the right lifting property 
	with respect to the maps of type \((B_0)\) implies that every arrow \(f: x \to y\) in \(S\) admits 
	a marked lift starting from any vertex \(x'\in X\) lying over \(x\). 
	We may hence conclude that \((X,T_X) \to (S,T_S)\) is an inner cocartesian fibration, as desired. To obtain the last statement, note that the last argument shows that when (2) holds every marked edge is \(f\)-cocartesian, while every \(f\)-cocartesian edge is in particular locally cocartesian, and hence marked by the definition of \(E_X\).
\end{proof}

One of the advantages of describing inner cocartesian fibrations in terms of \(\Beta_S\)-fibered object is that the latter can be characterized using a right lifting property with respect to a suitable collection of anodyne maps, see~\cite[Proposition 3.2.16]{LurieGoodwillie}. Our next goal is to show that a similar statement holds in the case of outer (co)cartesian fibrations. A step in that direction was already taken in~\cite{GagnaHarpazLanariEquiv} using the collection of maps given in Definition 2.13 of \loccit, but that collection was only partial. In what follows we identify the complete list of outer cartesian anodyne maps and show that the lifting property against them characterizes outer cartesian fibrations. 

\begin{notate}
Given a scaled simplicial set \((S,T_S)\), we write \(\trbis{(\Sms)}{(S,T_S)}\)
to denote the category of marked scaled simplicial sets \((X,E_X,T_X)\) equipped with a map of scaled simplicial sets \((X,T_X) \to (S,T_S)\).
\end{notate}

\begin{define}\label{d:anodyne-marked}
	Let \((S, T_S)\) be a scaled simplicial set. We call \ndef{outer cartesian anodyne maps} 
	the smallest weakly saturated class of maps in \(\trbis{(\Sms)}{(S,T_S)}\) containing the following maps:
	\begin{enumerate}[leftmargin=*]
		\item\label{item:outer_inner_horn}
		The inclusion 
		 \[
		  (\Lam^n_i,\varnothing,\{\Del^{\{i-1,i,i+1\}}\}_{|\Lam^n_i }) \hrar (\Del^n,\varnothing,\{\Del^{\{i-1,i,i+1\}}\})
		 \]
		 for every \(0 < i < n\) and every map \((\Del^n,\{\Del^{\{i-1,i,i+1\}}\}) \lrar (S,T_S)\).
		\item\label{item:outer_cartesian_lifts}
		The inclusion 
		\[
		 (\Lam^n_n,\{\Del^{\{n-1,n\}}\},\varnothing) \subseteq (\Del^n,\{\Del^{\{n-1,n\}}\},\varnothing)
		\]
		for every \(n \geq 1\) and every map \(\Del^n \lrar S\)
		(when \(n=1\) this should be read as the inclusion \(\Del^{\{1\}} \subseteq (\Del^1)^{\sharp}\)).
		\item\label{item:outer_cartesian_triangles}
		The inclusion 
		\[
		 (\Lam^n_0 \coprod_{\Del^{\{0,1\}}}\Del^0,\varnothing,\varnothing) \subseteq (\Del^n \coprod_{\Del^{\{0,1\}}}\Del^0,\varnothing,\varnothing)
		\]
		for every \(n \geq 2\) and every map \((\Del^n \coprod_{\Del^{\{0,1\}}}\Del^0,\varnothing) \lrar (S,T_S)\).
		\item
		The inclusion \({}^\flat\Del^2 \subseteq (\Del^2,\emptyset,\{\Del^2\})\) for every map \(\Delta^2_\sharp \to (S,T_S)\).
		\item
		The inclusion \((Q,\emptyset,Q_2) \subseteq (Q,E,Q_2)\) for every map \(Q_{\sharp} \lrar (S,T_S)\), where 
		\[
		 Q = \Del^0 \coprod_{\Del^{\{0,2\}}} \Del^3 \coprod_{\Del^{\{1,3\}}} \Del^0
		\]
		and \(E\) contains all the degenerate edges and in addition the edges \(\Del^{\{0,1\}}\) and \(\Del^{\{0,3\}}\). 
		\item\label{item:outer_thin}
		The inclusion 
		\[
		 (\Del^2,\{\Del^{\{0,1\}},\Del^{\{1,2\}}\},\{\Del^2\}) \subseteq {}^{\sharp}\Del^2
		\]
		for every map \(\Del^2_{\sharp} \lrar (S,T_S)\). 
	\end{enumerate}
	Dually, we let the collection of \ndef{outer cocartesian anodyne maps} to be the weakly saturated class generated by the opposites of the above maps.
\end{define}

The following proposition extends~\cite[Proposition 2.14]{GagnaHarpazLanariEquiv}:

\begin{prop}\label{p:char}
	Let \(\B\) be an \(\infty\)-bicategory, 
	\((X,E_X)\) a marked simplicial set and \(f\colon (X,T_X) \lrar \B\) a map which detects thin triangles.  
The object of \(\trbis{(\Sms)}{\B}\) determined by \((X,E_X,T_X)\) and \(f\)
	has the right lifting property with respect to outer cartesian anodyne maps
	if and only if \(f\colon (X,T_X) \to \B\) is an outer cartesian fibration and \(E_X\) is the collection of \(f\)-cartesian edges of \(X\).
\end{prop}
\begin{proof}
We first prove the ``only if'' direction.
	Since every degenerate edge in \(X\) belongs to \(E_X\) the right lifting property 
	with respect to outer cartesian anodyne maps of type (1),(2), (3) and (4) implies that \(f\) is an outer fibration and that every edge in \(E_X\) is \(f\)-cartesian. 
	In addition, the case \(n=1\) of maps of type (2) implies that for every arrow \(e\colon x \lrar y\) in \(\B\) 
	and for every \(\wtl{y}\in X\) such that \(f(\wtl{y}) = y\) there exists a marked (and hence \(f\)-cartesian) edge \(\wtl{e}\colon \wtl{x} \lrar \wtl{y}\) in \(X\) 
	such that \(f(\wtl{e}) = e\). We may hence conclude that \(f \colon (X,T_X) \lrar \B\) is an outer cartesian fibration
	with all marked edges being \(f\)-cartesian. Let us now show that every \(f\)-cartesian edge is marked. Let \(e\colon x \lrar y\) be a \(f\)-cartesian edge lying over an edge \(\ovl{e}\colon \ovl{x} \lrar \ovl{y}\) of \(\B\). Then there exists a marked edge \(e'\colon x' \lrar y\) in \(X\) such that \(f(e') = \ovl{e}\). By the above \(e'\) is \(f\)-cartesian, and so we may factor \(e\) through \(e'\), in the sense that we may find a thin triangle in \(X\) of the form
	\[ \xymatrix{
		& x' \ar^{e'}[dr] & \\
		x \ar^{u}[ur] \ar^{e}[rr]&& y\\
	}\]
	which lies over the degenerate triangle
	\[ \xymatrix{
		& \ovl{x} \ar^{\ovl{e}}[dr] & \\
		\ovl{x} \ar^{\Id}[ur] \ar^{\ovl{e}}[rr]&& \ovl{y}\\
	}\]
	By Corollary~\ref{c:2-out-of-3} we may conclude that \(u\) is \(f\)-cartesian, hence an equivalence by Remark~\ref{r:cart-equiv}. We now observe that the right lifting property with respect to outer cartesian anodyne maps of type (5) implies in particular that every equivalence in \(X\) is marked, and so in particular \(u\) is marked. Finally, since \(u\) and \(e'\) are marked an application of the right lifting property against maps of type (6) implies that \(e\) is marked, as desired.
	
We now prove the ``if'' direction, and so we assume that \(f\) is an outer cartesian fibration and \(E_X\) consists of the \(f\)-cartesian edges.
	Proposition~\ref{p:weak-to-strong} implies that every \(f\)-cartesian edge in \(X\) is strongly \(f\)-cartesian. 
	This together with the fact that \(f\) is an outer fibration implies that \((X,E_X,T_X)\) 
	has the right lifting property with respect to outer cartesian anodyne maps of type (1), (2), (3) and (4). 
	The right lifting property with respect to maps of type (6) follows directly from Corollary~\ref{c:2-out-of-3}. 
	In order to conclude the proof we wish to show that the map \(f\) has the right lifting property with respect to maps of type (5). 
	For this we note that \((X,T_X)\) is an \(\infty\)-bicategory in this case by Remark~\ref{r:weak-scaled}, and that any edge in \(Q\) is necessarily sent to an equivalence in \((X,T_X)\), which is therefore \(f\)-cartesian by virtue of Corollary~\ref{c:joyal}.
\end{proof}

We finish this subsection by discussing the compatibility of outer (co)cartesian anodyne maps with \(\ast\)-pushout-products. In particular, the following lemma extends~\cite[Lemma 2.17]{GagnaHarpazLanariEquiv}:

\begin{lemma}\label{l:pushout-join}
	Let \(f\colon X \lrar Y\) and \(g\colon A \lrar B\) be injective maps of marked-scaled simplicial sets. If either \(f\) is outer cartesian anodyne or \(g\) is outer cocartesian anodyne then the map of scaled simplicial sets
	\begin{equation}\label{e:pushout-join} 
	X \ast B \coprod_{X \ast A} Y \ast A \lrar Y \ast B 
	\end{equation}
	is a bicategorical trivial cofibration.
\end{lemma}
\begin{proof}
	Since the collection of trivial cofibrations is closed under taking opposites, it will suffice to verify the case where \(f\) is outer cartesian anodyne. For this, one may check the claim on generators, and so we may assume that \(g\) is either the inclusion \({}^{\flat}(\partial \Del^n) \hrar {}^{\flat}\Del^n\), the inclusion \({}^{\flat}\Del^2 \hrar (\Del^2,\emptyset,\{\Del^2\})\), or the inclusion \({}^{\flat}\Del^1 \hrar {}^{\sharp}\Del^1\), and \(f\) is one of the generating anodyne maps appearing in Definition~\ref{d:anodyne-marked}. We first note that when \(g\) is the map \({}^{\flat}\Del^2 \hrar (\Del^2,\emptyset,\{\Del^2\})\) then~\eqref{e:pushout-join} is an isomorphism. When \(g\) is the map \({}^{\flat}\Del^1 \hrar {}^{\sharp}\Del^1\) the map~\eqref{e:pushout-join} is an isomorphism except if \(f\) is \(\Del^{\{1\}} \hrar (\Del^1)^{\sharp}\), in which case the map~\eqref{e:pushout-join} takes the form 
	\[
	 (\Del^3,\{\Del^{\{1,2,3\}},\Del^{\{0,1,3\}},\Del^{\{0,1,2\}}\}) \lrar \Del^3_{\sharp} ,
	\]
	which is scaled anodyne by~\cite[Remark 3.1.4]{LurieGoodwillie}.
	
	We may hence assume that \(g\) is the inclusion \({}^{\flat}(\partial \Del^n) \hrar {}^{\flat}\Del^n\). For the first four types of generating anodyne maps this was proven in~\cite[Lemma 2.17]{GagnaHarpazLanariEquiv}. We now verify the remaining two cases:
	
	\begin{itemize}[leftmargin=*]
		\item
		When \(f\) is the inclusion \((\Del^2,\{\Del^{\{0,1\}},\Del^{\{1,2\}}\},\{\Del^2\}) \subseteq {}^{\sharp}\Del^2\) the map~\eqref{e:pushout-join} is isomorphic to the map
		\[
		 (\Del^{[2]\ast[n]},T) \lrar (\Del^{[2]\ast[n]},T') ,
		\]
		where \(T\) contains 
		all the triangles 
		of the form \(\Del^{\{0,1,i\}}\) and \(\Del^{\{1,2,i\}}\) while 
		\(T'\) contains \(T\) plus all the triangles of the form \(\Del^{\{0,2,i\}}\). 
		In this case we see that \((\Del^{[2]\ast[n]},T')\) can be obtained from 
		\((\Del^{[2]\ast[n]},T)\) by performing a sequence of pushouts along the maps
		\[
		 (\Del^3,\{\Del^{\{0,1,2\}},\Del^{\{0,1,3\}},\Del^{\{1,2,3\}}\},\Del^3_2) \lrar {}^{\sharp}\Del^3 ,
		\]
		which are scaled anodyne (see~\cite[Remark 3.1.4]{LurieGoodwillie}).
		\item
		When \(f\) is the inclusion \((Q,\emptyset,Q_2) \subseteq (Q,\{\Del^{\{0,1\}},\Del^{\{0,3\}}\},Q_2)\), 
		we set 
		\[W = \Del^{[0] \ast [n]} \coprod_{\Del^{\{0,2\}\ast [n]}}\Del^{[3]\ast[n]} \coprod_{\Del^{\{1,3\}\ast [n]}} \Del^{[0]\ast [n]}.\]
		The map~\eqref{e:pushout-join} is then  isomorphic to the map
		\(\left(W,T\right) \lrar \left(W,T'\right)\)
		where \(T\) contains 
		all the triangles which are contained in \(\Del^{\{0,1,2,3\}}\) and \(T'\) contains \(T\) and moreover
		all the triangles of the form 
		\(\Del^{\{0,1,i\}}\) and \(\Del^{\{0,3,i\}}\). In this case the map~\eqref{e:pushout-join} can be realized as a sequence of pushouts along the scaled anodyne maps
		\[ (\Del^4,\Del^4_1,T) \lrar (\Del^4,\Del^4_1,T \cup \{\Del^{\{0,3,4\}},\Del^{\{0,1,4\}}\}) \]
		as in Definition~\ref{d:anodyne}(ii).
	\end{itemize}

\end{proof}

\begin{cor}\label{cor:slice}
	Let \(\C\) be an \(\infty\)-bicategory and let \(f \colon K \lrar \C\) be a map of scaled simplicial sets. 
	Then the map of scaled simplicial sets 
	\[f \colon \ovl{\C}_{/f} \lrar \C\]
	is an outer cartesian fibration and the marked edges of \(\C_{/f}\) are exactly
	the \(f\)-cartesian edges. In particular, an edge \(e \colon \Del^1 \ast K \lrar \C\) in \(\ovl{\C}_{/f}\) 
	is \(f\)-cartesian if and only if for every vertex \(x\) of \(K\) the triangle \(e_{|\Del^1 \ast \{x\}}\) is thin. Dually, the map 
	\[f \colon \ovl{\C}_{f/} \to \C\] is an outer cocartesian fibration
	and the marked edges of \(\C_{f/}\) are precisely the \(f\)-cocartesian edges. 
\end{cor}

\begin{rem}
	The previous result also appears in \cite[\href{https://kerodon.net/tag/01WT}{Tag 01WT}]{LurieKerodon}, in a weaker form that only deals with the ``outer fibration'' part.
\end{rem}

\subsection{Cartesian lifts of lax transformations}
\label{s:lift}

In this section we relate the theory of outer cartesian fibrations as developed
so far in this work with the notion of lax transformations defined via the Gray product, see~\S\ref{s:scaled-gray}.
In particular, we prove the following:

\begin{prop}[Lifting lax transformations]\label{p:lax-lift}
	Let \(f\colon \E \rightarrow \B\) be an outer fibration of \(\infty\)-bicategories and \(K \subseteq L\) an inclusion of scaled simplicial sets. Consider a lifting problem of the form
\begin{equation}\label{e:lax-lift} 
\xymatrix{
\Del^{\{1\}} \otimes L \displaystyle\mathop{\coprod}_{\Del^{\{1\}} \otimes K} \Del^1_{\flat} \otimes K  \ar[r]^-{f}\ar[d] & \E \ar[d]^{f} \\
\Del^1_{\flat} \otimes L \ar[r]^{H}\ar@{-->}^{\wtl{H}}[ur] & \B \\
}
\end{equation}
	such that \(f\) sends every edge of the form \(\Del^1 \times \{v\}\) (for \(v \in K\)) to a \(f\)-cartesian edge. 	
	Suppose that for every \(u \in L_0 \setminus K_0\) there exists a \(f\)-cartesian edge with target \(f(\Del^{\{1\}} \times \{u\})\) which lifts \(H(\Del^1 \times \{u\})\). 
Then the dotted lift \(\wtl{H}\colon \Del^1_{\flat} \otimes B \rightarrow \E\) exists. Furthermore, \(\wtl{H}\) can be chosen so that the edges \(\wtl{H}(\Del^1 \times \{u\})\) for \(u \in L_0 \setminus K_0\) are any prescribed collection of \(f\)-cartesian lifts. 
\end{prop}

\begin{proof}
	Since \(f\colon \E \to \B\) is an outer fibration it detects thin triangles, and so a dotted lift in~\eqref{e:lax-lift} with the desired properties exists if and only if it exists on the level of underlying simplicial sets. We may hence assume without loss of generality that the \(L\) and \(K\) have 
	only degenerate triangles thin. 
	Arguing simplex by simplex it will suffice to prove the claim for \(L \subseteq K\) being the inclusion \(\partial \Del^n_{\flat}\subseteq \Del^n_{\flat}\).	
	In the case \(n=0\) the claim is tautological, since we assume the existence of cartesian lifts. 
	In the case \(n \geq 1\) the map \(\partial \Del^n_{\flat} \subseteq \Del^n_{\flat}\) is bijective on vertices and so we just need to construct a lift without the additional constraints on the edges. 
Consider the filtration
\[\Delta^1_{\flat} \mgr \partial \Del^{n}_{\flat}  \coprod_{\Del^{\{1\}}
\mgr \partial \Del^{n}_{\flat}} \Del^{\{1\}} \mgr \Del^n_{\flat} = 
X_0 \subseteq X_1 \subseteq \dots \subseteq X_{n+1} = \Del^1_{\flat} \mgr \Del^n_{\flat} ,\]
where 
\(X_{i+1}\) is the union of \(X_i\) and the image of the map
\(\tau_i\colon (\Del^{n+1},T^+_i) \to \Del^1_{\flat} \mgr \Del^{n}_{\flat}\) given on vertices by the formula
\[
 \tau_i(m) = \left\{\begin{matrix} (0,m) & m \leq i \\ (1,m-1) & m > i \end{matrix}\right. \ ,
\]
and \(T^+_i\) is the collection of all triangles in \(\Del^{n+1}\) which are either degenerate or of the form \(\Del^{\{i,i+1,k\}}\) for \(k > i+1\). We then observe that for \(i=0,...,n-1\) the inclusion
\(X_i \subseteq X_{i+1}\) is a pushout along \(\tau_i\) of the scaled inner horn \((\Lam^{n+1}_{i+1},{T^+_i}_{|\Lam^{n+1}_{i+1}}) \hrar (\Del^{n+1},T^+_i)\), 
while at the last step of the filtration the inclusion \(X_n \subseteq X_{n+1}\) is a pushout along \(\tau_n\) of the outer horn \((\Lam^{n+1}_{n+1})_{\flat} \subseteq \Del^{n+1}_{\flat}\), but since \(\tau_n\) sends \(\Del^{\{n,n-1\}}\) to \(\Del^1 \times \{n\}\) its image in \(\E\) is \(f\)-cartesian by assumption, and hence strongly \(f\)-cartesian by Proposition~\ref{p:weak-to-strong}. Since \(f\) is an outer fibration it then follows that the lift \(\wtl{H}\colon \Del^1_{\flat} \otimes \Del^n_{\flat} \to \E\) exists, as desired.
\end{proof}

\begin{rem}
Passing to opposites, Proposition~\ref{p:lax-lift} yields a dual statement for the case where
the lift is taken against the map
\[ L \otimes \Del^{\{0\}} \displaystyle\mathop{\coprod}_{K \otimes \Del^{\{0\}}} K \otimes \Del^1_{\flat} \to L \otimes \Del^1_{\flat} ,\]
assuming as above that edges of the form \(\{v\} \times \Del^1\) are sent \emph{\(f\)-cocartesian} edges.
In other words, we need to change \(\Del^{\{1\}}\) to \(\Del^{\{0\}}\) but also switch the order of the Gray product. In particular, we obtain cocartesian lifts for \emph{oplax} natural transformations given a lift of their domains. On the other hand, the analogue for \emph{inner cocartesian} fibrations, which is proven in~\cite[Lemma 4.1.7]{LurieGoodwillie}, states that such fibrations admit cocartesian lifts against
\[ \Del^{\{0\}} \otimes L  \displaystyle\mathop{\coprod}_{\Del^{\{0\}} \otimes K} \Del^1_{\flat} \otimes K  \to \Del^1_{\flat} \otimes L,\]
assuming again that edges of the form \(\Del^1 \times \{v\}\) are sent \(f\)-cocartesian edges.
In particular, they admit cocartesian lifts for lax transformations given a lift of their domains. Finally, passing to opposites one obtains that \emph{inner cartesian} fibrations admit cartesian lifts against 
\[ L \otimes \Del^{\{1\}} \displaystyle\mathop{\coprod}_{K \otimes \Del^{\{1\}}} K \otimes \Del^1_{\flat} \to L \otimes \Del^1_{\flat}\]
assuming that edges of the form \(\{v\} \times \Del^1\) are sent to \(f\)-cartesian edges. The last claim can also be proven using exactly the same filtration as in the proof of Proposition~\ref{p:lax-lift}, which this time will involve a slightly different scaling (cf.\ the proof of~\cite[Lemma 4.1.7]{LurieGoodwillie}). 
\end{rem}

\section{The bicategorical Grothendieck--Lurie correspondence}
\label{sec:correspondence}

In this section we will prove one of the principal results of the present paper by showing that the four types of fibrations studied in \S\ref{s:fibrations}, over a fixed base \(\B\), encode the four variance flavors of \(\B\)-indexed \(\Catoo\)-valued diagrams, a phenomenon we call \emph{the bicategorical Grothendieck--Lurie correspondence}. Our approach is as follows. First, in \S\ref{s:enriched} we will define analogues of the four fibration types in the setting of \(\s^+\)-enriched categories, and show that these are equivalent to the ones defined in the setting of scaled simplicial sets via the Quillen equivalence
\[
\xymatrixcolsep{1pc}
\vcenter{\hbox{\xymatrix{
			**[l]\Ss \xtwocell[r]{}_{\Nsc}^{\fCs}{'\perp}& **[r] \nCat{\s^+}}}} .
\]
The advantage of \(\nCat{\s^+}\) as a model for \((\infty,2)\)-categories is that it admits point-set models for the \((\ZZ/2)^2\)-symmetry of the theory of \((\infty,2)\)-categories, a fact we will exploit in \S\ref{sec:straightening} in order to show that this symmetry switches between the four types of fibrations. This reduces the bicategorical Grothendieck--Lurie correspondence to the inner cocartesian case, a statement essentially proven in~\cite{LurieGoodwillie} and that we extract to the present context in \S\ref{s:correspondence}.

\subsection{Fibrations of enriched categories}
\label{s:enriched}

In this section we will study analogous of the notions of inner/outer (co)cartesian fibrations in the setting of \(\Catoo\)-categories, by which we mean fibrant objects in \(\nCat{\s^+}\) with respect to the Dwyer-Kan model structure. Explicitly, this means that their mapping objects are \(\infty\)-categories marked by their equivalences, which is at the origin of the above term.

\begin{define}
	Let \(f\colon\C\rightarrow \D\) be a 
	map of \(\Catoo\)-categories. An arrow \(e\colon x \to y\) 
	in \(\C\) 
	is said to be \ndef{\(f\)-cartesian} if for every \(z \in \C\) the induced square
	\[\begin{tikzcd}[column sep=large]
		\C(z,x)\ar[r,"e\circ-"] \ar[d,"f_{z,x}"{swap}]&\C(z,y)\ar[d,"f_{z,y}"]\\
		\D(fz,fx) \ar[r,"f(e)\circ-"]& \D(fz,fx)
	\end{tikzcd}\]
	is a homotopy pullback square in \(\Set_+\). 
\end{define}

Recall that we use the term \emph{marked left} (resp.~\emph{right}) \emph{fibration} to indicate a map of marked simplicial sets \(f\colon X \to Y\) which detects marked edges and which is a left (resp.~right) fibration on the level of underlying simplicial sets.
We say that a map of \(\Catoo\)-categories \(f \colon \C \to \D\)
is \ndef{locally a marked left} (resp.~\ndef{right}) \ndef{fibration} if
for any pair of objects \(x, y\) of \(\C\), the induced map
\(\C(x, y) \to \D(fx, fy)\) is a marked left (resp.~right) fibration. 

\begin{define}
	\label{d:cart_fib_marked_cat}
	Let \(f\colon \C \rightarrow\D\) be
	a map of \(\Catoo\)-categories. 
	We say that \(f\) is an \ndef{inner cartesian fibration} (resp.~\ndef{outer cartesian fibration})
	if it satisfies the following properties:
	\begin{enumerate}
	\item Given \(y \in \C\) and an arrow \(e\colon x \to f(y)\) in \(\D\), there exists 
	a \(f\)-cartesian arrow \(\wtl{e}\colon x' \to y\) such that \(f(\wtl{e}) = e\). 
	\item The map \(f\colon \C \to \D\) is locally a marked right (resp.~left) fibration.
	\end{enumerate}
	We say \(f\colon \C \rightarrow \D\) is an inner (resp.~outer) cocartesian fibration if \(f^{\op}\colon \C^{\op} \rightarrow \D^{\op}\) is an inner (resp.~outer) cartesian fibration, where the operation \((-)^{\op}\) is defined by \(\E^{\op}(x,y) = \E(y,x)\) (see Construction~\ref{cn:op-action} and the discussion preceding it in \S\ref{sec:scaled}).
\end{define}

By a \emph{fibration of \(\Catoo\)-categories} we will simply mean a fibration between fibrant objects with respect to the Dwyer-Kan model structure. 

\begin{prop}\label{p:comparison}
Let \(f\colon \C \to \D\) be a fibration of \(\Catoo\)-categories. Then \(f\) is an inner (resp.~outer) cartesian fibration in the above sense if and only if 
\[\Nsc f\colon \Nsc \C \to \Nsc \D\] 
is an inner (resp.~outer) cartesian fibration in the sense of Definition~\ref{d:car-fibration}. In addition, an arrow in \(\C\) is \(f\)-cartesian if and only if the corresponding edge in \(\Nsc(\C)\) is \(f\)-cartesian.
\end{prop}

\begin{rem}
Since \(\Nsc(\C^{\op}) \cong \Nsc(\C)^{\op}\) the statement of Proposition~\ref{p:comparison} implies the same statement for inner/outer cocartesian fibrations.
\end{rem}

In the proof that follows we will use the following notation. We will denote by \(\Box^n = (\Del^1)^n\) the \(n\)-cube and by \(\partial \Box^n\) its boundary, so that the inclusion \(\partial \Box^n \subseteq \Box^n\) can be identified with the pushout-product of \(\partial \Del^1 \hrar \Del^1\) with itself \(n\) times. We also denote by \(\sqcap^{n-1,i}_1 \hrar \Box^{n-1}\) the iterated pushout-product 
\[[\partial \Del^1 \hrar \Del^1] \Box \hdots \Box [\Del^{\{\eps\}} \hrar \Del^1] \Box \hdots \Box [\partial \Del^1 \hrar \Del^1],\]
where \(\eps \in \{0,1\}\) and \([\Del^{\{\eps\}} \hrar \Del^1]\) appears in the \(i\)'th factor. 

\begin{proof}[Proof of Proposition~\ref{p:comparison}]
To begin, we recall from Remark~\ref{r:dwyer-kan} that we have canonical marked categorical equivalences 
\[ \Hom_{\Nsc \C}(x,y) \simeq \C(x,y) \quad\text{and}\quad \Hom_{\Nsc \D}(z,w) \simeq \D(z,w) .\]
By Proposition~\ref{p:mapping-2} we then see that if \(\Nsc(f)\) is an inner (resp.~outer) fibration then \(f_{x,y}\colon \C(x,y) \to \D(fx,fy)\) is weakly equivalent as an arrow to a marked right (resp.~left) fibration. Since \(f\) is a fibration in the Dwyer-Kan model structure we have that each \(f_{x,y}\) is a fibration between fibrant objects. Since the condition of being a marked right (resp.~left) fibration is given in terms of a suitable right lifting property this is equivalent to \(f_{x,y}\) itself being a marked right (resp.~left) fibration. Finally, Proposition~\ref{p:cart_1-simp_via_homs} implies that every \(\Nsc(f)\)-cartesian edge of \(\Nsc \C\) is also \(f\)-cartesian as an edge of \(\C\). Since the objects and arrows of \(\Nsc\C\) are in bijection with the objects and arrows of \(\C\), and the same goes for \(\D\), we now conclude that if \(\Nsc(f)\) is an inner (resp.~outer) cartesian fibration then \(f\) is an inner (resp.~outer) cartesian fibration of \(\Catoo\)-categories.
 
Now assume that \(f\) is an inner (resp.~outer) cartesian fibration of \(\Catoo\)-categories. Since \(f\) was assumed to be a fibration between fibrant objects it follows that \(\Nsc(f) \colon \Nsc \C \to \Nsc \D\) is a bicategorical fibration of \(\infty\)-bicategories, and in particular a weak fibration (Remark~\ref{r:bicategorical-weak}). As above, Proposition~\ref{p:cart_1-simp_via_homs} implies that every \(f\)-cartesian edge of \(\C\) is at least weakly \(f\)-cartesian as an edge of \(\Nsc\C\), and hence \(f\)-cartesian by Proposition~\ref{p:weak-to-strong}. Using again the bijection between objects and arrows of \(\Catoo\)-categories and their scaled nerves we conclude that \(\Nsc(f)\) satisfies Condition (2) of Definition~\ref{d:car-fibration}, that is, 
arrows in \(\Nsc\D\) admit \(\Nsc(f)\)-cartesian lifts.

We now show that \(f\) is an inner fibration (resp.~outer) fibration.
In the inner case, consider a lifting problem of the form
\[\begin{tikzcd}
(\Lambda^n_i)^{\flat} \ar[r] \ar[d] & \Nsc \C \ar[d,"\Nsc(f)"]\\
(\Delta^n)^{\flat} \ar[r] \ar[ur,dotted] & \Nsc \D . \
\end{tikzcd}\]
with \(0 < i < n\). By adjunction, this corresponds to a lifting problem of the form 
\begin{equation}
\label{t}
\begin{tikzcd}
 \fCs\Lambda^{n}_{i} \ar[rr,"h"] \ar[d]& &\C \ar[d,"f"]\\
 \fCs\Delta^n \ar[rr,"g"]&  &\D
\end{tikzcd}
\end{equation}
As a straightforward calculation shows, the lifting problem in~\eqref{t} corresponds, at the level of simplicial sets, to the following lifting problem:
\[\begin{tikzcd}
(\sqcap^{n-1,i}_1)^{\flat} \ar[d] \ar[r] & \C(h(0),h(n))\ar[d,"f_{h(0),h(n)}"]\\
(\Box^{n-1})^{\flat}\ar[r] & \D(g(0),g(n)) \ .
\end{tikzcd}\]
The last square then admits a lift since \(f\) is assumed to be locally a marked right fibration and \(\sqcap^{n-1,i}_1 \hrar \Box^{n-1}\) is right anodyne.

In the outer case, we have to solve any lifting problem of the form:
\[\begin{tikzcd}
(\Lambda^n_0 \coprod_{\Delta^{\{0,1\}}}\Delta^0)^{\flat} \ar[r] \ar[d] & \Nsc \C \ar[d,"\Nsc(f)"]\\
 (\Delta^n\coprod_{\Delta^{\{0,1\}}}\Delta^0)^{\flat} \ar[r] & \Nsc \D
\end{tikzcd}
\]
Arguing as above we see that this amounts to solving two lifting problems in the category of marked simplicial sets, namely:
\begin{equation}
\label{1st}
\begin{tikzcd}
(\partial \Box^{n-2})^{\flat} \ar[d] \ar[r] & \C(h(1),h(n)) \ar[d,"f_{h(1),h(n)}"]\\
(\Box^{n-2})^{\flat} \ar[r]& \D(g(1),g(n))
\end{tikzcd}
\end{equation}
and
\begin{equation}
\label{2nd}
\begin{tikzcd}
(\sqcap^{n-1,1}_0)^{\flat} \ar[r] \ar[d] & \C(h(0),h(n)) \ar[d,"f_{h(0),h(n)}"]\\
(\Box^{n-1})^{\flat} \ar[r]& \D(g(0),g(n))
\end{tikzcd}
\end{equation}
Since we collapsed the edge \(\Delta^{\{0,1\}}\), we get that pre-composition with the image of the map \(0\x{=}{\to}1\)
induces a commutative square of the form 
\[\begin{tikzcd}
\C(f(1),f(n)) \ar[r,"="] \ar[d,"f_{f(1),f(n)}"{swap}]&\C(f(0),f(n))\ar[d,"f_{f(0),f(n)}"]\\
\D(g(1),g(n)) \ar[r,"="]&\D(g(0),g(n))
\end{tikzcd}\]
Under this identification, the lifting problem \eqref{1st} corresponds to filling the missing \((i,1)\)-face of \(\Box^{n-1}\) 
in \eqref{2nd}. Therefore, solving \eqref{2nd} also produces a solution for \eqref{1st}. A solution to \eqref{2nd} 
then exists since \(\sqcap^{n-1,1}_0 \subseteq \Box^{n-1}\) is left anodyne and \(f\) is now assumed to be locally a marked left fibration.
\end{proof}

\begin{cor}\label{c:comparison}
For a map \(f \colon \E \to \B\) in \(\BiCat\) the following are equivalent:
\begin{enumerate}
\item
\(f\) can be represented by an inner/outer (co)cartesian fibration of \(\infty\)-bicategories.
\item
Under the equivalence \((\nCat{\s^+})_{\infty} \xrightarrow{\simeq} \BiCat^{\thi}\) the arrow \(f\) can be represented by an inner/outer (co)cartesian fibration of \(\Catoo\)-enriched categories.
\end{enumerate}
\end{cor}

\subsection{The \texorpdfstring{\(\co/\op\)}{co/op}-symmetry and cartesian fibrations}\label{sec:straightening}

In the previous section we defined inner/outer (co)cartesian fibrations for \(\Catoo\)-categories, and showed that these coincide with the corresponding notion of inner/outer (co)cartesian fibrations under the scaled nerve functor \(\Nsc\colon \nCat{\s^+} \to \Ss\), which is a right Quillen equivalence. Both \(\nCat{\s^+}\) and \(\Ss\) are models for the theory of \((\infty,2)\)-categories, and Corollary~\ref{c:comparison} suggests to consider the notions of inner/outer (co)cartesian fibrations model independently:

\begin{define}\label{d:cartesian-map}
We will refer to arrows in \(\BiCat\) which satisfy the equivalent conditions of Corollary~\ref{c:comparison} as inner/outer (co)cartesian \ndef{maps}. In addition, given a map \(f\colon\E \to \B\) in \(\BiCat\) and a 1-morphism \(e\colon \Del^1 \to \E\) in \(\E\), we will say that \(e\) is \(f\)-(co)cartesian if we can represent \(f\) by a map of \(\infty\)-bicategories such that \(e\) is represented by a \(f\)-(co)cartesian edge. Equivalently, by Proposition~\ref{p:comparison} this is the same as saying that we can represent \(f\) by a map of \(\Catoo\)-enriched categories such that \(e\) is represented by a \(f\)-(co)cartesian arrow.
\end{define}

\begin{lemma}\label{l:op-fibration}
Let \(f\colon \C \to \D\) be a map of \(\Catoo\)-categories and \(e\colon [1] \to \C\) an arrow in \(\C\).
Then the following are equivalent:
\begin{enumerate}
\item
\(e\colon [1] \to \C\) is \(f\)-cartesian.
\item
\(e^{\op}\colon [1]^{\op} \cong [1] \to \C^{\op}\) is \(f^{\op}\)-cocartesian.
\item
\(e^{\co}\colon [1]^{\co}\cong[1] \to \C^{\co}\) is \(f^{\co}\)-cartesian.
\item
\(e^{\coop}\colon [1]^{\coop}\to \C^{\coop}\) is \(f^{\coop}\)-cocartesian.
\end{enumerate}
In addition, \(f\) is locally a marked right fibration if and only \(f^{\co}\) is locally a marked left fibration, while the operation \((-)^{\op}\) preserves locally marked left/right fibrations. 
\end{lemma}
\begin{proof}
All the claims follow directly from the definitions and the fact that the operation \((-)^{\op}\) on the level of marked simplicial sets preserves and detects homotopy cartesian squares and switches between marked left fibrations and marked right fibrations.
\end{proof}

The following two corollaries directly follow:
\begin{cor}\label{c:op-fibration}
Let \(f\colon \C \to \D\) be a map of \(\Catoo\)-categories.
Then the following are equivalent:
\begin{enumerate}
\item
\(f\colon \C \to \D\) is an inner cartesian fibration.
\item
\(f^{\op}\colon \C^{\op} \to \D^{\op}\) is an inner cocartesian fibration.
\item
\(f^{\co}\colon \C^{\co} \to \D^{\co}\) is an outer cartesian fibration.
\item
\(f^{\coop}\colon \C^{\coop} \to \D^{\coop}\) is an outer cocartesian fibration. 
\end{enumerate}
\end{cor}

\begin{cor}\label{c:op-fibration-2}
Let \(f\colon \B \to \E\) be a map in \(\BiCat\). Then the following are equivalent:
\begin{enumerate}
\item
\(f\colon \E \to \B\) is an inner cartesian map.
\item
\(f^{\op}\colon \E^{\op} \to \B^{\op}\) is an inner cocartesian map.
\item
\(f^{\co}\colon \E^{\co} \to \B^{\co}\) is an outer cartesian map.
\item
\(f^{\coop}\colon \E^{\coop} \to \B^{\coop}\) is an outer cocartesian map.
\end{enumerate}
In addition, an edge \(e\colon \Del^1 \to \E\) is \(f\)-cartesian if and only if \(e^{\op}\) is \(f^{\op}\)-cocartesian, if and only if \(e^{\co}\) is \(f^{\co}\)-cartesian, and if and only if \(e^{\coop}\) is \(f^{\coop}\)-cocartesian.
\end{cor}

\begin{define}\label{d:car-bicat}
For an \(\infty\)-bicategory \(\B\), let us denote by \(\Car^\inn(\B)\) (resp.~\(\Car^\out(\B)\), 
	\(\coCar^\inn(\B)\), \(\coCar^\out(\B\))) the sub-bicategories of \((\BiCat)_{/\B}\) spanned by the
	inner cartesian fibrations (resp.~outer cartesian fibrations,
	inner cocartesian fibrations, outer cocartesian fibrations) over \(\B\) and the 1-morphisms which preserve (co)cartesian edges. 
\end{define}

The following corollary is the main conclusion of the present section. To formulate it, we recall from Remark~\ref{r:twisted-op} that the equivalence \((-)^{\co}\colon \BiCat^{\thi} \xrightarrow{\simeq} \BiCat^{\thi}\) extends to a bicategorical equivalence \((-)^{\co}\colon \BiCat \to \BiCat\), while the equivalence \((-)^{\op}\colon \BiCat^{\thi} \xrightarrow{\simeq} \BiCat^{\thi}\) becomes a bicategorical equivalence of the form \((-)^{\op}\colon \BiCat \to \BiCat^{\co}\).

\begin{cor}\label{c:goal}
For a fixed \(\B \in \BiCat\), the induced bicategorical equivalence \((-)^{\co}\colon (\BiCat)_{/\B} \xrightarrow{\simeq} (\BiCat)_{/\B^{\co}}\) restricts to give bicategorical equivalences 
\[ \Car^{\inn}(\B) \xrightarrow{\simeq} \Car^{\out}(\B^{\co}) \quad\text{and}\quad \coCar^{\inn}(\B) \xrightarrow{\simeq} \coCar^{\out}(\B^{\co}).\]
Similarly, the induced bicategorical equivalence \((-)^{\op}\colon (\BiCat)_{/\B} \xrightarrow{\simeq} (\BiCat)_{/\B^{\op}}^{\co}\) restricts to give bicategorical equivalences 
\[ \Car^{\inn}(\B) \xrightarrow{\simeq} \coCar^{\inn}(\B^{\op})^{\co} \quad\text{and}\quad \Car^{\out}(\B) \xrightarrow{\simeq} \coCar^{\out}(\B^{\op})^{\co}.\]
\end{cor}

\subsection{Straightening and unstraightening}\label{s:correspondence}

In light of Proposition~\ref{p:inner-fibred}, the \(\infty\)-bicategory \(\coCar^\inn(\B)\) of Definition~\ref{d:car-bicat} can be identified with the scaled coherent nerve of the fibrant \(\s^+\)-enriched category \([\trbis{(\s^+)}{\B}]^{\circ} \subseteq \trbis{(\s^+)}{\B}\) spanned by the fibrant(-cofibrant) objects with respect to the \(\Beta_{\B}\)-fibered model structure. In this subsection we will use the connection, together with the results of the previous subsection, in order to extract from Lurie's straightening-unstraightening theorem the bicategorical Grothendieck--Lurie correspondence for all variance flavors.

Let \(\C\) be a \(\s^+\)-enriched category and \(\phi\colon \fCs(\B) \to \C\) a Dwyer-Kan equivalence. By Lemma~\ref{l:unstraightening-simplicial}, Lurie's scaled unstraightening functor induces a Dwyer-Kan equivalence of \(\Catoo\)-categories
\[ [(\s^+)^{\C}]^{\circ} \xrightarrow{\simeq} [\trbis{(\s^+)}{\B}]^{\circ},\]
and hence an equivalence of \(\infty\)-bicategories
\begin{equation}\label{e:unstr} 
\Nsc[(\s^+)^{\C}]^{\circ} \xrightarrow{\simeq} \Nsc[\trbis{(\s^+)}{\B}]^{\circ} \simeq \coCar^\inn(\B)
\end{equation}
We now claim that the \(\infty\)-bicategory \(\Nsc[(\s^+)^{\C}]^{\circ}\) 
is naturally equivalent to the \(\infty\)-bicategory of functors \(\Nsc(\C) \to \Nsc([\s^+]^{\circ}) \simeq \Catoo\). To see this let us first construct a map
\begin{equation}\label{e:rectification} 
\Nsc([(\s^+)^{\C}]^{\circ}) \to \Fun(\B,\Nsc[\s^+]^{\circ}).
\end{equation}
Given a scaled simplicial set \(K\), maps from \(K\) to the left hand side in~\eqref{e:rectification} correspond by adjunction to enriched functors \(\fCs(K) \to [(\s^+)^{\C})]^{\circ}\), which in turn correspond to enriched functors \(\fCs(K) \times \C \to [\s^+]^{\circ}\) satisfying a certain condition. On the other hand, maps from \(K\) to the right hand side in~\eqref{e:rectification} correspond to maps \(K \times \B \to \Nsc[\s^+]^{\circ}\), and hence to enriched functors \(\fCs(K \times \B) \to [\s^+]^{\circ}\). The map~\eqref{e:rectification} is then obtained by restriction along 
\[\fCs(K \times \B) \to \fCs(K) \times \fCs(\B) \to \fCs(K) \times \C.\]

\begin{prop}\label{p:rectification}
The map~\eqref{e:rectification} is an equivalence of \(\infty\)-bicategories.
\end{prop}
\begin{proof}
By the (enriched) Quillen equivalences of~\cite[Proposition A.3.3.8(1)]{HTT} we may as well assume that \(\C\) is fibrant. 
We now argue as in the proof of~\cite[Proposition 4.2.4.4]{HTT}. In particular, writing \(\s^+\) as a sufficiently filtered colimit of small \(\C\)-chunks \(\U\) in the sense of~\cite[Definition A.3.4.9]{HTT} 
(see also~\cite[Definition A.3.4.1]{HTT} for the notion of a chunk of an enriched model category), we may reduce to showing that for every small \(\C\)-chunk \(\U\) the map
\[\Nsc[\U^{\C}]^{\circ} \xrightarrow{\simeq} \Fun(\B,\Nsc(\U^{\circ}))
\]
is an equivalence of \(\infty\)-bicategories. Consider the composed map
\begin{equation}\label{e:composed}
\B \times \Nsc[\U^{\C}]^{\circ} \to \B \times \Fun(\B,\Nsc(\U^{\circ})) \to \Nsc(\U^{\circ}),
\end{equation}
the second one being the evaluation map.
Since \(\Nsc(\U^{\circ})\) is an \(\infty\)-bicategory and the bicategorical model structure is cartesian closed the second map
exhibits \(\Fun(\B,\Nsc(\U^{\circ}))\) as an internal mapping object in the homotopy category of \(\Ss\) (with respect to the bicategorical model structure). It will hence suffice to show that the composed map~\eqref{e:composed} exhibits \(\Nsc[\U^{\C}]^{\circ}\) as the same internal mapping object. Unwinding the definitions, this composed map identifies with the composed map
\[ \B \times \Nsc[\U^{\C}]^{\circ} \xrightarrow{\simeq} \Nsc(\C) \times \Nsc[\U^{\C}]^{\circ} \cong \Nsc(\C \times [\U^{\C}]^{\circ}) \to \Nsc(\U^{\circ}),\]
where the first map is a bicategorical equivalence since \(\C\) is now assumed fibrant and the second map is the image under \(\Nsc\) of the evaluation map 
\[ \C \times [\U^{\C}]^{\circ} \to \U^{\circ} .\]
Since \(\fCs \dashv \Nsc\) is a Quillen equivalence and this last evaluation map is between fibrant objects, it will now suffice to verify that it exhibits \([\U^{\C}]^{\circ}\) as an internal mapping object in \(\nCat{\s^+}\). Indeed, this last statement is established in~\cite[Corollary
A.3.4.14]{HTT}.
\end{proof}

Combining the map~\eqref{e:rectification} with the inverse of the equivalence~\eqref{e:unstr} we now obtain the following conclusion:

\begin{cor}[Lurie]\label{c:straightening-inner}
For an \(\infty\)-bicategory \(\B \in \BiCat\) there is a natural equivalence of \(\infty\)-bicategories
\[ \coCar^{\inn}(\B) \simeq \Fun(\B,\Catoo).\]
\end{cor}

\begin{cor} 
\label{c:S-U-for-outer-fibs}
For an \(\infty\)-bicategory \(\B \in \BiCat\) there are natural equivalences of \(\infty\)-bicategories
\[ \coCar^{\out}(\B) \simeq \Fun(\B^{\co},\Catoo),\]
\[ \Car^{\inn}(\B) \simeq \Fun(\B^{\coop},\Catoo),\]
and
\[ \Car^{\out}(\B) \simeq \Fun(\B^{\op},\Catoo).\]
\end{cor}
\begin{proof}
Combining 
Corollary~\ref{c:goal} and Corollary~\ref{c:straightening-inner} we obtain the first equivalence above as a composite of equivalences
\[ \coCar^{\out}(\B) \stackrel[\simeq]{(-)^{\co}}{\longrightarrow} \coCar^{\inn}(\B^{\co}) \simeq  \Fun(\B^{\co},\Catoo) ,\]
The same argument deduces the third desired equivalence from the second one. To obtain the second equivalence we again invoke Corollary~\ref{c:goal} and Corollary~\ref{c:straightening-inner}
to obtain equivalences of \(\infty\)-bicategories 
\[\Car^{\inn}(\B) \simeq \coCar^{\inn}(\B^{\op})^{\co} \simeq
 \Fun(\B^{\op},\Catoo)^{\co} \simeq \Fun(\B^{\coop},\Catoo^{\co}) \]
and finish by identifying \(\Catoo^{\co} \simeq \Catoo\) via the functor \((-)^{\op}\), see Example~\ref{ex:op-on-cat}. 
\end{proof}

\begin{rem}\label{r:unstr-base-change}
By~\cite[Remark 3.5.16 and Remark 3.5.17]{LurieGoodwillie} the scaled unstraightening functor intertwines base change with restriction. In particular, given a map \(f\colon \B \to \B'\) of \(\infty\)-categories, the base change and restriction functors fit in a commutative square of right quillen functors
\[ \xymatrix{
(\s^+)^{\fCs(\B')}
\ar[rr]^{\B \times_{\B'} (-)}\ar[d] && (\s^+)^{\fCs(\B)}  \ar[d] \\
\trbis{(\s^+)}{\B'}  \ar[rr]^{\fCs(f)^*} && \trbis{(\s^+)}{\B}\ ,
}\]
whose vertical arrows are the respective unstraightening functors. Applying
the operation~\(\Nsc([-]^{\circ})\) and taking into account the identification of Proposition~\ref{p:rectification} we obtain a commutative square of \(\infty\)-bicategories
\[ \xymatrix{
\Fun(\B',\Catoo) \ar[r]\ar[d]_-{\simeq} & \Fun(\B,\Catoo) \ar[d]^-{\simeq} \\
\coCar^{\inn}(\B')  \ar[r] & \coCar^{\inn}(\B) \ ,
}\]
expressing the fact that the Grothendieck--Lurie correspondence intertwines between base change on the level of fibrations and restriction on the level of diagrams. Since all the four flavors of the Grothendieck--Lurie correspondence in Corollary~\ref{c:S-U-for-outer-fibs} are deduced from the inner cocartesian one by acting with the \((\ZZ/2)^2\)-symmetry spanned by \((-)^{\op}\) and \((-)^{\co}\) (which certainly preserves the notion of base change) it follows that all four flavors enjoy the exact same base-change--restriction compatibility.
\end{rem}

\section{Lax transformations and thick slice fibrations}
\label{sec:thick-slice}

In the setting of \(\infty\)-categories, the usual slice construction, featuring prominently in the theory of \(\infty\)-categories, is given a ``thickened'' counterpart in~\cite[\S 4.2.1]{HTT}.  
This counterpart is equivalent to the usual one, but offers occasional technical advantages. In this section we consider the analogous situation in the setting of \(\infty\)-bicategories. 
An important difference is however present: while in~\cite{HTT} the thickened joint and slice constructions constitute a completely equivalent alternative, whose role is mostly technical, 
here they offer an important conceptual advantage. More precisely, while there is only one type of join \(X \ast Y\) of two marked-scaled simplicial sets given the order of factors, 
the thick join comes in \ndef{two} flavors, which we call the inner and outer thick join. This means that, given a marked-scaled simplicial set \(K\) and diagram 
\(f\colon \ovl{K} \lrar \C\) in an \(\infty\)-bicategory \(\C\), there are now not only two different slice constructions \(\C_{/f}\) and \(\C_{f/}\),  
but \ndef{four} different (thickened) slice constructions, which we will denote by \(\C^{/f}_{\inn},\C^{/f}_{\out}, \C^{f/}_{\inn}\) and \(\C^{f/}_{\out}\). 
This allows to incorporate into the slice construction the \((\ZZ/2)^2\)-symmetry of the theory of \((\infty,2)\)-categories, which we heavily relied on in \S\ref{sec:correspondence} when discussing the Grothendieck--Lurie correspondence.
In particular, the projections from the four types of slice constructions to \(\C\) constitute examples of the four types of fibrations we studied above, that is, inner/outer cartesian fibrations and inner/outer cocartesian fibrations, respectively.

In addition to the theoretical advantage of avoiding a break of symmetry, the framework of four slice constructions allows one to define and study the corresponding four types of \((\infty,2)\)-cateogrical limits, namely, the lax limit, the oplax limit, the lax colimit and the oplax colimit, as well as their marked (or \emph{partially} lax) versions, see \S\ref{subsec:universal} below.

The thickened slice construction which we will introduce and study in \S\ref{s:thick-join} is based on a marked version of the Gray product, to which we dedicate \S\ref{s:gray}. In \S\ref{s:representable} study the relation between slice fibrations over an object and representable functors via the bicategorical Grothendieck--Lurie correspondence discussed in \S\ref{sec:correspondence}. Finally, in \S\ref{s:lax-trans} we turn to the general case of slice fibrations and relate its fibers to various \(\infty\)-categories of lax transformations.

\subsection{Gray products of marked-scaled simplicial sets}
\label{s:gray}

In this section we consider a version of the Gray product in the setting of marked-scaled simplicial sets, which we will use in \S\ref{s:thick-join} to construct the thickened join and slice constructions. We note that
a definition of the Gray product in the setting of scaled simplicial sets was given by the authors in
~\cite{GagnaHarpazLanariGrayLaxFunctors}, where it was also shown to be equivalent to Verity's Gray product studied in~\cite{VerityWeakComplicialI}, under the Quillen equivalence between the bicategorical and 2-trivial complicial model structures~\cite{GagnaHarpazLanariEquiv}. The version we use here takes as input marked-scaled simplicial sets and returns a scaled simplicial set. Such a construction can be defined in several ways, though, as in the case of stratified sets one needs to choose between a definition which is associative but not entry-wise colimit preserving and a definition that is entry-wise colimit preserving but not associative (see~\cite[\S 5.1]{VerityWeakComplicialI}). To circumvent this difficulty we define a Gray product here for any tuple \(X_1,...,X_n\) of marked-scaled simplicial sets, so as to avoid the need to iterate binary Gray products.

\begin{define}\label{d:gray-marked}
For \(n \geq 2\) and marked-scaled simplicial sets \(X_1,...,X_n \in \Sms\) we define their associated Gray product
\[ X_1 \mgr \cdots \mgr X_n \in \Ss \]
to be the scaled simplicial set whose underlying simplicial set is the cartesian product of \(X_1,...,X_n\) and such that a triangle 
\(\sig = (\sig_1,...,\sig_n)\colon\Del^2_{\flat} \to X_1 \mgr \cdots \mgr X_n\) is thin if and only if
the following conditions hold:
\begin{enumerate}[leftmargin=*]
\item
Each \(\sig_i\) is thin in \(X_i\).
\item
There exists a \(j \in \{1,...,n\}\) such that \(\sig_i\) is degenerate for \(i \neq j\), \({\sig_i}_{|\Del^{\{0,1\}}}\) is marked for \(i > j\) and \({\sig_i}_{|\Del^{\{1,2\}}}\) is marked for \(i < j\).
\end{enumerate}
\end{define}

\begin{warning}
To avoid confusion, we explicitly point out that the above Gray product takes as input a sequence of marked-scaled simplicial sets, and outputs just a scaled simplicial set, to which we associate no particular marking.
\end{warning}

\begin{rem}\label{r:unmarked}
If \(X_1,..,X_n\) have only their degenerate edges marked then the Gray product of Definition~\ref{d:gray-marked} coincides with the iteration of the binary Gray product of scaled simplicial sets defined in~\cite[\S 2]{GagnaHarpazLanariGrayLaxFunctors}. 
\end{rem}

\begin{rem}\label{r:almost-associative}
For marked-scaled simplicial sets \(X_1,...,X_n,Y\) there is a natural map
\[ (X_1 \otimes \cdots \otimes X_n)^{\flat} \otimes Y \to X_1 \otimes \cdots \otimes X_n \otimes Y \]
which is an isomorphism on underlying simplicial sets, but is generally only an inclusion on thin triangles. 
This map is however an isomorphism in the particular case where \(X_1,...,X_n\) have no non-degenerate marked edges.
\end{rem}

\begin{rem}\label{r:colimit-preserving}
For fixed marked-scaled simplicial sets \(X_1,...,X_{n}\) the functors \(Y \mapsto X_1 \otimes \cdots \otimes X_n \otimes Y\) and \(Y \mapsto Y \otimes X_1 \otimes \cdots \otimes X_n\) are colimit preserving functors from \(\Sms\) to \(\Ss\). The proof proceeds exactly as the proof of~\cite[Lemma 142]{VerityComplicial}.
\end{rem}

\begin{rem}\label{r:symmetric}
	The Gray product is \emph{not} symmetric in general. 
	For marked-scaled simplicial sets \(X_1,...,X_n\) we however have the relation
	\[
	 X_1 \mgr \cdots \mgr X_n = \bigl(X_n^\op \mgr \cdots \mgr X_1^\op \bigr)^\op .
	\]
\end{rem}

\begin{example}
	While the Gray product \(\prescript{\flat}{}\Del^1 \mgr \prescript{\flat}{}\Del^1\)
	is a square
	\[
	 \begin{tikzcd}[column sep=3em, row sep=large]
	  (0,0) \ar[r] \ar[d] \ar[rd, ""{swap, name=diag}] &
	  (1,0) \ar[d] \\
	  (0,1) \ar[r] & (1,1)
	  \ar[Rightarrow, from=diag, to=2-1]
	  \ar[from=2-1, to=1-2, phantom, "\simeq"{description, pos=0.75}]
	 \end{tikzcd}
	\]
	in which exactly one of the triangles is thin, 
	the Gray products
	\[
	 \prescript{\sharp}{}\Del^1 \mgr \prescript{\flat}{}\Del^1
	 \quad,\quad 
	 \prescript{\flat}{}\Del^1 \mgr \prescript{\sharp}{}\Del^1
	 \quad\mathrm{and}\quad
	 \prescript{\sharp}{}\Del^1 \mgr \prescript{\sharp}{}\Del^1
	\]
	are all squares in which \ndef{both} triangles are thin.
\end{example}

As in the unmarked case (see~\cite[\S 2]{GagnaHarpazLanariGrayLaxFunctors}), 
the marked Gray product admits equivalent variants. 
Let \(X,Y\) be marked-scaled simplicial sets and let \(T_{\gr} \subseteq (X \times Y)_2\) denote the collection of triangles which are thin in \(X \mgr Y\), as described in Definition~\ref{d:gray-marked}. 
Let \(T_- \subseteq T_{\gr}\) denote the subset of those triangles \(\sig = (\sig_X,\sig_Y) \in T_{\gr}\)
for which either both \(\sig_X\) and \(\sig_Y\) are degenerate or at least one of \(\sig_X,\sig_Y\) degenerates to a \ndef{point}. On the other hand, let \(T_+\) be the set of those triangles 
\((\sig_X,\sig_Y)\colon \prescript{\flat}{}\Del^2 \lrar X \times Y\) for which both \(\sig_X\) and \(\sig_Y\)
are thin and such that either \({\sig_X}_{|\Del^{\{1,2\}}}\) is marked or \({\sig_Y}_{|\Del^{\{0,1\}}}\) is marked. Then we have a sequence of inclusions
\[
 T_- \subseteq T_{\gr} \subseteq T_+ .
\]

\begin{prop}\label{p:marked-variants}
	Let \((X,E_X,T_X)\) and \((Y,E_Y,T_Y)\) be two marked-scaled simplicial sets and let \(T_{\gr}\) be the collection of thin triangles in \((X,E_X,T_X) \mgr (Y,E_Y,T_Y)\). Then the maps 
\begin{equation}\label{e:marked-variants}
	 (X \times Y,T_-) \hrar (X \times Y,T_{\gr}) \hrar (X \times Y,T_+)
\end{equation}
	are bicategorical trivial cofibrations. 
\end{prop}
\begin{proof}
We will show that for every triangle \(\sig \in T_+\) there is a \(3\)-simplex \(\rho\colon \Del^3 \to X \times Y\) and an \(i \in \{1,2\}\) such that \(\eta_{|\Del^{\{0,i,3\}}} = \sig\) while the three other faces of \(\rho\) lie in \(T_{\gr}\). This will imply that the second map in~\eqref{e:marked-variants} is a sequence of pushouts along maps of the form \((\Del^3,T_i) \to (\Del^3)_{\sharp}\), where \(T_i\) denotes all triangles except \(\Del^{\{0,i,3\}}\), and is hence scaled anodyne by~\cite[Remark 3.1.4]{LurieGoodwillie}. We will then apply the same argument to show that the first map in~\eqref{e:marked-variants} is scaled anodyne.

Given a triangle \(\sig = (\sig_X,\sig_Y)\colon \Del^2 \to X \times Y\), let us denote by 
\[\rho^{i,j}_\sig = (s^i\sig_X,s^j\sig_Y)\colon \Del^3 \to X \times Y\] 
the \(3\)-simplex in \(X \times Y\) whose \(X\) component is the degenerate \(3\)-simplex obtained by pre-composing \(\sig_X\) with the surjective map \([3] \to [2]\) hitting \(i \in [2]\) twice, and whose \(Y\) component is obtained by pre-composing \(\sig_Y\) with the surjective map \([3] \to [2]\) hitting \(j\) twice. We note in particular that  
\[d^2\rho^{1,2}_{\sig} = d^2\rho^{2,1}_{\sig} =\sig \quad\text{and}\quad d^1\rho^{1,0}_{\sig} = d^1\rho^{0,1}_{\sig} = \sig .\]
Now suppose that \(\sig\) belongs to \(T_+\). Then \(\sig_X\) and \(\sig_Y\) are thin, 
and either \({\sig_X}_{|\Del^{\{1,2\}}}\) is marked in \(X\) or \({\sig_Y}_{|\Del^{\{0,1\}}}\) is marked in \(Y\). If \({\sig_X}_{|\Del^{\{1,2\}}}\) is marked in \(X\) then the \(3\)-simplex \(\rho^{1,0}_{\sig}\) has the property that all its faces are in \(T_{\gr}\) except possibly its face opposite the vertex \(2\), which is \(\sig\). Similarly, if \({\sig_Y}_{|\Del^{\{0,1\}}}\) is marked in \(Y\) then \(\rho^{2,1}_{\sig}\) has the property that all its faces are in \(T_{\gr}\) except possibly its face opposite the vertex \(1\), which is again \(\sig\). We may hence conclude that the map \((X \times Y,T_{\gr}) \hrar (X \times Y,T_+)\) is scaled anodyne.

Now suppose that \(\sig\) belongs to \(T_{\gr}\). Then \(\sig_X\) and \(\sig_Y\) are thin, and either \(\sig_X\) is degenerate and \({\sig_X}_{|\Del^{\{1,2\}}}\) is marked or \(\sig_Y\) is degenerate and \({\sig_Y}_{|\Del^{\{0,1\}}}\) is marked in \(Y\). We now separate into \emph{four} cases: 
\begin{enumerate}
\item
If \(\sig_X\) degenerates along \(\Del^{\{1,2\}}\) then the \(3\)-simplex \(\rho^{1,0}_{\sig}\) has the property that all its faces are in \(T_{-}\) except possibly its face opposite \(1\), which is \(\sig\).
\item
If \(\sig_X\) degenerates along \(\Del^{\{0,1\}}\) and \({\sig_X}_{|\Del^{\{1,2\}}}\) is marked and then the \(3\)-simplex \(\rho^{1,2}_{\sig}\) has the property that all its faces are in \(T_{-}\) except possibly its face opposite \(2\), which is \(\sig\).
\item
If \(\sig_Y\) degenerates along \(\Del^{\{0,1\}}\) then the \(3\)-simplex \(\rho^{2,1}_{\sig}\) has the property that all its faces are in \(T_{-}\) except possibly its face opposite \(2\), which is \(\sig\).
\item
If \(\sig_Y\) degenerates along \(\Del^{\{1,2\}}\) and \({\sig_Y}_{|\Del^{\{0,1\}}}\) is marked and then the \(3\)-simplex \(\rho^{0,1}_{\sig}\) has the property that all its faces are in \(T_{-}\) except possibly its face opposite \(1\), which is \(\sig\).
\end{enumerate}
We may hence conclude that the map \((X \times Y,T_{-}) \hrar (X \times Y,T_{\gr})\) is scaled anodyne, and so the proof is complete.
\end{proof}

The following proposition extends~\cite[Proposition 2.16]{GagnaHarpazLanariGrayLaxFunctors}:
	
\begin{prop}\label{p:push-prod-2}
	Let \(f\colon X \lrar Y\) be a monomorphism of marked-scaled simplicial sets 
	and let \(g\colon Z \lrar W\) be a scaled anodyne map of \ndef{scaled simplicial sets}. 
	Then the pushout-products
	\[
	 f \hmgr g^{\flat} \colon \left[X \mgr W^{\flat}\right] \coprod_{X \mgr Z^{\flat}} 
	 \left[Y \mgr Z^{\flat}\right] \lrar Y \mgr W^{\flat}
	\]
	and
	\[
	 g^{\flat} \hmgr f \colon \left[W^{\flat} \mgr X\right] \coprod_{Z^{\flat} \mgr X} \left[Z^{\flat} \mgr Y\right] \lrar W^{\flat} \mgr Y
	\]
	are scaled anodyne maps. 
\end{prop}
\begin{proof}
	To prove this statement we can assume that that \(g\) is one of the generating scaled anodyne maps appearing in Definition~\ref{d:anodyne} and that \(f\) is either the inclusion 
	\(\prescript{\flat}{}(\partial \Del^n) \hrar \prescript{\flat}{}\Del^n\) for \(n \geq 0\), 
	the inclusion \(\prescript{\flat}{}\Del^1\hrar \prescript{\sharp}{}\Del^1\),
	or the inclusion \(\prescript{\flat}{}\Del^2 \subseteq (\Delta^2, \emptyset, \{\Del^2\})\).
Now the first and last cases follow, in light of Remark~\ref{r:unmarked}, from the unmarked analogue of the present proposition, see~\cite[Proposition 2.16]{GagnaHarpazLanariGrayLaxFunctors}.	
On the other hand, if \(f\) is the inclusion \(\prescript{\flat}{}\Del^1 \hrar \prescript{\sharp}{}\Del^1\), 
then \(f \hmgr g^{\flat}\) and \(g^{\flat} \hmgr f\) are isomorphisms except when \(g\) is the inclusion 
\((\Lam^2_1)_{\flat} \subseteq \Del^2_{\sharp}\). In this last case, the map \(g^{\flat} \hmgr f\) identifies with the inclusion 	
\[ (\Del^2 \times \Del^1,T) \subseteq (\Del^2 \times \Del^1)_{\sharp}, \]
where \(T\) is the collection of all triangles except
\[
	\Del^{\{(0,0),(1,0),(2,1)\}}\quad\text{and}\quad\Del^{\{(0,0),(2,0),(2,1)\}}.
\]
To see that this is scaled anodyne, it suffices by \cite[Remark 3.1.4]{LurieGoodwillie} to note that the \(3\)-simplex \(\rho\colon \Del^3 \to \Del^2 \times \Del^1\) spanned by the vertices \((0,0),(1,0),(1,1),(2,1)\) has the property that all its faces except the one opposite \((1,1)\) are in \(T\), while the face opposite \((1,1)\) is \(\Del^{\{(0,0),(1,0),(2,1)\}}\), while the \(3\)-simplex \(\rho\colon \Del^3 \to \Del^2 \times \Del^1\) spanned by the vertices \((0,0),(1,0),(2,0),(2,1)\) has the property that all its faces except the one opposite \((1,0)\) are in \(T \cup \{\Del^{\{(0,0),(1,0),(2,1)\}}\}\), while the face opposite \((1,0)\) is \(\Del^{\{(0,0),(2,0),(2,1)\}}\).
The case of \(f \hmgr g^{\flat}\) admits a completely analogous argument.
\end{proof}

\begin{cor}\label{c:marked-gray-quillen}
For every marked-scaled simplicial set \(X\) the functors
\[ X \mgr (-)^{\flat}\colon \Ss \to \Ss\]
and
\[ (-)^{\flat} \mgr X\colon \Ss \to \Ss\]
are left Quillen functors with respect to the bicategorical model structure.
\end{cor}
\begin{proof}
It is straightforward to verify that the functors in question preserve colimits and monomorphisms, and so it is left to verify that they preserve trivial cofibrations. We prove this for the first functor, the proof for the second one proceeds in a completely analogous manner.
In light of Proposition~\ref{p:push-prod-2} and the collection of generating trivial cofibrations established in~\cite{GagnaHarpazLanariEquiv}, it will suffice to check that for every marked-scaled simplicial set \(X\) the map
\[ \ovl{X} = X \mgr \{0\} \to X \mgr J_{\sharp}^{\flat} \]
is a bicategorical equivalence, where \(J = \cosk_0(\{0,1\})\) is the nerve of the walking isomorphism and \(\ovl{X}\) is the underlying scaled simplicial set of \(X\). 
Then \(X \mgr J_{\sharp}^{\flat}\) and \(\ovl{X} \mgr J_{\sharp}\) have isomorphic underlying simplicial sets, with the former having potentially more thin triangles than the latter. However, since \(J\) is a Kan complex~\cite[Corollary 2.17]{GagnaHarpazLanariGrayLaxFunctors} tells us that the map
\begin{equation}\label{e:gray-to-cart} 
\ovl{X} \mgr J_{\sharp} \to \ovl{X} \times J_{\sharp}
\end{equation}
from the Gray product to the cartesian product is a trivial cofibration. Since the thin triangles of \(X \mgr J_{\sharp}^{\flat}\) are also thin in \(\ovl{X} \times J_{\sharp}\) it follows that the map
\[ X \mgr J_{\sharp}^{\flat} \to \ovl{X} \times J_{\sharp}\]
is a pushout of~\eqref{e:gray-to-cart}, and is hence a trivial cofibration as well. It will hence suffice to check that \(\ovl{X} \times \{0\} \to \ovl{X} \times J_{\sharp}\) is a trivial cofibration, which is a consequence of the bicategorical model structure being cartesian.
\end{proof}

\begin{notate}
For a marked-scaled simplicial set \(X\),  
we denote by
\[ Y \mapsto \RMap(X,Y) \quad\text{and}\quad Y \mapsto \LMap(X,Y) \]
the right adjoints of the left Quillen functors of Corollary~\ref{c:marked-gray-quillen}.
To avoid confusion we point out that while \(X\) is a marked-scaled simplicial set, \(Y\),\(\RMap(X,Y)\) and \(\LMap(X,Y)\) are scaled simplicial sets. Explicitly, an \(n\)-simplex of \(\RMap(X,Y)\) is given by a map of scaled simplicial sets
\[
 \prescript{\flat}{}\Del^n \mgr X \lrar Y 
\]
and a \(2\)-simplex \(\prescript{\flat}{}\Del^2 \mgr X\lrar Y\) 
is thin if it factors through \(\prescript{\sharp}{}\Delta^2 \mgr X\). 
Similarly, an \(n\)-simplex of \(\LMap(X,Y)\) is given by a map of scaled simplicial sets
\[
 X \mgr \prescript{\flat}{}\Del^n \lrar Y 
\]
and the scaling is determined as above.
\end{notate}

\begin{rem}
It follows from Remark~\ref{r:unmarked} that if \(X,Y\) are scaled simplicial sets then \(\RMap(X^{\flat},Y)\) and \(\LMap(X^{\flat},Y)\) coincide with the scaled simplicial sets \(\RMap(X,Y)\) and \(\LMap(X,Y)\) recalled in \S\ref{s:scaled-gray}, and so this notation overloading should not cause any confusion.
\end{rem}

\begin{rem}\label{r:symmetric-2}
By Remark~\ref{r:symmetric} we have natural isomorphisms
\[ \RMap(X^{\op},Y^{\op}) \cong \LMap(X,Y)^{\op} \]
and
\[ \LMap(X^{\op},Y^{\op}) \cong \RMap(X,Y)^{\op} \]
\end{rem}

Being right adjoints to left Quillen functors, 
the functors
\[ \RMap(K,-),\LMap(K,-)\colon \Ss \to \Ss \]
are right Quillen functors for any scaled simplicial set \(K\)
with respect to the bicategorical model structure. In particular, if \(\C\) is an \(\infty\)-bicategory then \(\RMap(K,\C)\) and \(\LMap(K,\C)\) are \(\infty\)-bicategories. 
The objects of the \(\infty\)-bicategory \(\RMap(K,\C)\) correspond to functors \(\ovl{K} \lrar \C\), where we recall from Definition~\ref{d:underlying} that \(\ovl{K}\) stands for the underlying \emph{scaled} simplicial set of \(K\). If all the edges in \(K\) are marked then \(\RMap(K,\C) \simeq \LMap(K,\C)\) and both coincide with the \(\infty\)-bicategory \(\Fun(K,\C)\) of functors \(\ovl{K} \to \C\). 
On the other hand, if only the degenerate edges in \(K\) are marked then, as in \S\ref{s:scaled-gray}, the morphisms in \(\RMap(K,\C)\) correspond to \ndef{lax natural transformations}. 
Dually, in the case of \(\LMap(K,\C)\) we obtain functors and oplax natural transformations between them.

\subsection{The thick join and slice constructions}\label{s:thick-join}

\begin{define}\label{d:thick-join}
	Let \(X\) and \(Y\) be two marked-scaled simplicial sets. We define the \ndef{inner thick join} \(X \diamond_{\inn} Y \in \Ss\) by the formula
	\[ X \diamond_{\inn} Y = \ovl{X} \coprod_{X \mgr \Del^{\{0\}} \mgr Y} \bigl(X \mgr {}^{\flat}\Del^1 \mgr Y \bigr)\coprod_{X \mgr \Del^{\{1\}} \mgr Y} \ovl{Y} ,\]
	and the \ndef{outer thick join} \(X \diamond_{\out} Y \in \Ss\) by the formula
	\[ X \diamond_{\out} Y = \ovl{X} \coprod_{Y \mgr \Del^{\{0\}} \mgr X} \bigl(Y \mgr {}^{\flat}\Del^1 \mgr X \bigr)\coprod_{Y \mgr \Del^{\{1\}} \mgr X} \ovl{Y} .\]
	Here, we use triple Gray products as in Definition~\ref{d:gray-marked}. In particular, the input of the thick join consists of marked-scaled simplicial sets, while its output is a scaled simplicial set.
\end{define}

For a fixed marked-scaled simplicial \(K\) with underlying scaled simplicial set \(\ovl{K}\), we may consider the assignment
\[ 
X \mapsto X \diamond_{\inn} K\quad (\mathrm{resp.}\ X \mapsto X \diamond_{\out} K)
 \]
as a functor 
\[\Set^{+,\sca}_{\Del} \lrar \Set^{\sca}_{\ovl{K}/} .\] 
As such, it is a colimit preserving functor which admits a right adjoint 
\[\Set^{\sca}_{\ovl{K}/} \lrar \Set^{+,\sca}_{\Del}\] 
by the adjoint functor theorem. Given a map \(f\colon \ovl{K} \lrar S\) of scaled simplicial set, considered as an object of \(\Set^{\sca}_{\ovl{K}/}\), we will denote by \(S^{/f}_{\inn}\) (resp.~\(S^{/f}_{\out}\)) the marked-scaled simplicial set obtained by applying the above right adjoint. In particular, the marked-scaled simplicial sets \(S^{/f}_{\inn}\) and \(S^{/f}_{\out}\) are characterized by mapping properties of the form
\[
  \Hom_{\Sms}(X,S^{/f}_{\inn}) \cong \Hom_{(\Ss)_{\ovl{K}/}}(X \diamond_{\inn} K,S)
\]
and
 \[
  \Hom_{\Sms}(X,S^{/f}_{\out}) \cong \Hom_{(\Ss)_{\ovl{K}/}}(X \diamond_{\out} K,S),
 \]
respectively. In a similar manner, we may consider the right adjoint to the functor \(\Set^{+, \sca}_{\Del} \lrar \Set^{\sca}_{\ovl{K}/}\)
given by the assignment
\[
X \mapsto K \diamond_{\inn} X\quad (\mathrm{resp.}\ X \mapsto K \diamond_{\out} X) .
\]
For any map \(f\colon \ovl{K} \lrar S\) of scaled simplicial set, considered as an object of \(\Set^{\sca}_{\ovl{K}/}\), 
we will denote by \(S^{f/}_{\inn}\) (resp.~\(S^{f/}_{\out}\)) the marked-scaled simplicial set obtained by applying this right adjoint. 
In particular, the marked-scaled simplicial sets \(S^{f/}_{\inn}\) and \(S^{f/}_{\out}\) are characterized by mapping properties of the form
\[
  \Hom_{\Sms}(X,S^{f/}_{\inn}) \cong \Hom_{(\Ss)_{\ovl{K}/}}(K \diamond_{\inn} X,S)
\]
and
 \[
  \Hom_{\Sms}(X,S^{f/}_{\out}) \cong \Hom_{(\Ss)_{\ovl{K}/}}(K \diamond_{\out} X,S),
 \]
respectively. We will then denote by \(\ovl{S}^{/f}_{\inn},\ovl{S}^{/f}_{\out}, \ovl{S}^{f/}_{\inn}\) and \(\ovl{S}^{f/}_{\out}\) the underlying scaled simplicial sets of \(S^{/f}_{\inn},S^{/f}_{\out}, S^{f/}_{\inn}\) and \(S^{f/}_{\out}\), respectively.

\begin{rem}\label{r:opposites}
	By Remark~\ref{r:symmetric} we have canonical isomorphisms 
	\[
	 (X \diamond_{\inn} Y)^{\op} \cong Y^{\op} \diamond_{\inn} X^{\op}
	 \quad\text{and}\quad
	 (X \diamond_{\out} Y)^{\op} \cong Y^{\op} \diamond_{\out} X^{\op}.
	\] 
	As a result, if \(f\colon \ovl{K} \to S\) is a map of scaled simplicial sets then we have canonical isomorphisms 
	\((S^{f/}_{\inn})^{\op} \cong (S^{\op})^{/f^{\op}}_{\inn}\) and \((S^{f/}_{\out})^{\op} \cong (S^{\op})^{/f^{\op}}_{\out}\).
\end{rem}

\begin{rem}\label{r:diamond-quillen}
It follows from Corollary~\ref{c:marked-gray-quillen} that for a fixed marked-scaled simplicial set \(K\) the functors
\[ (-)^{\flat} \diamond_{\inn} K\colon \Ss \to (\Ss)_{\ovl{K}/} \quad\quad (-)^{\flat} \diamond_{\out} K \colon \Ss \to (\Ss)_{\ovl{K}/}\]
\[ K \diamond_{\inn} (-)^{\flat}\colon \Ss \to (\Ss)_{\ovl{K}/} \quad\text{and}\quad K \diamond_{\out} (-)^{\flat}\colon \Ss \to (\Ss)_{\ovl{K}/},\]
are left Quillen functors, where \((\Ss)_{\ovl{K}/}\) is considered with the slice model structure associated to the bicategorical model structure. It then follows that their right adjoints
\[ (\Ss)_{\ovl{K}/}  \ni [f\colon \ovl{K} \to S] \mapsto \ovl{S}^{/f}_{\inn}\in \Ss  \quad\quad (\Ss)_{\ovl{K}/} \ni [f\colon \ovl{K} \to S] \mapsto \ovl{S}^{/f}_{\out}\in \Ss,\]
\[ (\Ss)_{\ovl{K}/}  \ni [f\colon \ovl{K} \to S] \mapsto \ovl{S}^{f/}_{\inn}\in \Ss  \quad\text{and}\quad (\Ss)_{\ovl{K}/} \ni [f\colon \ovl{K} \to S] \mapsto \ovl{S}^{f/}_{\out}\in \Ss \]
are right Quillen functors. 
\end{rem}

It follows from Remark~\ref{r:diamond-quillen} that if \(\C\) is an \(\infty\)-bicategory, \(K\) is a marked-scaled simplicial set and \(f\colon \ovl{K} \to \C\) is a diagram, then \(\ovl{\C}^{/f}_{\inn},\ovl{\C}^{/f}_{\out},\ovl{\C}^{f/}_{\inn}\) and \(\ovl{\C}^{f/}_{\out}\) are \(\infty\)-bicategories.
If all the edges in \(K\) are marked then the objects of these 
\(\infty\)-bicategories correspond to pseudo-natural cones on \(f\), while if only the degenerate edges are marked they correspond to lax (or op-lax) cones. In general, we may consider them as \emph{partially lax cones}, with the amount of ``pseudo-naturality'' encoded by the collection of marked edges. Regardless of the marked edges, the morphisms in these slice \(\infty\)-bicategories always correspond to lax (or op-lax) transformation of cones, with the marked edges in 
\(\C^{/f}_{\inn},\C^{/f}_{\out},\C^{f/}_{\inn}\) and \(\C^{f/}_{\out}\)	respectively indicating those lax transformation which are pseudo-natural.

\begin{example}
Let \(\C\) be an \(\infty\)-bicategory. 
When \(K = \Del^0\) and \(f\colon \Del^0 \lrar \C\) is the inclusion of the vertex \(x \in \C\)
then we will denote \(\C^{f/}_{\inn}\) and \(\C^{f/}_{\out}\) by \(\C^{x/}_{\inn}\) and \(\C^{x/}_{\out}\) respectively, and similarly for \(\C^{/x}_{\inn}\) and \(\C^{/x}_{\out}\). In this case the vertices of \(\C^{x/}_{\inn}\) and \(\C^{x/}_{\out}\) are just arrows \(x \to y\) of \(\C\) with source \(x\). An edge in \(\C^{x/}_{\inn}\) from \(x \to y\) to \(x \to z\) is a diagram
	\[
	 \begin{tikzcd}
	 	x \ar[d, equal] \ar[r] \ar[rd, ""{name=d-up}, ""{swap, name=d-down}]&
	 	y  \ar[d] \\
	 	x \ar[r] & z
	 	\ar[from=d-down, to=1-2, phantom, "\simeq"]
	 	\ar[from=d-down, to=2-1, Rightarrow]
	 \end{tikzcd} .
	\]
	On the other hand, an edge of \(\C^{x/}_{\out}\) from \(x \to y\) to \(x \to z\) is a diagram
	\[
	 \begin{tikzcd}
	 	x \ar[d, equal] \ar[r] \ar[rd, ""{name=d-up}, ""{swap, name=d-down}]&
	 	y  \ar[d] \\
	 	x \ar[r] & z
	 	\ar[from=d-up, to=1-2, Rightarrow]
	 	\ar[from=d-up, to=2-1, phantom, "\simeq"]
	 \end{tikzcd}.
	\]
 \end{example}

\begin{example}\label{example:thick_slices}
	When \(K=\prescript{\flat}{}\Delta^1\) and \(f \colon \prescript{\flat}{}\Delta^1 \to \C\) is given by an edge \(e\colon x \to y\) in \(\C\) then we will denote \(\C^{f/}_{\inn}\) and \(\C^{f/}_{\out}\) by \(\C^{e/}_{\inn}\) and \(\C^{e/}_{\out}\) respectively, and similarly for \(\C^{/e}_{\inn}\) and \(\C^{/e}_{\out}\).
The vertices of \(\C^{e/}_{\inn}\) are then given by diagrams of the form
	\[
	 \begin{tikzcd}
	 	x \ar[d, "e"'] \ar[r] \ar[rd, ""{name=d-up}, ""{swap, name=d-down}]&
	 	z  \ar[d, equal] \\
	 	y \ar[r] & z
	 	\ar[from=d-up, to=1-2, Rightarrow]
	 	\ar[from=d-up, to=2-1, phantom, "\simeq"]
	 \end{tikzcd},
	\]
	while the vertices of \(\C^{e/}_{\out}\) are diagrams of the form
	\[
	 \begin{tikzcd}
	 	x \ar[r] \ar[d, "e"'] \ar[rd, ""{name=d-up}, ""{swap, name=d-down}]&
	 	z  \ar[d, equal] \\
	 	y \ar[r] & z
	 	\ar[from=d-down, to=1-2, phantom, "\simeq"]
	 	\ar[from=d-down, to=2-1, Rightarrow]
	 \end{tikzcd} .
	\]
\end{example}

\begin{rem}\label{r:mapping-space}
	Let \(\C\) be a \(\infty\)-bicategory. 
	For any vertex \(x\) in \(\C\) the underlying marked simplicial set of \(\C^{x/}_{\inn}\) 
	is isomorphic to the marked simplicial set denoted by \(\C^{x/}\) in~\cite[Notation 4.1.5]{LurieGoodwillie}. 
	In particular, the fiber \((\C^{x/}_{\inn})_y\) of the projection \(\C^{x/}_{\inn} \lrar \C\) over \(y \in \C\) 
	is the marked simplicial set \(\Hom_{\C}(x,y)\) used in~\cite{LurieGoodwillie} as a model for the mapping \(\infty\)-category from \(x\) to \(y\), see Notation~\ref{n:mapping}. 
	The same holds for the fiber of \(\C^{/y}_{\out}\) over \(x\) (since these two fibers are naturally isomorphic). 
	By Remark~\ref{r:opposites} these fibers are also isomorphic to \(((\C^{\op})^{/x}_{\inn})^{\op}_y\) and \(((\C^{\op})^{y/}_{\out})^{\op}_x\). Similarly, we have natural isomorphisms
	\[ (\C^{x/}_{\out})_y \cong (\C^{/y}_{\inn})_x \cong ((\C^{\op})^{y/}_{\inn})^{\op}_x \cong ((\C^{\op})^{/x}_{\out})^{\op}_y \cong \Hom_{\C^{\op}}(y,x)^{\op} \] 
	and the latter is categorically equivalent (though generally not isomorphic) as a marked simplicial set to \(\Hom_{\C}(x,y)^{\op}\) by Remark~\ref{r:dwyer-kan}.
\end{rem}

\begin{rem}\label{r:st-gray}
The thick join of Definition~\ref{d:thick-join} can be used to express the cone construction appearing in the definition of the scaled straightening functor, see \S\ref{s:straightening}. More precisely, if \((S,T_S)\) is a scaled simplicial set, \(\phi\colon \fCs(S,T_S) \to \C\) a Dwyer-Kan equivalence, \((X,E_X)\) is a marked simplicial set and \(f\colon X \to S\) is a map, then 
\[ [\Str_{\phi}(X,E_X)](v) = \Cone_{\phi}(X,E_X)(\ast,v),\]
where \(\Cone_{\phi}(X,E_X)\) can be described in terms of the thick join as
\[ \Cone_{\phi}(X,E_X) = \fCs(\Del^0 \diamond_{\inn} (X,E_X,\emptyset))\coprod_{\fCs(X_{\flat})} \C .\]
\end{rem}

\begin{rem}\label{r:unst-rep}
It follows from Remark~\ref{r:st-gray} that if \(\D\) is a fibrant \(\s^+\)-enriched category and \(x \in \D\) is an object, then the underlying marked simplicial set of \(\Nsc(\D)^{x/}_{\inn}\) is naturally isomorphic to the unstraightening with respect to the counit map 
\[\psi\colon\fCs(\Nsc(\D)) \to \D\] of the functor \(\D(x,-)\colon \D \to \s^+\) corepresented by \(\D\). Indeed, for a marked simplicial set \((X,E_X)\) and a map \(p\colon X_{\flat} \to \Nsc(\D)\) with adjoint \(p^{\ad}\colon \fCs(X_{\flat}) \to \D\), we get from the above that enriched natural transformations \(\St^{\sca}_{\psi}(X,E_X) \to \D(x,-)\) are in bijection with enriched functors
\[ \fCs(\Del^0 \diamond_{\inn} (X,E_X,\emptyset)) \to \D\]
which restrict to \(p^{\ad}\) on \(\fCs(X_{\flat})\) and send the cone point to \(x\). By adjunction, these correspond to maps \(\Del^0 \diamond_{\inn} (X,E_X,\emptyset) \to \Nsc(\D)\) extending \(p\) and sending the cone point to \(x\), and hence to maps \((X,E_X,\emptyset) \to \Nsc(\D)^{x/}_{\inn}\) over \(\D\).
\end{rem}

Let us now compare the outer thick join construction to the standard join construction. 

\begin{prop}
	\label{prop:join-eq}
For marked-scaled simplicial sets \(X,Y\) let	
\begin{equation}\label{e:thin-is-thick}
r\colon X \diamond_{\out} Y \lrar X \ast Y
\end{equation} 
be the unique map which is compatible with the canonical inclusions
\(X \hrar X \diamond_{\out} Y \hookleftarrow Y\) and 
\(X \hrar X \ast Y \hookleftarrow Y\).
Then \(r\) is a bicategorical equivalence.
\end{prop}

Before we give the proof of Proposition~\ref{prop:join-eq} let us take a minute to verify that the map \(r\) is indeed well-defined. First, on the level of underlying simplicial sets, note that \(n\)-simplices of \(X \ast Y\) corresponds to a triple \((i,\sig_{i-1},\sig_{n-i})\) where \(i \in \{0,...,n+1\}\) corresponds to a choice of a partition \([n] = \{0,...,i-1\} \ast \{i,...,n\}\) (the two extreme options corresponding to partitions in which one part is empty), \(\sig_{i-1}\) is an \(\{0,...,i-1\}\)-simplex of \(X\), and \(\sig_{n-i}\) is an \(\{i,...,n\}\)-simplex of \(Y\). By convention the set of \(\{\}\)-simplices of any simplicial set is a singleton, and these extreme partitions correspond to the simplices of \(X \ast Y\) which are in the image of the inclusions \(X \hrar X \ast Y \hookleftarrow Y\). On the other hand, an \(n\)-simplex of \(X \diamond_{\out} Y\) is given by an equivalence class of triples \((\rho_Y,\tau,\rho_X)\), of \(n\)-simplices of \(Y\), \(\Del^1\), and \(X\), respectively. The \(n\)-simplex \(\tau\colon \Del^n \to \Del^1\) then determines a partition \([n] = \tau^{-1}(0) \ast \tau^{-1}(1)\), and the map \(r\) sends \((\rho_Y,\tau,\rho_X)\) to \((\min(\tau^{-1}(1)),{\rho_X}_{|\tau^{-1}(0)},{\rho_Y}_{|\tau^{-1}(1)})\). One may then verify that this is the only option that is compatible with the simplicial face and degeneracy maps, and that behaves in the prescribed manner when the partition is one of the two extreme cases. It is clear from this description that the map \(r\) is surjective on \(n\)-simplices for every \(n\).

Let us now verify that this map sends thin triangles of \(X \diamond_{\out} Y\) to thin triagles of \(X \ast Y\). Unwinding the definitions we see that a triangle given by a class of a triple of triangles \((\rho_Y,\tau,\rho_X)\) is thin in \(X \diamond_{\out} Y\) if and only if one of the following possibilities hold:
\begin{itemize}
\item
\(\tau\) sends all vertices to \(0\) and \(\rho_X\) is thin in \(X\).
\item
\(\tau\) sends all vertices to \(1\) and \(\rho_Y\) is thin in \(Y\).
\item
\(\tau\) is given on vertices by \(0,1,2 \mapsto 0,0,1\), \(\rho_Y\) is thin, \(\rho_X\) is degenerate and \({\rho_X}_{|\Del^{\{0,1\}}}\) is marked in \(X\).
\item
\(\tau\) is given on vertices by \(0,1,2 \mapsto 0,1,1\), \(\rho_X\) is thin, \(\rho_Y\) is degenerate and \({\rho_Y}_{|\Del^{\{1,2\}}}\) is marked in \(Y\).
\end{itemize}
Examining Definition~\ref{d:join} we then see that in all these cases the associated triangle \((\min(\tau^{-1}(1)),{\rho_X}_{|\tau^{-1}(0)},{\rho_Y}_{|\tau^{-1}(1)})\) is thin in \(X \ast Y\). In addition, one directly verifies that the map \(r\) is surjective on thin triangles.

The proof of Proposition~\ref{prop:join-eq} will require the next lemma. In what follows, for integers \(p,q \geq 0\) we write 
\[ \Del^p \diamond_{\out} \Del^q := \Del^p \coprod_{\Del^q \times \Del^{\{0\}} \times \Del^p} \bigl(\Del^q \times \Del^1 \times \Del^p \bigr)\coprod_{\Del^q \times \Del^{\{1\}} \times \Del^p} \Del^q \]
for the underlying simplicial set of \({}^{\flat}\Del^p \diamond_{\out} {}^{\flat}\Del^q\).
\begin{lemma}\label{lem:join-eq}
Let \(T\) denote the collection of triangles in \(\Del^p \diamond_{\out} \Del^q\) which are thin in \(\prescript{\flat}{}\Del^p \diamond_{\out} \prescript{\flat}{}\Del^q\), and \(T'\) the collection of all triangles in \(\Del^p \diamond_{\out} \Del^q\) whose image in \(\Del^p \ast \Del^q\) is degenerate. Then \(T \subseteq T'\) and the inclusion 
\begin{equation}\label{e:add-all}
(\Del^p \diamond_{\out} \Del^q,T) \subseteq (\Del^p \diamond_{\out} \Del^q,T')
\end{equation}
is scaled anodyne.
\end{lemma}
\begin{proof}
We will denote the vertices of \(\Del^p \diamond_{\out} \Del^q\) 
by triples \([v,\eps,u]\) with \(v \in [q], \eps \in [1],u \in [p]\), under the equivalence relation in which \([v,0,u] \sim [v',0,u]\) for every \(v,v' \in [q],u \in [p]\) and \([v,1,u] \sim [v,1,u']\) for every \(v \in [q],u,u'\in [p]\) .
Using the identification \(\Delta^p \ast \Delta^q \cong \Delta^{p+q+1}\) on the level of the underlying simplicial sets, the map
\(\Del^p \diamond_{\out} \Del^q \to \Delta^p\ast\Del^q\)
can be written on vertices by 
\[
	r([v,\eps,u])=
	\begin{cases}
	 u & \eps=0\\
	 v+p+1 & \eps=1
	\end{cases}
\]
We note that all the thin triangles in \(\prescript{\flat}{}\Del^p\ast \prescript{\flat}{}\Del^q\) are degenerate, while the \emph{non-degenerate} thin triangles of \(\prescript{\flat}{}\Del^p\diamond_{\out} \prescript{\flat}{}\Del^q\) are given by the classes of those triples \((\rho_{\Del^q},\tau,\rho_{\Del^p})\) such that \(\tau\) is surjective, \(\rho_{\Del^q}\) degenerates along \(\Del^{\{1,2\}}\) and \(\rho_{\Del^p}\) degenerates along \(\Del^{\{0,1\}}\). 
In particular, these all map to degenerate triangles in \(\Del^p \ast \Del^q\), and so we have \(T \subseteq T'\). On the other hand, the non-degenerate triangles in \(T'\) are given by the classes of those triples \((\rho_{\Del^q},\tau,\rho_{\Del^p})\) such that \(\tau\) is surjective, and such that either both \(\tau\) and \(\rho_{\Del^q}\) degenerate along \(\Del^{\{1,2\}}\) or both \(\tau\) and \(\rho_{\Del^p}\) degenerate along \(\Del^{\{0,1\}}\). 

We will now show that for every non-degenerate triangle \(\sig \in T'\) there is a \(3\)-simplex \(\eta\colon \Del^3 \to \Del^p \diamond_{\out} \Del^q\) and an \(i \in \{1,2\}\) such that \(\eta_{|\Del^{\{0,i,3\}}} = \sig\) while the three other faces of \(\eta\) lie in \(T\). This will imply that~\eqref{e:add-all} is a sequence of pushouts along maps of the form \((\Del^3,T_i) \to (\Del^3)_{\sharp}\), where \(T_i\) denotes all triangles except \(\Del^{\{0,i,3\}}\), and is hence scaled anodyne by~\cite[Remark 3.1.4]{LurieGoodwillie}. 
Now if we take a triangle in \(T'\) of the form \((\rho_{\Del^q},\tau,\rho_{\Del^p})\) such that both \(\tau\) and \(\rho_{\Del^q}\) degenerate along \(\Del^{\{1,2\}}\) then we let \(\eta\) be the \(3\)-simplex \((\rho_{\Del^q} \circ \alp,\tau \circ \beta,\rho_{\Del^p} \circ \gam)\), 
where \(\alp\colon \Del^3 \to \Del^2\) is given on vertices by \(0,1,2,3 \mapsto 0,1,1,2\), \(\beta\colon \Del^3 \to \Del^2\) is given on vertices by \(0,1,2,3 \mapsto 0,1,1,2\) and \(\gam\colon \Del^3 \to \Del^2\) is given on vertices by \(0,1,2,3 \mapsto 0,0,1,2\).
Notice that the restriction of \(\rho_{\Del^p} \circ \gam\) to \(\Del^{\{1, 2, 3\}}\)
is quotiented to the point in \(\Del^p \diamond_{\out} \Del^q\).
Similarly, if we take a triangle in \(T'\) of the form \((\rho_{\Del^q},\tau,\rho_{\Del^p})\) such that both \(\tau\) and \(\rho_{\Del^p}\) degenerate along \(\Del^{\{0,1\}}\) then we let \(\eta\) be the \(3\)-simplex \((\rho_{\Del^q} \circ \alp,\tau \circ \beta,\rho_{\Del^p} \circ \gam)\), where \(\alp\colon \Del^3 \to \Del^2\) is given on vertices by \(0,1,2,3 \mapsto 0,1,2,2\), \(\beta\colon \Del^3 \to \Del^2\) is given on vertices by \(0,1,2,3 \mapsto 0,1,1,2\) and \(\gam\colon \Del^3 \to \Del^2\) is given on vertices by \(0,1,2,3 \mapsto 0,1,1,2\).
Notice that the restriction of \(\rho_{\Del^q} \circ \alp\) to \(\Del^{\{0, 1, 2\}}\)
is quotiented to the point in \(\Del^p \diamond_{\out} \Del^q\).
\end{proof}

\begin{proof}[Proof of Proposition~\ref{prop:join-eq}]
Let us say that a pair \((X,Y)\) is \ndef{good} if the map~\eqref{e:thin-is-thick} is an equivalence. We then observe that for a fixed \(X\) the operations \(Y \mapsto X \diamond_{\out} Y\) and \(Y \mapsto X \ast Y\) preserve monomorphisms, pushout squares and filtered colimits. On the other hand, pushout squares with parallel legs cofibrations are always homotopy pushout squares (since all objects in \(\Ss\) are cofibrant), and filtered colimits are always homotopy colimits since bicategorical equivalences are closed under filtered colimits. 
We then conclude that for a fixed \(X\), the collection of \(Y\) for which \((X,Y)\) is good is closed under pushouts with parallel legs cofibrations and filtered colimits. It will hence suffice to prove the claim for \(Y=\prescript{\flat}{}\Del^q\), \(Y=\Del^2_{\sharp}\) and \(Y=(\Del^1)^{\sharp}\). Applying this argument for \(X\) instead of \(Y\) we may equally suppose that \(X\) is either \(\prescript{\flat}{}\Del^p, \Del^2_{\sharp}\) or \((\Del^1)^{\sharp}\). We now observe that for every marked-scaled simplicial set 
there are pushout squares of scaled simplicial sets
\[ \xymatrix{
X \diamond_{\out} \prescript{\flat}{}\Del^2 \ar[d]\ar[r] & X \diamond_{\out} (\Del^2_{\sharp})^{\flat} \ar[d] \\
X \ast \prescript{\flat}{}\Del^2 \ar[r] & X \ast (\Del^2_{\sharp})^{\flat} \\
}
\quad\quad
\xymatrix{
X \diamond_{\out} \prescript{\flat}{}\Del^1 \ar[d]\ar[r] & X \diamond_{\out} (\Del^1_{\flat})^{\sharp} \ar[d] \\
X \ast \prescript{\flat}{}\Del^1 \ar[r] & X \ast (\Del^1_{\flat})^{\sharp} \ .
}\]
Indeed, since the horizontal maps are isomorphisms on the level of underlying simplicial sets, this follows from the fact that the vertical maps are surjective on thin triangles. These squares are then also homotopy pushout squares with respect to the bicategorical model structure since they have parallel legs cofibrations and all objects cofibrant. We may consequently assume without loss of generality that
\(X=\prescript{\flat}{}\Del^p\) and \(Y=\prescript{\flat}{}\Del^q\). 
In light of Lemma~\ref{lem:join-eq} it will now suffice to show that the map \((\Del^p \diamond_{\out} \Del^q,T') \to (\Del^p \ast \Del^q)_{\flat}\) is a bicategorical equivalence.

As in the proof of Lemma~\ref{lem:join-eq} we will denote the vertices of \(\Del^p \diamond_{\out} \Del^q\) by triples \([v,\eps,u]\) with \(v \in [q], \eps \in [1],u \in [p]\), 
under the equivalence relation \([v,0,u] \sim [v',0,u]\) for every \(v,v' \in [q],u \in [p]\) and \([v,1,u] \sim [v,1,u']\) for every \(v \in [q],u,u'\in [p]\), so that 
the map
\(\Del^p \diamond_{\out} \Del^q \to \Delta^p\ast\Del^q\)
can be written on vertices by 
\[
	r([v,\eps,u])=
	\begin{cases}
	 u &,\ \eps=0,\\
	 v+p+1 &,\ \eps=1.
	\end{cases}
\]
We now define maps of scaled simplicial sets
\[
 \begin{tikzcd}
	& (\Del^q \times \Del^1 \times \Del^p,T'') \ar[dr] & \\
	(\Delta^p\ast\Del^q)_{\flat} \ar[ur, "\wtl{s}"] \ar[rr, "s"] &&  (\Del^p\diamond_{\out}\Del^q,T')
 \end{tikzcd}
\]
where \(T''\) is the preimage of \(T'\) and  
\(\wtl{s}\) is induced by the order preserving map on vertices
\[
  \wtl{s}(i)=
  \begin{cases}
	[0,0,i] &,\ i\leq p,\\
	[i-p-1,1,p] &,\ i >p.
  \end{cases}
\]
It is immediate to check that \(rs=\Id\). On the other hand, if we denote by 
\[\wtl{r}\colon (\Del^q \times \Del^1 \times \Del^p,T'') \to (\Del^p \diamond_{\out} \Del^q,T') \to (\Del^p \ast \Del^q)_{\flat}\]
so that the composite \(\wtl{s}\circ \wtl{r}\)
is given on vertices by the order preserving map
\[
  \wtl{s}\circ \wtl{r}([v,\eps,u])=
  \begin{cases}
	[0,0,u] &,\ \eps=0,\\
	[v,1,p] &,\ \eps=1.
  \end{cases}
\]
We will now exhibit a zig-zag of natural transformations from \(sr\) to the identity. 
More precisely, we construct a map of scaled simplicial sets 
\[u\colon (\Del^p \diamond_{\out} \Del^q,T') \to (\Del^p \diamond_{\out} \Del^q,T')\] 
and a pair of transformations
\[h,k\colon (\Del^p \diamond_{\out} \Del^q,T') \times \Del^1_{\flat}\to (\Del^p \diamond_{\out}\Del^q,T') \] 
such that
\[
 h_{|\{0\}}=k_{|\{0\}}=u,\quad \ h_{|\{1\}}=sr \quad\text{and}\quad k_{|\{1\}}=\Id.
\]
Furthermore, these transformations will satisfy the additional property that \(h_{\vert x\times \Del^1}\) and \(k_{\vert x\times \Del^1}\) are degenerate in \(\Del^p \diamond_{\out} \Del^q\) for every vertex \(x\) of \(\Del^p \diamond_{\out} \Del^q\). The map~\eqref{e:thin-is-thick} is then a bicategorical equivalence by~\cite[Corollary~7.8]{GagnaHarpazLanariEquiv}.

We now finish the proof by constructing \(h,k\) and \(u\) as above. For \(u\), we define it to be the map
induced on quotients by the endomorphism \(\wtl{u}\) of \((\Del^q \times \Del^1 \times \Del^p,T'')\)
given on vertices by the order preserving map 
\[\wtl{u}([v,\eps,u])=\begin{cases}
[0,0,u] &,\ \eps=0,\\
[v,1,u] &,\ \eps=1.
\end{cases}\]
This order preserving map satisfies pointwise the inequalities \(\wtl{u}([v,\eps,u]) \leq [v,\eps,u]\) and \(\wtl{u}([v,\eps,u]) \leq \wtl{s}\wtl{r}([v,\eps,u])\). Therefore there are natural transformations of simplicial sets
\[
\wtl{h},\wtl{k}\colon (\Del^q \times \Del^1 \times \Del^p) \times \Del^1 \to \Del^q \times \Del^1 \times \Del^p
\]
with \(\wtl{h}_{|\{0\}}=\wtl{k}_{|\{0\}}=u\) and \(\wtl{h}_{|\{1\}}=\wtl{s}\circ\wtl{r}\),
\(\wtl{k}_{|\{1\}}=\Id\). 
Both these homotopies have the property that when projected down to \(\Del^p \ast \Del^q\) they yield the identify transformation from \(r\) to itself. Since \(T'\) consists by definition of those triangles whose image in \(\Del^p \ast \Del^q\) is degenerate we see that \(\wtl{h}\) and \(\wtl{k}\) refine to scaled maps
\[
 \wtl{h},\wtl{k}\colon(\Del^q \times \Del^1 \times \Del^p,T'')\times \Del^1_{\flat} \to (\Del^q \times \Del^1 \times \Del^p,T'').
\]
Finally, by direct inspection they also pass to the quotient, so we get the desired maps
\[
 h,k\colon (\Del^p\diamond_{\out} \Del^q,T') \times \Del^1_{\flat} \to (\Del^p\diamond_{\out} \Del^q,T').
\]
We are left with checking that both \(h\) and \(k\) are constant along all edges 
of the form \(w\times \Del^1\) for \(w\in \Del^p\diamond_{\out} \Del^q\). We have four distinct cases to analyze:
\begin{itemize}
\item the edge \(\wtl{h}_{\vert [v,0,u]\times \Del^1}\) is the degenerate edge on \([0,0,u]\);
\item the edge \(\wtl{h}_{\vert [v,1,u]\times \Del^1}\) is \([v,1,u]\to [v,1,p]\), whose image in \(\Del^p \diamond_{\out} \Del^q\) is degenerate;
\item the edge \(k_{\vert [v,0,u]\times \Del^1}\) is the edge \([0,0,u] \to [v,0,u]\), whose image in \(\Del^p \diamond_{\out} \Del^q\) is degenerate;
\item the edge \(k_{\vert [v,1,u]\times \Del^1}\) is the degenerate edge on \([v,1,u]\).
\end{itemize}
We may finally conclude that the map \(r\) is bicategorical equivalence.
\end{proof}

As a first corollary of Proposition~\ref{prop:join-eq} we obtain the following analogue of
Lemma~\ref{lem:join-eq} for the thick outer join: 
\begin{cor}\label{cor:thick-pushout-join}
	Let \(f\colon X \lrar Y\) and \(g\colon A \lrar B\) be injective map of marked-scaled simplicial sets. If either \(f\) is outer cartesian anodyne or \(g\) is outer cocartesian anodyne then the map of scaled simplicial sets
	\begin{equation}\label{e:thick-pushout-join} 
	[X \diamond_{\out} B] \coprod_{X \diamond_{\out} A} [Y \diamond_{\out} A]\lrar Y \diamond_{\out} B 
	\end{equation}
	is a bicategorical trivial cofibration.
\end{cor}
\begin{proof}
The map~\eqref{e:thick-pushout-join} is clearly a cofibration and so it will suffice to show that it is a bicategorical equivalence. This follows from the analogous claim for the \(\ast\)-pushout-product in Lemma~\ref{l:pushout-join}, together with the comparison of Proposition~\ref{prop:join-eq}.
\end{proof}

\begin{cor}\label{c:outer-slice}
	Let \(\C\) be an \(\infty\)-bicategory and \(f\colon \ovl{K} \lrar \C\) a map of scaled simplicial sets. 
	Then we have a bicategorical equivalence
	\[
	 \begin{tikzcd}[column sep=small]
		\ovl{\C}_{/f} \ar[rr, "\simeq"]\ar[dr, "p"'] && \ovl{\C}^{/f}_{\out} \ar[dl, "q"] \\
		& \C &
	 \end{tikzcd}
	\]
	of \(\infty\)-bicategories over \(\C\). 
	In addition, \(q\) is an outer cartesian fibration which is fiberwise equivalent to \(p\). 
	Similarly, we have a bicategorical equivalence of the form
	\[
	 \begin{tikzcd}[column sep=small]
		\ovl{\C}_{f/} \ar[rr, "\simeq"]\ar[dr, "p"'] && \ovl{\C}^{f/}_{\out} \ar[dl, "q"] \\
		& \C &
	 \end{tikzcd}
	\]
	which is also a fiberwise equivalence of outer cocartesian fibrations over~\(\C\).
\end{cor}
\begin{proof}
By Remark~\ref{r:diamond-quillen} for every marked-scaled simplicial set \(K\) the functors
\[ (-)^{\flat} \diamond_{\out} K\colon \Ss \to (\Ss)_{\ovl{K}/} \quad\text{and}\quad K \diamond_{\out} (-)^{\flat}\colon \Ss \to (\Ss)_{\ovl{K}/} \]
are left Quillen functors. On the other hand, the functors
\[ (-)^{\flat} \ast K\colon \Ss \to \Ss \quad\text{and}\quad K \ast (-)^{\flat}\colon \Ss \to \Ss \]
preserves colimits and cofibrations, and hence also trivial cofibration by the comparison of \(\diamond\) and \(\ast\) of Proposition~\ref{prop:join-eq}. We may hence consider the natural transformation appearing in that proposition as a transformation between two left Quillen functors, which is then shown to be a levelwise weak equivalence. By~\cite[Corollary~1.4.4(b)]{HoveyModelCategories} we may conclude that the adjoint transformations \(\ovl{\C}_{/f} \to \ovl{\C}^{/f}_{\out}\) and \(\ovl{\C}_{f/} \to \ovl{\C}^{f/}_{\out}\) between the corresponding right adjoints are bicategorical equivalences whenever \(\C\) is an \(\infty\)-bicategory. In light of Corollary~\ref{c:fiberwise} it will now suffice to show that \(\ovl{\C}^{/f}_{\out} \to \C\) is an outer cartesian fibration and \(\ovl{\C}^{f/}_{\out} \to \C\) is an outer cocartesian fibration. Both these claims follow as in Corollary~\ref{cor:slice} from Corollary~\ref{cor:thick-pushout-join}.
\end{proof}

\begin{prop}\label{p:cocar}
	Let \(K\) be a marked-scaled simplicial set, \(\C\) an \(\infty\)-bicategory and \(f\colon \ovl{K} \lrar \C\) a map of scaled simplicial sets. 
	Then the maps \(\ovl{\C}^{/f}_{\inn} \lrar \C\) and \(\ovl{\C}^{f/}_{\inn} \lrar \C\) are inner cartesian and cocartesian fibrations respectively.
\end{prop}
\begin{proof}
We prove the claim for \(\ovl{\C}^{f/}_{\inn}\). The case of \(\ovl{\C}^{/f}_{\inn}\) then follows by taking opposites, see Remark~\ref{r:opposites}.
Let \([f]\colon \Del^0 \lrar \LMap(K,\C)\) be the map corresponding to \(f\). By Remarks~\ref{r:unmarked} and~\ref{r:almost-associative} we have for a scaled simplicial set \(Z\) an isomorphism \(K \mgr (\Del^1_{\flat} \mgr Z)^{\flat} \cong K \mgr {}^{\flat} \Del^1 \mgr Z^{\flat}\), which induces an isomorphism
\[ K \mgr (\Del^0 \diamond_{\inn} Z)^{\flat} \coprod_{K \mgr Z^{\flat}} Z \cong
K \diamond_{\inn} Z^{\flat}\]
of functors \(\Ss \to (\Ss)_{\ovl{K}/}\) natural in \(Z\). Passing to right adjoints, we obtain a pullback square
	\begin{equation}
	\label{slice iso}
		\begin{tikzcd}
		\ovl{\C}^{f/}_{\inn} \ar[r]\ar[d] & \ovl{\LMap}(K,\C)^{[f]/}_{\inn} \ar[d] \\
		\C \ar[r] & \LMap(K,\C)
	 \end{tikzcd}
	\end{equation}

where the top right corner stands for the underlying scaled simplicial set of the marked-scaled simplicial set \(\LMap(K,\C)^{[f]/}_{\inn}\).
Replacing \(\C\) with \(\LMap(K,\C)\) we may hence assume without loss of generality that \(K = \Del^0\) and \(f\colon \ovl{K} \to \C\) is given by a vertex \(x \in \C\). 
	
	By Remark~\ref{r:mapping-space} and~\cite[Proposition 4.1.6]{LurieGoodwillie} the underlying marked simplicial set of \(\C^{x/}_{\inn}\) constitutes a \(\Beta_{\C}\)-fibrant object
	of \(\trbis{(\s^+)}{\C}\). 
	Applying Proposition~\ref{p:inner-fibred} and Remark~\ref{rem:weaker} it will hence suffice to prove that the thin triangles in \(\C^{x/}_{\inn}\) are exactly those whose image in \(\C\) is thin. This follows from Lemma~\ref{l:lem-for-cocar} just below.
\end{proof}

\begin{lemma}\label{l:lem-for-cocar}
Let \(T\) denote the collection of those triangles in \(\Del^1 \times \Del^2\) which are either thin in \(\Del^1_{\flat} \mgr \Del^2_{\flat}\) or are contained in \(\partial \Del^1 \times \Del^2\). Then the inclusion
\[ (\Del^1 \times \Del^2,T) \to \Del^1_{\flat} \otimes \Del^2_{\sharp} \]
is scaled anodyne.
\end{lemma}
\begin{proof}
Direct inspection shows that the only thin \(2\)-simplex of \(\Del^1 \mgr \Del^2_{\sharp}\) which is not in \(T\) 
is the triangle \(\sig = \Del^{\{(0,0),(1,1),(1,2)\}}\). 
Let \(\rho\colon \Del^3 \lrar \Del^1 \times \Del^2\) be the \(3\)-simplex spanned by \((0,0),(1,0),(1,1),(1,2)\). Then \(\rho\) sends the triangles \(\Del^{\{0,1,2\}}\), \(\Del^{\{0,1,3\}}\) and \(\Del^{\{1,2,3\}}\) to triangles in \(T\), while \(\Del^{\{0,2,3\}}\) maps to \(\sig\). The desired map is then scaled anodyne by~\cite[Remark 3.1.4]{LurieGoodwillie}. 
\end{proof}

\subsection{Representable fibrations}
\label{s:representable}

Given an \(\infty\)-bicategory \(\C\), we may then construct a model for the Yoneda embedding of \(\C\) by picking a fibrant \(\s^+\)-category \(\D\) equipped with a bicategorical equivalence \(\eta\colon \C \xrightarrow{\simeq} \Nsc(\D)\), and considering the composed functor
\[ j_{\C}\colon \C^{\op} \xrightarrow{\simeq} \Nsc(\D^{\op}) \xrightarrow{\Nsc(j_{\D})} \Nsc[(\s^+)^{\D}]^{\circ} \xrightarrow{\simeq} \Fun(\C,\Catoo),\]
where \(j_{\D}\colon \D^{\op}\to (\s^+)^{\D}\) is the enriched Yoneda embedding of \(\C\) (which takes values in \(\Nsc[(\s^+)^{\D}]^{\circ}\) when \(\D\) is fibrant), and the last map is the bicategorical equivalence of Proposition~\ref{p:rectification}. The functor \(j_{\C}\) can then be morally described as sending \(x \in \C\) to the functor \(\Hom_{\C}(x,-)\colon \C^{\op} \to \Catoo\) corepresented by \(x\). Since \(j_{\D}\) is fully-faithful in the enriched sense we have that \(j_{\C}\) is fully-faithful in the bicategorical sense. 

By definition, the Yoneda image \(j_{\C}(x)\) of an object \(x\) is the vertex of \(\Fun(\C,\Catoo)\) determined by the enriched functor \(j_{\C}(\eta(x)) = \D(\eta(x),-)\colon \D \to \s^+\). We may also encode the latter by unstraightening it to an inner cocartesian fibration over \(\C\).
By Remark~\ref{r:unst-rep} and the compatibility of unstraightening with base change we see that the this inner cocartesian fibration is given explicitly by the fibration
\[ \C \times_{\Nsc(\D)} \ovl{\Nsc}(\D)^{x/}_{\inn}  \to \C ,\]
where \(\ovl{\Nsc}(\D)^{x/}_{\inn}\) is the underlying scaled simplicial set of \(\Nsc(\D)^{x/}_{\inn}\).
Since \(\eta\colon \C \to \Nsc(\D)\) is a bicategorical equivalence the induced map 
\[
 \begin{tikzcd}
	\ovl{\C}^{x/}_{\inn}\ar[dr]\ar[rr, "\simeq"] && \C \times_{\Nsc(\D)} \ovl{\Nsc}(\D)^{x/}_{\inn} \ar[dl] \\
	& \C &
 \end{tikzcd}
\]
is an equivalence of inner cocartesian fibrations over \(\C\) by Corollary~\ref{c:fiberwise} and Remark~\ref{r:mapping-space}, and hence we conclude that the functor corepresented by \(x\) classifies the associated inner slice fibration.
In this section we will see that the analogous statements hold for all four types of slice fibrations
\[
\begin{tikzcd}[column sep=tiny]
	\ovl{\C}^{x/}_{\inn} \ar[drrr] && \ovl{\C}^{x/}_{\out}\ar[dr] && \ovl{\C}^{/x}_{\inn} \ar[dl] && \ovl{\C}^{/x}_{\out}\ar[dlll] \\
	&&& \C &&&
\end{tikzcd}.
\]
We will do this by showing that they all admit the same type of a universal property, exhibiting them as \emph{freely generated} by \(x \in \C\). We will deduce from this that the \((\ZZ/2)^2\)-symmetry of \(\BiCat^{\thi}\) switches between these four fibrations, and consequently that they are all classified by the functors (co)represented by \(x\), with respect to the appropriate variance flavor.

To facilitate the discussion, let us work with a variable \(\var \in \{\out,\inn\}\), which we will call the \emph{variance parameter}. We will then refer to inner/outer (co)cartesian fibrations as \(\var\)-(co)cartesian fibrations.

\begin{notate}\label{n:fun-car}
Let \(p\colon X \to Y\) be a bicategorical fibration of scaled simplicial sets and \(K\) a marked-scaled simplicial set equipped with a map \(f\colon \ovl{K} \to Y\). We will denote by 
\[\Fun^{\car}_{Y}(K,X),\Fun^{\coc}_{Y}(K,X) \subseteq \Fun^{\thi}_{Y}(\ovl{K},X) = \Fun^{\thi}(\ovl{K},X) \times_{\Fun^{\thi}(\ovl{K},Y)} \{f\} \]
the full subcategory spanned by those maps \(\ovl{K} \to X\) over \(Y\) which send the marked edges of \(K\) to \(p\)-cartesian (resp.~\(p\)-cocartesian) edges of \(X\).
\end{notate}

\begin{rem}\label{r:restrict-fibration}
In the situation of Notation~\ref{n:fun-car}, the assumption that \(p \colon X \to Y\) is a bicategorical fibration implies that \(\Fun(\ovl{K},X) \to \Fun(\ovl{K},Y)\) is a bicategorical fibration, and hence that \(\Fun_Y(\ovl{K},X) := \Fun(\ovl{K},X) \times_{\Fun(\ovl{K},Y)} \{f\}\) is an \(\infty\)-bicategory. We may then identify \(\Fun^{\thi}_{Y}(\ovl{K},X)\) with the core \(\infty\)-category of \(\Fun_Y(\ovl{K},X)\).
More generally, if \(K \hrar L\) is an inclusion of marked-scaled simplicial sets then 
\[\Fun_Y(\ovl{L},X) \to \Fun_Y(\ovl{K},X)\]
is a bicategorical fibration of \(\infty\)-bicategories, and hence
\[\Fun^{\thi}_Y(\ovl{L},X) \to \Fun^{\thi}_Y(\ovl{K},X)\]
is a categorical fibration of \(\infty\)-categories. Since the condition of being a (co)cartesian edge is closed under equivalences (see Remark~\ref{r:car-invariant}) it follows that
\[\Fun^{\car}_Y(L,X) \to \Fun^{\car}_Y(K,X)
 \quad\text{and}\quad
 \Fun^{\coc}_Y(L,X) \to \Fun^{\coc}_Y(K,X)
\]
are categorical fibrations as well.
\end{rem}

\begin{define}\label{d:var-equivalence}
Let \(Y\) be a scaled simplicial set, \(h\colon K \to L\) a map of marked-scaled simplicial sets and \(f\colon \ovl{L} \to Y\) a map of scaled simplicial sets. For a variance paramter \(\var \in \{\inn,\out\}\), we will say that \(h\) is a \emph{\(\var\)-cartesian equivalence} over \(Y\) if for every \(\var\)-cartesian fibration \(p\colon X \to Y\) the restriction map
\begin{equation}\label{e:restrict-car} 
h^*\colon\Fun^{\car}_{Y}(L,X) \to \Fun^{\car}_{Y}(K,X) 
\end{equation}
is an equivalence of \(\infty\)-categories. Similarly, we define \(\var\)-cocartesian equivalence in the same manner using mapping \(\infty\)-categories into \(\var\)-cocartesian fibrations over~\(Y\).
\end{define}

\begin{example}\label{ex:equiv-is-cofinal}
In the situation of Definition~\ref{d:var-equivalence}, if the induced map \(\ovl{h}\colon \ovl{K} \to \ovl{L}\) is a bicategorical equivalence and the marked edges in \(L\) are the images of the marked edges in \(K\) then \(h\) is both a \(\var\)-cartesian and a \(\var\)-cocartesian equivalence over \(X\). Indeed, in this case the map~\eqref{e:restrict-car} is a base change of the trivial fibration
\[\Fun^{\thi}_Y(\ovl{L},X) \to \Fun^{\thi}_Y(\ovl{K},X).\]
\end{example}

\begin{rem}\label{r:cofinal-is-equivalence}
In the situation of Definition~\ref{d:var-equivalence}, suppose that \(Y\) is an \(\infty\)-bicategory and that 
\[\xymatrix{
\ovl{K} \ar[rr]^{\ovl{h}}\ar[dr] && \ovl{L} \ar[dl] \\
& Y & \\
}\]
is a map of \(\var\)-cartesian fibrations over \(Y\) such that the marked edges of \(K\) and \(L\) are exactly the cartesian edges over \(Y\). Then \(h\colon K \to L\) is a \(\var\)-cartesian equivalence over \(Y\) if and only if \(\ovl{h}\) is an equivalence in the sub-bicategory \(\Car^{\var}(Y)\) of \((\BiCat)_{/Y}\), which is the same as saying that \(\ovl{h}\) is a bicategorical equivalence.
\end{rem}

We now turn to establishing the universal properties of the slice fibrations over an object. We start with the inner cocartesian case, which follows directly from the work of~\cite{LurieGoodwillie}.
\begin{prop}\label{p:rep-inn}
Let \(\C\) be an \(\infty\)-bicategory. Then for every \(x \in \C\) the inclusion of marked-scaled simplicial sets \(\{\id_x\} \subseteq \C^{x/}_{\inn}\) is an inner cocartesian equivalence over \(\C\).
\end{prop}
\begin{proof}
This follows directly from~\cite[Proposition 4.1.8]{LurieGoodwillie}, which says that the map of marked simplicial sets underlying \(\{\id_x\} \subseteq \C^{x/}_{\inn}\) is \(\Beta_{\C}\)-anodyne, and in particular a weak equivalence in the \(\Beta\)-fibered model structure on \((\s^+)_{/\C}\). We also note that this gives the result on the level of mapping \(\infty\)-categories, 
since the \(\Beta\)-fibered model structure is compatible with the action of \(\s^+\) on \((\s^+)_{/\C}\).
\end{proof}

Passing the opposites and using Remark~\ref{r:opposites} we obtain:
\begin{cor}\label{c:rep-inn}
Let \(\C\) be an \(\infty\)-bicategory. Then for every \(x \in \C\) the inclusion \(\{\id_x\} \subseteq \C^{/x}_{\inn}\) is an inner cartesian equivalence over \(\C\).
\end{cor}

We now establish the analogous statements for outer slice fibrations.
\begin{prop}\label{p:rep-out}
Let \(\C\) be an \(\infty\)-bicategory. Then for every \(x \in \C\) the inclusion \(\{\id_x\} \subseteq \C^{/x}_{\out}\) is an outer cartesian equivalence over \(\C\).
\end{prop}

We reproduce the argument of~\cite[Proposition 4.1.8]{LurieGoodwillie} in the outer cartesian setting. For this, we will require the following lemma:
\begin{lemma}\label{l:lax-lift}
Let \(\C\) be an \(\infty\)-bicategory and \(K \to L\) an inclusion of marked scaled simplicial sets. Let \(E\) be the collection of those edges \((e,e')\colon {}^{\flat}\Del^1 \to {}^{\flat}\Del^1 \otimes K\) such that \(e'\) is marked and either \(e\) or \(e'\) are degenerate, and define \(E'\) in a similar manner for \({}^{\flat}\Del^1 \otimes L\). Then for any map \({}^{\flat}\Del^1 \otimes L \to \C\) the inclusion
\[ [\Del^{\{1\}} \otimes L] \displaystyle\mathop{\coprod}_{\Del^{\{1\}} \otimes K}({}^{\flat}\Del^1 \otimes K,E)\to ({}^{\flat}\Del^1 \otimes L,E') \]
is an outer cartesian equivalence over \(\C\), where \(({}^{\flat}\Del^1\otimes K,E)\) denotes the marked-scaled simplicial set whose underlying scaled simplicial set is \({}^{\flat}\Del^1 \otimes K\) and whose set of marked edges is \(E\), and similarly for \((\Del^1_{\flat} \otimes L,E')\).
\end{lemma}
\begin{proof}
We need to show that for every outer cartesian fibration \(\E \to \C\) the restriction functor
\[\Fun^{\car}_{\C}(({}^{\flat}\Del^1 \otimes L,E'),\E) \to \Fun^{\car}_{\C}\Big([\Del^{\{1\}} \otimes L] \coprod_{\Del^{\{1\}} \otimes K}({}^{\flat}\Del^1 \otimes K,E),\E\Big)\]
is an equivalence of \(\infty\)-categories. For this we first note that by Corollary~\ref{c:2-out-of-3} we have a pullback square
\[\xymatrix{
\Fun^{\car}_{\C}(({}^{\flat}\Del^1 \otimes L,E'),\E) \ar[r]\ar[d] & \Fun^{\car}_{\C}\Big([\Del^{\{1\}} \otimes L] \displaystyle\mathop{\coprod}_{\Del^{\{1\}} \otimes K}({}^{\flat}\Del^1 \otimes K,E),\E\Big) \ar[d] \\
\Fun^{\car}_{\C}(({}^{\flat}\Del^1 \otimes L,E_0'),\E) \ar[r] & \Fun^{\car}_{\C}\Big([\Del^{\{1\}} \otimes L]\displaystyle\mathop{\coprod}_{\Del^{\{1\}} \otimes K}({}^{\flat}\Del^1 \otimes K,E_0),\E\Big) \\
}\]
where \(E_0\) denotes the set of edges of the form \(\Del^1 \times \{v\}\) for \(v \in K\) and \(E_0'\) the set of edges of the form \(\Del^1 \times \{u\}\) for \(u \in L\). It will hence suffice to prove that the bottom horizontal map is a trivial fibration. In particular, we may ignore the marking on \(K\) and \(L\) and work with their underlying scaled simplicial sets \(\ovl{K}\) and \(\ovl{L}\).
Now, to show that the map in question has the right lifting property with respect to any inclusion \(Z \subseteq W\) of scaled simplicial sets translates to finding a lift in a square of the form
\begin{equation}\label{e:lax-lift-families} 
\xymatrix{
[\Del^1_{\flat} \otimes \ovl{L}] \times Z \displaystyle\mathop{\coprod}_{[\Del^{\{1\}} \otimes \ovl{L} \coprod_{\Del^{\{1\}} \otimes \ovl{K}} \Del^1_{\flat} \otimes \ovl{K}] \times Z} \big[\Del^{\{1\}} \otimes \ovl{L} \coprod_{\Del^{\{1\}} \otimes \ovl{K}} \Del^1_{\flat} \otimes \ovl{K}\big] \times W \ar[r]\ar[d] & \E \ar[d]^{p} \\
[\Del^1_{\flat} \otimes \ovl{L}] \times W \ar[r] & \B \\
}
\end{equation}
which sends the edges in the set \(E_0' \times W_1\) to \(p\)-cartesian edges.

Let \(K' \to L' = (\ovl{K} \to \ovl{L}) \Box (Z \to W)\) be the pushout product of \(\ovl{K} \to \ovl{L}\) and \(Z \to W\) with respect to the cartesian product of scaled simplicial set. To carry on the proof we would like to replace the left vertical map in the above square with the Gray pushout-product of \(\Del^{\{1\}} \hrar \Del^1_{\flat}\) and \(K' \to L'\). These are identical on the level of the underlying simplicial sets (since the pushout-product of simplicial sets is associative), but have different scaling, since for general scaled simplicial sets \(A,B,C\) one has \(A \mgr (B \times C) \not{\cong} (A \mgr B) \times C\). Explicitly, a triangle \((\sig_A,\sig_B,\sig_C)\) in the cartesian product of \(A,B\) and \(C\) is thin in \(A \mgr (B \times C)\) if and only if \(\sig_A,\sig_B\) and \(\sig_C\) are all thin and in addition either \(\sig_A\) degenerates along \(\Del^{\{1,2\}}\) or both \(\sig_B\) and \(\sig_C\) degenerates along \(\Del^{\{0,1\}}\). On the other hand, \((\sig_A,\sig_B,\sig_C)\) is thin in \((A \mgr B) \times C\) if and only if \(\sig_A,\sig_B\) and \(\sig_C\) are all thin and either \(\sig_A\) degenerates along \(\Del^{\{1,2\}}\) or \(\sig_B\) degenerates along \(\Del^{\{0,1\}}\). In particular, we have a canonical inclusion of scaled simplicial sets
\[ A \mgr (B \times C) \subseteq (A \mgr B) \times C,\]
and so the square~\eqref{e:lax-lift-families} restricts to a square
\begin{equation}\label{e:lax-lift-families-2} 
\begin{tikzcd}
\Del^{\{1\}} \otimes L' \displaystyle\mathop{\coprod}_{\Del^{\{1\}} \otimes K'} \Del^1_{\flat} \otimes K'  \ar[r, "f"]\ar[d] & \E \ar[d, "p"] \\
\Del^1_{\flat} \otimes L' \ar[r, "H"] \ar[ur, dashed, "\wtl{H}"{description}] & \B
\end{tikzcd}
\end{equation}
Since the map \(\E \to \B\) is an outer fibration it detects thin triangles, and so the dotted lift exists in~\eqref{e:lax-lift-families} if and only if it exists in~\eqref{e:lax-lift-families-2}.
The existence of a lift in~\eqref{e:lax-lift-families-2} is then given by Proposition~\ref{p:lax-lift}.
\end{proof}

\begin{proof}[Proof of Proposition~\ref{p:rep-out}]
Consider the composite
\[e\colon {}^{\flat}\Del^1 \otimes {}^{\flat}\Del^1 \otimes \C^{/x}_{\out} \to {}^{\flat}\Del^1 \otimes \C^{/x}_{\out} \to \C^{/x}_{\out}\diamond_{\out}\Del^0 \to \C, \]
where the first map is induced by the map \({}^{\flat}\Del^1 \otimes {}^{\flat}\Del^1 \to {}^{\flat}\Del^1\) given on vertices by \(i,j \mapsto \max(i,j)\), the second is the quotient map, and the third is obtained from the counit of the adjunction between outer join and slice. The restriction of \(e\) to \(\Del^{\{1\}} \otimes {}^{\flat}\Del^1 \otimes \C^{/x}_{\out}\) is then constant with value \(x \in \C\) and so descends to give a map
\[ ({}^{\flat}\Del^1 \otimes \C^{/x}_{\out})^{\flat} \diamond_{\out} \Del^0 \to \C. \]
The last map then transposes to a map
\[ r\colon {}^{\flat}\Del^1 \otimes \C^{/x}_{\out} \to \ovl{\C}^{/x}_{\out} ,\]
which we read as a lax transformation from the identity on \(\ovl{\C}^{/x}_{\out}\) to the composite \(\ovl{\C}^{/x}_{\out} \to \{\id_x\} \to \ovl{\C}^{/x}_{\out}\). Furthermore, this transformation is constant when restricted to \(\{\id_x\} \subseteq \ovl{\C}^{/x}_{\out}\), and sends every edge of the form \({}^{\flat}\Del^1 \otimes \{v\}\) in \({}^{\flat}\Del^1 \otimes \C^{/x}_{\out}\) to an edge which is marked in \(\C^{/x}_{\out}\). The map \(r\) then fits into a commutative diagram of marked-scaled simplicial sets
\[ \xymatrix{
\Del^{\{0\}} \times \{\id_x\} \ar[r]\ar[d] &  \Del^{\{0\}}\times \C^{/x}_{\out} \ar[d] \\
[{}^{\sharp}\Del^1 \times \{\id_x\}] \displaystyle\mathop{\coprod}_{\Del^{\{1\}} \times \{\id_x\}}[\Del^{\{1\}} \times \C^{/x}_{\out}] \ar[d]\ar[r] & ({}^{\flat}\Del^1 \otimes \C^{/x}_{\out},E) \ar[d]^{r} \\
\{\id_x\} \ar[r] & \C^{/x}_{\out} \\
}\]
where \(({}^{\flat}\Del^1 \otimes \C^{/x}_{\out},E)\) denotes the marked-scaled simplicial set whose underlying scaled simplicial set is \({}^{\flat}\Del^1 \otimes \C^{/x}_{\out}\) and whose set of marked edges is the set \(E\) consisting of those edges \((e,e') \colon \Del^1_{\flat} \to {}^{\flat}\Del^1 \otimes \C^{/x}_{\out}\) such that \(e'\) is marked and either \(e\) or \(e'\) is degenerate. This shows that the inclusion \(\{\id_x\} \subseteq \C^{/x}_{\out}\) is a retract, over \(\C\), of the inclusion 
\begin{equation}\label{e:inclusion} 
[{}^{\sharp}\Del^1 \times \{\id_x\}] \displaystyle\mathop{\coprod}_{\Del^{\{1\}} \otimes \{\id_x\}}[\Del^{\{1\}} \times \C^{/x}_{\out}] \to ({}^{\flat}\Del^1 \otimes \C^{/x}_{\out},E) ,
\end{equation}
where to avoid confusion we emphasize that we consider the object on the right as living over \(\C\) via the map 
\[ {}^{\flat}\Del^1 \otimes \C^{/x}_{\out} \xrightarrow{r} \ovl{\C}^{/x}_{\out} \to \C.\]
It will hence suffice to show that~\eqref{e:inclusion} is an outer cartesian equivalence over \(\C\). 
Indeed, this is a particular case of Lemma~\ref{l:lax-lift}.
\end{proof}

Passing the opposites and using Remark~\ref{r:opposites} we obtain:
\begin{cor}\label{c:rep-out}
Let \(\C\) be an \(\infty\)-bicategory. Then for every \(x \in \C\) the inclusion \(\{\id_x\} \subseteq \C^{x/}_{\out}\) is an outer cocartesian equivalence over \(\C\).
\end{cor}

\begin{cor}\label{c:representables}
Under the equivalences of Corollary~\ref{c:straightening-inner} and Corollary~\ref{c:S-U-for-outer-fibs} the fibrations
\[
\begin{tikzcd}[column sep=tiny]
	\ovl{\C}^{x/}_{\inn} \ar[drrr] && \ovl{\C}^{x/}_{\out}\ar[dr] && \ovl{\C}^{/x}_{\inn} \ar[dl] && \ovl{\C}^{/x}_{\out}\ar[dlll] \\
	&&& \C &&&
\end{tikzcd}
\]
correspond to the functors (op)(co)represented by \(x \in \C\), respectively. In addition, the \((\ZZ/2)^{2}\)-action on \(\BiCat\) switches between these four fibrations, so that we have equivalences
\begin{equation}\label{e:symmetry}
\begin{tikzcd}[column sep=tiny]
	\C^{x/}_{\inn} \ar[drrr] & \simeq & ((\C^{\co})^{x/}_{\out})^{\co} \ar[dr] & \simeq &
		((\C^{\op})^{/x}_{\inn})^{\op} \ar[dl] & \simeq & ((\C^{\coop})^{/x}_{\out})^{\coop} \ar[dlll] \\
	&&& \C &&&
\end{tikzcd}.
\end{equation}
\end{cor}

\begin{proof}
The equivalences~\eqref{e:symmetry}, two of which are already visible on the level of the simplicial construction as described in Remark~\ref{r:opposites}, are implied by Propositions~\ref{p:rep-inn} and~\ref{p:rep-out} and Corollaries~\ref{c:rep-inn} and~\ref{c:rep-out}, which characterize each of these fibrations by the same type of universal mapping property. 
By these equivalences and the way that the Lurie-Grothendieck correspondence is constructed in Corollary~\ref{c:S-U-for-outer-fibs}, to prove the first claim it is enough to consider the case of \(\C^{x/}_{\inn}\). Indeed, as explained above, it follows from Remark~\ref{r:unst-rep} that \(\C^{x/}_{\inn}\) is equivalent to the unstraightening of the functor corepresented by \(x\). 
\end{proof}

\begin{example}[Universal fibrations]\label{ex:universal}
Consider the \(\infty\)-bicategory \(\Catoo\) of small \(\infty\)-categories and let \(\ast := \Del^0 \in \Catoo\) be the terminal \(\infty\)-bicategory. Combining Corollary~\ref{c:representables} and Remark~\ref{r:unst-rep} we get that the inner cocartesian fibration 
\[p^{\uni}_{\inn}\colon(\Catoo)^{\ast/}_{\inn} \to \Catoo\]
corresponds via the bicategorical Grothendieck-Lurie correspondence to the identity functor \(\Catoo \to \Catoo\). We will consequently refer to \(p^{\uni}_{\inn}\) as the \emph{universal inner cocartesian fibration}. By the compatibility of the Grothendieck-Lurie correspondence with base change (see Remark~\ref{r:unstr-base-change}) it then follows that for every map \(\chi\colon \B \to \Catoo\) the inner cocartesian fibration classified by \(\chi\) is the base change 
\[\B \times_{\Catoo}(\Catoo)^{\ast/}_{\inn} \to \B \]
of the universal inner cocartesian fibration along \(\chi\). By Corollary~\ref{c:representables} the slice fibrations
\[ (\Catoo^{\co})^{\ast/}_{\out} \to \Catoo^{\co} \quad,\quad (\Catoo^{\op})^{/\ast}_{\inn} \to \Catoo^{\op} \quad\text{and}\quad (\Catoo^{\coop})^{/\ast}_{\out} \to \Catoo^{\coop} \]
are obtained from \(p^{\uni}_{\inn}\) by applying the three non-trivial symmetries of \(\BiCat\), and are hence all classified by the identity map \(\Catoo \to \Catoo\). These play the analogous role for the classification of outer cocartesian, inner cartesian and outer cartesian fibrations, and we will similarly consider them as the universal fibrations of their flavour. We note that the functor \(\C \mapsto \C^{\op}\) realizes an equivalence \(\Catoo^{\co} \simeq \Catoo\), and we can use this equivalence in order to identify the universal outer cocartesian fibration with 
\[(\Catoo)^{\ast/}_{\out} \to \Catoo.\] 
We then conclude that 
for a map \(\chi\colon \B^{\co} \to \Catoo\), the outer cocartesian fibration classified by \(\chi\) is the base change \(\B \times_{\Catoo}\colon(\Catoo)^{\ast/}_{\out}\) taken along the composed functor \(\B \xrightarrow{\chi^{\co}} \Catoo^{\co} \xrightarrow{(-)^{\op}} \Catoo\).
\end{example}

We now consider the question of identifying when a given inner/outer (co)cartesian fibration \(\E \to \C\) is (co)representable by an object \(x \in \C\). To fix ideas, let us consider the cocartesian case, and fix a variance parameter \(\var \in \{\out,\inn\}\) as above.

\begin{define}\label{d:represents}
Let \(p\colon\E \to \C\) be a \(\var\)-cocartesian fibration of \(\infty\)-bicategories. We will say that an object \(x \in \E\) is \emph{\(p\)-universal} if the inclusion \(\{x\} \subseteq \E^{\natural}\) is a \(\var\)-cocartesian equivalence over \(\C\), where \(\E^{\natural}\) denotes the marked-scaled simplicial set whose underlying scaled simplicial set is \(\E\) and whose marked edges are the \(p\)-cocartesian ones. In this case we will also say that \(x \in \E\) exhibits \(\E\) as \emph{corepresented} by \(p(x)\).
\end{define}

\begin{rem}\label{r:universal-represents}
In the situation of Definition~\ref{d:represents},
the inclusion \(\{x\} \subseteq \E^{\natural}\) can always be extended to a map \(\C^{p(x)/}_{\var} \to \E^{\natural}\) which sends \(\id_{p(x)}\) to \(x\): indeed, the restriction
\[ \Fun^{\coc}_{\C}(\C^{p(x)/}_{\inn},\E) \to \Fun^{\coc}_{\C}(\{\id_{p(x)}\},\E)\]
is trivial fibration of \(\infty\)-categories by Proposition~\ref{p:rep-inn} and a Remark~\ref{r:restrict-fibration}. Using again Proposition~\ref{p:rep-inn} it now follows that \(x\) is \(p\)-universal if and only if the resulting map \(\C^{/p(x)}_{\var} \to \E^{\natural}\) is a \(\var\)-cocartesian equivalence over \(\C\), or equivalently, an equivalence of \(\var\)-cocartesian fibrations over \(\C\) (see Remark~\ref{r:cofinal-is-equivalence}). 
\end{rem}

\begin{rem}\label{r:universal-invariant}
In the situation of Definition~\ref{d:represents}, the collection of \(p\)-universal objects in \(\E\) is closed under equivalence. Indeed, if \(x \simeq y\) are two equivalent objects in \(\E\) then there exists a map \(\eta\colon J_{\sharp} \to \E\) such that \(\eta(0)=x\) and \(\eta(1) =y\). Since both \(\{0\} \subseteq J_{\sharp}\) and \(\{1\} \subseteq J_{\sharp}\) are bicategorical equivalences it follows from Example~\ref{ex:equiv-is-cofinal} that \(\{x\} \subseteq \E^{\natural}\) is a \(\var\)-cocartesian equivalence over \(\C\) if and only if \(J_{\sharp}^{\flat} \to \E^{\natural}\) is a \(\var\)-cocartesian equivalence over \(\C\), and the same goes for \(y\). It then follows that \(x\) is \(p\)-universal if and only if \(y\) is \(p\)-universal.
\end{rem}

\begin{rem}\label{r:preserves-detects}
If \(f\colon\E \to \E'\) is an equivalence of \(\var\)-cocartesian fibrations over \(\C\) then \(f\) preserves and detects universal objects by Example~\ref{ex:equiv-is-cofinal}. 
\end{rem}

The fully-faithfulness of the Yoneda embedding suggests that if a \(\var\)-cocartesian fibration \(\E\to \C\) is classified by a corepresentable functor, then the corepresenting object \(x\) is essentially unique. The following proposition makes this statement precise:

\begin{prop}\label{p:rep-unique}
Let \(p\colon\E \to \C\) be a \(\var\)-cocartesian fibration of \(\infty\)-bicategories. Let \(\X \subseteq \E\) be the sub-bicategory spanned by the \(p\)-universal objects and the \(p\)-cocartesian morphisms between them. Then \(\X\) is either empty or a contractible Kan complex.
\end{prop}
\begin{proof}
Suppose that \(\X\) is non-empty, so that there exists a \(p\)-universal object \(x \in \E\). By Remark~\ref{r:universal-represents} the inclusion \(\{x\} \subseteq \E\) extends to an equivalence 
\[ \xymatrix{
\ovl{\C}^{p(x)/}_{\var} \ar[rr]^-{f}_-{\simeq}\ar[dr]_{q} && \E \ar[dl]^{p} \\
&\C& \\
}\]
of \(\var\)-cocartesian fibrations over \(\C\). Let \(\Y \subseteq \ovl{\C}^{p(x)/}_{\var}\) be the full sub-bicategory spanned by the \(q\)-universal objects and the \(q\)-cocartesian edges between them.
Combining Remark~\ref{r:universal-invariant} and Remark~\ref{r:preserves-detects} we may deduce that \(f\) induces an equivalence \(\Y \xrightarrow{\simeq} \X\).  
It will hence suffice to prove that \(\Y\) is a contractible Kan complex. 

We now claim that an object \([\alp\colon p(x) \to y] \in \ovl{\C}^{p(x)/}_{\var}\) is \(q\)-universal if and only if \(\alp\) is invertible in \(\C\). To see this, extend the inclusion \(\{[\alp\colon p(x) \to y]\} \subseteq \ovl{\C}^{p(x)}_{\var}\) to a map of \(\var\)-cocartesian fibrations
\[ \xymatrix{
\ovl{\C}^{y/}_{\var} \ar[rr]^-{g}_-{\simeq}\ar[dr]_{q'} && \ovl{\C}^{p(x)/}_{\var} \ar[dl]^{q} \\
&\C& \\
}\]
so that \([\alp]\) is universal if and only if \(g\) is a bicategorical equivalence. By Corollary~\ref{c:fiberwise} this is equivalent to the induced map
\[ (\ovl{\C}^{y/}_{\var})_z \to (\ovl{\C}^{p(x)/}_{\var})_z \]
being a categorical equivalence on underlying simplicial sets for every \(z \in \C\). Now by Remark~\ref{r:mapping-space} we may identify this map with the induced map
\[ (-)\circ \alp\colon \Hom_{\C}(y,z) \to \Hom_{\C}(p(x),z) \]
when \(\var=\inn\), and with the opposite of this map (up to categorical equivalence) when \(\var=\out\). These maps are all equivalences precisely when \(\alp\) is invertible.

Now if \(\alp\colon p(x) \to y\) and \(\bet\colon p(x) \to z\) are two \(q\)-universal objects then by Corollary~\ref{cor:slice} a \(q\)-cocartesian arrow from \(\alp\) to \(\beta\) corresponds to a commutative square
	\[
	 \begin{tikzcd}
	 	p(x) \ar[d, equal] \ar[r,"\alp"] \ar[rd, ""{name=d-up}, ""{swap, name=d-down}]&
	 	y  \ar[d] \\
	 	p(x) \ar[r,"\beta"] & z
	 	\ar[from=d-down, to=1-2, phantom, "\simeq"]
	 	\ar[from=d-down, to=2-1, phantom, "\simeq"] \ ,
	 \end{tikzcd} 
	\]
and since \(\alp\) and \(\beta\) are invertible the arrow \(y \to z\) is invertible as well. In particular, every arrow in \(\Y\) is invertible. We now claim that every triangle in \(\Y\) is thin. To see this, note that a triangle in \(\Y\) corresponds to a map \(\rho\colon\Del^1_{\flat} \mgr \Del^2_{\flat} \to \C\) such that \(\rho_{|\Del^{\{0\}} \times \Del^2_{\flat}}\) is constant with image \(p(x)\), and by the above we also have that \(\rho\) sends every edge of \(\Del^1_{\flat} \mgr \Del^2_{\flat}\) to an equivalence in \(\C\) and every triangle in \(\Del^1_{\flat} \mgr \Del^2_{\flat}\) whose projection to \(\Del^2_{\flat}\) is degenerate to a thin triangle in \(\C\). 
By~\cite[Corollary 3.5]{GagnaHarpazLanariEquiv} the triangle \(\rho_{|\Del^{\{1\}}_{\flat} \otimes \Del^2_{\flat}}\) is also thin in \(\C\). 
In fact, the proof of that corollary actually shows that \(\rho\) sends every triangle in \(\Del^1_{\flat} \otimes \Del^2_{\flat}\) to a thin triangle in \(\C\). We may consequently identify \(\Y\) with the subgroupoid \((\C^{\simeq})^{p(x)/} \subseteq \C^{p(x)/}_{\var}\). We now finish the proof by noting that \((\C^{\simeq})^{p(x)/}\) is a contractible Kan complex, being an \(\infty\)-groupoid with an initial object.
\end{proof}

\subsection{Categories of lax transformations}\label{s:lax-trans}

In this subsection we study the relation between slice fibrations and certain \(\infty\)-categories of lax transformations. By the latter we mean the following: 

\begin{define}
Let \(\C\) be an \(\infty\)-bicategory and \(K\) a marked-scaled simplicial set. For two diagrams \(f,g\colon \ovl{K} \to \C\) we define \(\RNat_K(f,g)\) and \(\LNat_K(f,g)\) to be the mapping \(\infty\)-categories from \(f\) to \(g\) in the \(\infty\)-bicategories \(\RMap(K,\C)\) and \(\LMap(K,\C)\), respectively.
\end{define}

Our main goal in the present subsection is to identify the fibers of the slice fibrations over a diagram in \(\C\) in terms of suitable spaces of lax natural transformations. We also deduce in the end a useful invariance property with respect to the restriction along inner/outer (co)cartesian equivalences \(K \to L\).
We begin with the following statement, identifying \(\infty\)-categories of lax transformations to/from a diagram which is constant on an object \(x \in \C\) in terms cartesian lifts to the slice fibration over/under \(x\):
 
\begin{prop}\label{p:fiber}
	Let \(\C\) be an \(\infty\)-bicategory, \(K\) a marked-scaled simplicial set and \(f \colon \ovl{K} \to \C\) be a functor. Then for
\(x \in \C\) there are natural equivalences
\[ \RNat_K(\ovl{x},f) \simeq \Fun^{\coc}_\C(K,\C^{x/}_{\inn}) \quad,\quad \RNat_K(f,\ovl{x}) \simeq \Fun^{\car}_\C(K,\C^{/x}_{\out}),\]
\[ \LNat_K(f,\ovl{x}) \simeq \Fun^{\car}_\C(K,\C^{/x}_{\inn})^{\op}  \quad\text{and}\quad \LNat_K(\ovl{x},f) \simeq \Fun^{\coc}_\C(K,\C^{x/}_{\out})^{\op},\]
where \(\ovl{x}\) denotes the constant map \(\ovl{K} \to \C\) with value \(x\).
\end{prop}
\begin{proof}
Let us prove the first pair of equivalences. The second pair can then be deduced by replacing \(f\colon K \to \C\) with \(f^{\op}\colon K^{\op} \to \C^{\op}\) using Remarks~\ref{r:mapping-space} and~\ref{r:symmetric-2}. To obtain the first equivalence it will suffice by Remark~\ref{r:mapping-space} to produce an isomorphism
\begin{equation}\label{e:nat-fun}
\bigl(\RMap(K,\C)^{/f}_{\out}\bigr)_{\ovl{x}} \cong \Fun^{\coc}_\C(K,\C^{x/}_{\inn}) .
\end{equation}
Indeed, using Remark~\ref{r:almost-associative} we see that the scaled simplicial set on the left hand side represents the sub-functor of 
\[Z \mapsto \Fun({}^{\flat}\Del^1 \mgr Z^{\flat}\mgr K,\C)\] spanned by those maps \({}^{\flat}\Del^1 \mgr Z^{\flat} \mgr K \to \C\) whose restriction to \(\Del^{\{1\}} \otimes Z^{\flat} \otimes K\) is given by \(Z^{\flat} \otimes K \to K \xrightarrow{f} \C\) and whose restriction to \(\Del^{\{0\}} \otimes Z^{\flat} \otimes K\) is given by \(Z^{\flat} \otimes K \to \Del^0 \xrightarrow{x} \C\). On the other hand, the scaled simplicial set on the right hand side of~\eqref{e:nat-fun} represents the sub-functor of \(Z \mapsto \Fun(\prescript{\flat}{}\Del^1 \mgr (Z^{\flat} \times K),\C)\) defined by the same conditions on the values at \(\Del^{\{0\}}\) and \(\Del^{\{1\}}\). We now observe that we have natural inclusions of scaled simplicial sets (which are isomorphisms on underlying simplicial sets)
\begin{equation}\label{e:nat-fun-2} 
{}^{\flat}\Del^1 \mgr Z^{\flat} \mgr K \leftarrow {}^{\flat}\Del^1 \mgr (Z^{\flat} \mgr K,E) \to \prescript{\flat}{}\Del^1 \mgr (Z^{\flat} \times K)\end{equation}
where \(E\) is the set of edges which are marked in the cartesian product \(Z^{\flat} \times K\). We hence obtain a zig-zag of maps relating the two sides of~\eqref{e:nat-fun}. But this zig-zag is in fact a zig-zag of isomorphisms since the left map in~\eqref{e:nat-fun-2} is a bicategorical trivial cofibrations by Proposition~\ref{p:marked-variants} and the right map becomes scaled anodyne after collapsing \(\partial \Del^1 \times Z \times \ovl{K}\) to \(\Del^{\{0\}} \coprod \Del^{\{1\}} \times \ovl{K}\) by Lemma~\ref{l:lem-for-cocar}.

For the second equivalence, we invoke again Remark~\ref{r:mapping-space} and produce instead an isomorphism 
\[\bigl(\RMap(K,\C)^{f/}_{\inn}\bigr)_{\ovl{x}} \cong \Fun^{\coc}_\C(K,\C^{/x}_{\out}).\]
The argument in this case then proceeds exactly as above by noting that both sides represent again a common sub-functor of \(Z \mapsto \Fun({}^{\flat}\Del^1 \mgr Z^{\flat} \mgr K,\C)\) and \(Z \mapsto \Fun({}^{\flat}\Del^1 \mgr (Z^{\flat} \times K),\C)\) defined in a similar manner, where this time the conditions at \(\Del^{\{0\}}\) and \(\Del^{\{1\}}\) are switched.
\end{proof}

Applying Proposition~\ref{p:fiber} in the case of \(\C=\Catoo,x=\ast\) and using Example~\ref{ex:universal} we obtain: 

\begin{cor}\label{c:fiber}
	Let \(K\) be a marked-scaled simplicial set and \(\chi \colon \ovl{K} \to \Catoo\) be a functor. Let \(\E \to \ovl{K}\) be the inner cocartesian fibration classified by \(\chi\) and \(\E' \to \ovl{K}\) the outer cocartesian fibration classified by 
\[\ovl{K}^{\co} \xrightarrow{\chi^{\co}} \Catoo^{\co} \xrightarrow{\op} \Catoo.\] 
Then there are natural equivalences
\[ \RNat_K(\ast,\chi) \simeq \Fun^{\coc}_{\Catoo}(K,(\Catoo)^{\ast/}_{\inn})\simeq \Fun^{\coc}_{\ovl{K}}(K,\E) ,\]
and
\[ \LNat_K(\ast,\chi) \simeq \Fun^{\coc}_{\Catoo}(K,(\Catoo)^{\ast/}_{\out})^{\op} \simeq \Fun^{\coc}_{\ovl{K}}(K,\E')^{\op}.\]
\end{cor}

Combining Proposition~\ref{p:fiber} and Corollary~\ref{c:fiber} we then conclude:
\begin{cor}\label{c:C-to-cat}
	Let \(\C\) be an \(\infty\)-bicategory, \(K\) a marked-scaled simplicial set and \(f \colon \ovl{K} \to \C\) be a functor. Then for
\(x \in \C\) there are natural equivalences
\[
 \begin{split}
 	\RNat_K(\ovl{x},f) &\simeq \RNat_K(\ast,\Hom_{\C}(x,f(-))), \\
	\LNat_K(\ovl{x},f) &\simeq \LNat_K(\ast,\Hom_{\C}(x,f(-)))
 \end{split}
\]
and
\[
 \begin{split}
	\LNat_K(f,\ovl{x}) &\simeq \RNat_{K^{\op}}(\ast,\Hom_{\C}(f(-),x)), \\
	\RNat_K(f,\ovl{x}) &\simeq \LNat_{K^{\op}}(\ast,\Hom_{\C}(f(-),x)).
 \end{split}
\]
\end{cor}
To avoid confusion, we note that in Corollary~\ref{c:C-to-cat} the transformations on the left side of each equivalence concern \(\C\)-valued diagrams indexed by \(K\), whereas the transformations on the right of each equivalence concern \(\Catoo\)-valued diagrams, indexed by either \(K\) or \(K^{\op}\). The notations \(\Hom_{\C}(x,f(-)))\) and \(\Hom_{\C}(f(-),x))\) refer to the post composition of \(f\colon \ovl{K} \to \C\) with the functors represented and corepresented by \(x\).

\begin{proof}[Proof of Corollary~\ref{c:C-to-cat}]
The second pair of equivalences can be deduced from the first by replacing \(f\colon K \to \C\) with \(f^{\op}\colon K^{\op} \to \C^{\op}\), using the equivalences
\[
 \LNat_{K}(f,\ovl{x}) \simeq \RNat_{K^{\op}}(f^{\op},\ovl{x})
 \quad\text{and}\quad
 \RNat_{K}(f,\ovl{x}) \simeq \LNat_{K^{\op}}(f^{\op},\ovl{x}),
\]
see Remark~\ref{r:symmetric-2}.
The first pair of equivalences follows by combining Proposition~\ref{p:fiber} with Corollary~\ref{c:fiber} (applied to \(\chi := \Hom_{\C}(x,f(-))\)).
\end{proof}

We now turn to discussing the slice fibrations
\[
	\begin{tikzcd}[column sep=tiny]
	\ovl{\C}^{f/}_{\inn} \ar[drrr] && \ovl{\C}^{f/}_{\out}\ar[dr] && \ovl{\C}^{/f}_{\inn} \ar[dl] && \ovl{\C}^{/f}_{\out}\ar[dlll] \\
	&&& \C &&&
	\end{tikzcd}
\]
and the functors that classify them under the bicategorical Grothendieck--Lurie correspondence.

\begin{prop}\label{p:slice-represent}
Let \(K\) be a marked-scaled simplicial set and \(f\colon \ovl{K} \to \C\) a diagram.
\begin{itemize}
\item
The inner cocartesian fibration \(\ovl{\C}^{f/}_{\inn} \to \C\)
is classified by the functor 
\[x \mapsto \LNat_K(f,\ovl{x}) \simeq \RNat_{K^{\op}}(\ast,\Hom_{\C}(f(-),x)).\]
\item
The outer cocartesian fibration \(\ovl{\C}^{f/}_{\out} \to \C\) is classified by the \(\co\)-functor 
\[x \mapsto \RNat_K(f,\ovl{x})^{\op} \simeq \LNat_{K^{\op}}(\ast,\Hom_{\C}(f(-),x))^{\op}.\]
\item
The inner cartesian fibration \(\ovl{\C}^{/f}_{\inn} \to \C\) is classified by the \(\co\)-presheaf 
\[x \mapsto \RNat_K(\ovl{x},f)^{\op} \simeq \RNat_K(\ast,\Hom_{\C}(x,f(-)))^{\op}.\]
\item
The outer cartesian fibration \(\ovl{\C}^{/f}_{\out} \to \C\) is classified by the presheaf \[x \mapsto \LNat_K(\ovl{x},f) \simeq \LNat_K(\ast,\Hom_{\C}(x,f(-))).\]
\end{itemize}
Here, in all cases \(\ovl{x}\) denotes the constant diagram \(\ovl{K} \to \C\) with value \(x\).
\end{prop}
\begin{proof}
As in the proof of Proposition~\ref{p:cocar} we have a pullback square
	\[
	 \begin{tikzcd}
		\ovl{\C}^{f/}_{\inn} \ar[r]\ar[d] & \ovl{\LMap}(K, \C)^{[f]/}_{\inn} \ar[d] \\
		\C \ar[r] & \LMap(K, \C)
	 \end{tikzcd}
	\]
where \(\ovl{\LMap}(K,\C)^{[f]/}_{\inn}\) denotes the underlying scaled simplicial set of the marked-scaled simplicial set \(\LMap(K,\C)^{[f]/}_{\inn}\). The desired result then follows from Corollary~\ref{c:representables} and the compatibility of the straightening-unstraightening equivalence with base change, see~\cite{LurieGoodwillie}. The equivalence between lax transformations of \(\C\)-valued and \(\Catoo\)-valued diagrams is then given by Corollary~\ref{c:C-to-cat}.

The proof of the other three variances is the same: the analogous pullback square exists in all four cases, and the compatibility with base change of the Lurie-Grothendieck correspondence in the inner cocartesian cases implies all other variances, see Remark~\ref{r:unstr-base-change}.
\end{proof}

\begin{cor}[Invariance of slice fibrations]\label{c:invariance}
Let \(h\colon K \to L\) be a map of marked-scaled simplicial sets. Let \(\C\) be an \(\infty\)-bicategory and \(f\colon \ovl{L} \to \C\) be a scaled map. For a variance parameter \(\var \in \{\inn,\out\}\), if \(h\) is a \(\var\)-cocartesian (resp.~\(\var\)-cartesian) equivalence over \(\C\) then the projections
\[ \ovl{\C}^{/f}_{\var} \to \ovl{\C}^{/fh}_{\var} \quad(\text{resp.}\quad \ovl{\C}^{f/}_{\inn} \to \ovl{\C}^{fh/}_{\inn}) \]
are equivalences of \(\infty\)-bicategories, respectively.
\end{cor}

\begin{rem}\label{r:equivalence}
In the situation of Corollary~\ref{c:invariance}, if the induced map \(\ovl{h}\colon \ovl{K} \to \ovl{L}\) is a bicategorical equivalence and the marked edges in \(L\) are the images of the marked edges in \(K\) then \(h\) is in particular an inner/outer (co)cartesian equivalence over \(\C\) (see Example~\ref{ex:equiv-is-cofinal}). We then deduce that in this case the four restriction maps
\[ \ovl{\C}^{/f}_{\inn} \to \ovl{\C}^{/fh}_{\inn} \quad,\quad \ovl{\C}^{/f}_{\out} \to \ovl{\C}^{/fh}_{\out}, \]
\[ \ovl{\C}^{f/}_{\inn} \to \ovl{\C}^{fh/}_{\inn} \quad\text{and}\quad \ovl{\C}^{f/}_{\out} \to \ovl{\C}^{fh/}_{\out} \]
are all equivalence of \(\infty\)-bicategories.
\end{rem}

\begin{proof}[Proof of Corollary~\ref{c:invariance}]
Replacing \(h \colon K \to L\) and \(f\colon \ovl{L} \to \C\) by \(h^{\op}\colon K^{\op} \to L^{\op}\) and \(f^{\op}\colon \ovl{L}^{\op} \to \C^{\op}\) switches between cartesian and cocartesian fibrations, and so it will suffice to prove the case where \(h\) is a \(\var\)-cocartesian equivalence. 

Applying Corollary~\ref{c:fiberwise} it will suffice to show that for every \(x \in \C\) the map
\[ (\ovl{\C}^{/f}_{\var})_x \to (\ovl{\C}^{/f\iota}_{\var})_x \]
is an equivalence of \(\infty\)-categories. Using Proposition~\ref{p:slice-represent} and Proposition~\ref{p:fiber} we may identify this map with the map
\[ \Fun^{\coc}_{\C}(L,\ovl{\C}^{x/}_{\var})^{\op} \to \Fun^{\coc}_{\C}(K,\ovl{\C}^{x/}_{\var})^{\op} .\]
This, in turn, is an equivalence by the assumption that \(K \to L\) is a \(\var\) cocartesian equivalence over \(\C\).
\end{proof}

\section{Limits and colimits in \pdfoo-bicategories}\label{sec:limits}

In this section we define and study a notion of 2-(co)limit suitable for \(\infty\)-bicategories. These (when exist) are associated to a diagram \(f\colon \ovl{K} \to \C\) in an \(\infty\)-bicategory \(\C\), indexed by the underlying scaled simplicial set of a marked-scaled simplicial set \(K\). They come in four flavors, depending on a variance parameter \(\var \in \{\inn,\out\}\), and on whether we take limits or colimits. We give the basic definitions and extract some of their properties in \S\ref{subsec:universal}. In particular, we show that 2-(co)limits can be characterized by the functor they (co)represent (see Corollary~\ref{c:limit-represents}). In \S\ref{s:cofinal} we introduce the notions of inner/outer cofinal and coinitial maps, and show that restriction along them does not affect inner/outer colimits and limits, respectively. We also give an equivalent characterization of limit cones in terms of cofinality, and deduce a unicity result for these cones (Corollary~\ref{c:unicity}). In \S\ref{s:weighted} we study the notion of weighted limits and colimits, which can be expressed using the same language introduced in \S\ref{subsec:universal}. We give a characterization of these (co)limits in terms of the functors they (co)represent, and also show that any 2-(co)limit can be expressed as a weighted (co)limit with respect to a suitable weight. Finally, in \S\ref{s:model-categories} we show that 2-(co)limits in \(\infty\)-bicategories which come from model categories \(\M\) tensored over marked simplicial sets exist and can be computed in terms of weighted homotopy colimits, under mild assumptions on \(\M\).

\subsection{Universal cones}\label{subsec:universal}

\begin{notate}\label{d:cone}
	Let \(K\) be a marked-scaled simplicial set. 
	We will denote by \(K^{\triangleleft}_{\inn}\) the marked-scaled simplicial set whose underlying 
	scaled simplicial set is given by \(\ovl{K}^{\triangleleft}_{\inn} := \{\ast\} \diamond_{\inn} K\), and whose marked edges are the union of the marked edges of \(K\)
	as well as every edge that contains \(\ast\). 
	Similarly, we will denote by \(K^{\triangleleft}_{\out}\) the marked-scaled simplicial set 
	whose underlying scaled simplicial set is \(\{\ast\} \diamond_{\out} K\) and whose marked edges are defined in the analogous way.
\end{notate}

For \(K\) a marked-scaled simplicial set and \(\C\) and \(\infty\)-bicategory,
we shall call a diagram \(\ovl{K}^{\triangleleft}_{\inn} \lrar \C\)
an \ndef{inner cone diagram} and a diagram \(\ovl{K}^{\triangleleft}_{\out} \lrar \C\)
an \ndef{outer cone diagram}.

\begin{define}\label{d:limit}
	Let \(\C\) be an \(\infty\)-bicategory, \(K\) a marked-scaled simplicial set and
	\(g\colon \ovl{K}^{\triangleleft}_{\inn} \lrar \C\) be an inner cone diagram. 
	We will say that \(g\) is an \ndef{inner limit cone} on \(f := g_{|\ovl{K}}\) if the projection
	\[ \ovl{\C}^{/g}_{\inn} \lrar \ovl{\C}^{/f}_{\inn} \]
	is an equivalence of \(\infty\)-bicategories.
	Similarly, we will say that an outer cone diagram \(g\colon K^{\triangleleft}_{\out} \lrar \C\) is an \ndef{outer limit cone} on \(f\) 
	if the projection
	\[ \ovl{\C}^{/g}_{\out} \lrar \ovl{\C}^{/f}_{\out} \]
	is an equivalence of \(\infty\)-bicategories. 
	The definition of inner and outer \ndef{colimit cones} is defined in a similar way using the right cones \(K^{\triangleright}_{\inn}\) and \(K^{\triangleright}_{\out}\).
\end{define}

To facilitate the following discussion, let us fix a variance parameter \(\var \in \{\out,\inn\}\). 
We will refer to \(\var\)-(co)limits as in definition~\ref{d:limit} as \emph{2-(co)limits}. We note that while the diagram \(f\colon \ovl{K} \to \C\) does not take into account the marking on \(K\), the slice \(\infty\)-bicategory appearing in Definition~\ref{d:limit} (and hence the associated notion of \(2\)-(co)limit) critically depends on it. In particular, when all edges in \(K\) are marked the associated notion of \(2\)-(co)limit should be considered as a form of \emph{pseudo-(co)limits}, while if only the degenerate edges are marked it should be considered rather as a \emph{(op)lax (co)limit}, see also Remark~\ref{r:pseudo-lax} below.

\begin{example}\label{ex:final}
Suppose that \(K=\emptyset\). Then \(\ovl{K}^{\triangleleft}_{\var} = \Del^0\) and the data of a \(\var\)-cone in \(\C\) is simply the data of an object \(x \in \C\). By definition, this objects determines a \(\var\)-limit cone over \(\emptyset\) if and only if the map \(\ovl{\C}^{/x}_{\var} \to \C\) is an equivalence of \(\infty\)-bicategories. Since this map is a \(\var\)-cartesian fibration Corollary~\ref{c:fiberwise} tells us that it is an equivalence if and only if its fibers are categorically equivalent to \(\Del^0\). By Remark~\ref{r:mapping-space} this is the same as saying that the mapping \(\infty\)-categories \(\Hom_{\C}(y,x)\) are categorically equivalent to \(\Del^0\) for every \(y \in x\). In this case we say that \(x\) is a \defn{final object} of \(\C\). We note that in this case it does not matter if the variance parameter is \(\inn\) or \(\out\). Dually, inner and outer colimits of the empty diagram are given by objects \(x \in \C\) such that \(\Hom_{\C}(x,y)\) is categorically equivalent to \(\Del^0\) for every \(y\). We will then say that such an \(x\) is an \defn{initial object} of \(\C\).
\end{example}

\begin{rem}\label{r:op-limits}
Let \(\C\) be an \(\infty\)-bicategory, \(K\) a marked-scaled simplicial set and \(g\colon \ovl{K}^{\triangleleft}_{\var} \to \C\) a cone diagram for some variance parameter \(\var \in \{\inn,\out\}\). Then \(g\) is a \(\var\)-limit cone if and only if \[g^{\op}\colon (\ovl{K}^{\triangleleft}_{\var})^{\op} \cong (\ovl{K}^{\op})^{\triangleright}_{\var} \to \C^{\op} \]
is a \(\var\)-colimit cone.
This follows directly from the definition in light of the behavior of slice fibrations under opposites described in Remark~\ref{r:opposites}.
\end{rem}

Let now \(\C\) be an \(\infty\)-bicategory, \(K\) a marked-scaled simplicial set, 
and	\(g\colon \ovl{K}^{\triangleleft}_{\var} \lrar \C\) a \(\var\)-cone in \(\C\) extending \(f = g_{|K}\). 
Consider the diagram of \(\infty\)-bicategories 
\begin{equation}\label{e:zig-zag} 
\begin{tikzcd}
& \ovl{\C}^{/g}_{\var} \ar[dr, end anchor = north west]\ar[dl, end anchor = north east] & \\
\ovl{\C}^{/g(\ast)}_{\var} && \ovl{\C}^{/f}_{\var} \ .
\end{tikzcd}
\end{equation}
By definition we have that the map \(g\) is a \(\var\)-limit cone if and only if the right diagonal map is an equivalence of \(\infty\)-bicategories. 
We now claim that the left diagonal map is always an equivalence:

\begin{lemma}\label{l:cone}
Let \(\C\) be an \(\infty\)-bicategory and \(K\) a marked-scaled simplicial set equipped with a map \(g\colon \ovl{K} \to \C\).
Then the inclusion \(\{\ast\} \subseteq K^{\triangleleft}_{\var}\) is a \(\var\)-cocartesian equivalence over \(\C\).
\end{lemma}

\begin{proof}
When \(\var=\out\) this map is a pushout of the map \(K \otimes \Del^{\{0\}} \subseteq K \mgr {}^{\flat}\Del^1\), which is an outer cocartesian equivalence over \(\C\) by (the dual of) Lemma~\ref{l:lax-lift}. When
\(\var=\inn\) it is a pushout of the map \(\Del^{\{0\}} \otimes K \subseteq {}^{\flat}\Del^1 \otimes K\) which is an inner cocartesian equivalence over \(\C\) by~\cite[Lemma 4.1.7]{LurieGoodwillie}. 
\end{proof}

\begin{cor}\label{c:cone}
Let \(\C\) be an \(\infty\)-bicategory and \(K\) a marked-scaled simplicial set. Let \(\var \in \{\inn,\out\}\) be a variance parameter. Then for any diagram \(g\colon \ovl{K}^{\triangleleft}_{\var} \to \C\) the projection \(\ovl{\C}^{/g}_{\var} \to \ovl{\C}^{/g(\ast)}_{\var}\) is a trivial fibration.
\end{cor}
\begin{proof}
This follows follows from Lemma~\ref{l:cone} and Corollary~\ref{c:invariance}. 
\end{proof}

Combining Corollary~\ref{c:cone} with Corollary~\ref{c:representables} we find that \(\ovl{\C}^{/g}_{\var} \to \C\) is a model for the \(\var\)-cartesian fibration represented by \(g(\ast)\). In particular, the cone \(g\) determines a map from the \(\var\)-functor represented by \(g(\ast)\) and the \(\var\)-functor associated to the \(\var\)-cocartesian fibration \(\ovl{\C}^{/f}_{\var} \to \C\), which we have identified in Proposition~\ref{p:slice-represent} in terms of (partially) lax natural transformations. By \(\var\)-functor here we mean a functor \(\C^{\eps} \to \Catoo\),
where the variance \(\eps = \emptyset, \co\) of \(\C\) is determined by  the variable \(\var\).
The condition that \(g\) is a \(\var\)-limit cone is exactly the condition that this map is an equivalence. Summarizing this discussion (and using Remark~\ref{r:op-limits} to obtain its dual counterpart of colimits) we may conclude the following:

\begin{cor}\label{c:limit-represents} 
Let \(K\) be a marked-scaled simplicial set and \(\C\)
an \(\infty\)-bicategory.
\begin{enumerate}
\item
An inner cone \(g\colon K^{\triangleleft}_{\inn} \to \C\) extending \(f = g_{|K}\) is an inner limit cone if and only if it determines a natural equivalence of \(\infty\)-categories
\[\Hom_{\C}(x,g(\ast)) \simeq \RNat_K(\ast,\Hom_{\C}(x,f(-))).\]
\item
An outer cone \(g\colon K^{\triangleleft}_{\out} \to \C\) extending \(f = g_{|K}\) is an outer limit cone if and only if it determines a natural equivalence of \(\infty\)-categories
\[\Hom_{\C}(x,g(\ast)) \simeq \LNat_K(\ast,\Hom_{\C}(x,f(-))).\]
\item
An inner cone \(g\colon K^{\triangleright}_{\inn} \to \C\) extending \(f = g_{|K}\) is an inner colimit cone if and only if it determines a natural equivalence of \(\infty\)-categories 
\[\Hom_{\C}(g(\ast),x) \simeq \RNat_{K^{\op}}(\ast,\Hom_{\C}(f(-),x)).\]
\item
An outer cone \(g\colon K^{\triangleright}_{\out} \to \C\) extending \(f = g_{|K}\) is an outer colimit cone if and only if it determines a natural equivalence of \(\infty\)-categories
\[\Hom_{\C}(g(\ast),x) \simeq \LNat_{K^{\op}}(\ast,\Hom_{\C}(f(-),x)).\]
\end{enumerate}
\end{cor}

\begin{rem}\label{r:pseudo-lax}
To avoid confusion, let us recall that the notions of \(\gr\)- and \(\opgr\)- natural transformations appearing in Corollary~\ref{c:limit-represents}  
strongly depend on the marking on \(K\). In particular, when all edges in \(K\) are marked this coincides with the usual notion of a natural transformation (that is, maps in the functor bicategories \(\Fun(K,\Catoo)\) and \(\Fun(K^{\op},\Catoo)\)). On the other hand, if only the degenerate edges are marked these correspond instead to lax (or oplax) transformations. In general, we have that \(2\)-(co)limits represent \(\infty\)-categories of ``partially lax'' natural transformation, where the precise level of pseudo-naturality depends on the marking on \(K\).
\end{rem}

The following corollary is essentially an equivalent reformulation of the statement of Corollary~\ref{c:limit-represents} 
which is more convenient for certain applications:

\begin{cor}\label{c:locally}
Let \(\C\) be an \(\infty\)-bicategory, \(K\) a marked-scaled simplicial set and \(f\colon \ovl{K} \to \C\) a diagram. Then a \(\var\)-cone \(g\colon \ovl{K}^{\triangleleft}_{\var}\) is a \(\var\)-limit cone if and only if for every \(x\in \C\) the restriction map
\[ \Fun^{\coc}_{\C}(K^{\triangleleft}_{\var},\ovl{\C}^{x/}_{\var}) \to \Fun^{\coc}_{\C}(K,\ovl{\C}^{x/}_{\var}) \]
is an equivalence of \(\infty\)-categories. Dually, a \(\var\)-cone \(g\colon \ovl{K}^{\triangleright}_{\var}\) is a \(\var\)-colimit cone if and only if for every \(x\in \C\) the restriction map
\[ \Fun^{\car}_{\C}(K^{\triangleright}_{\var},\ovl{\C}^{/x}_{\var}) \to \Fun^{\car}_{\C}(K,\ovl{\C}^{/x}_{\var}) \] 
is an equivalence of \(\infty\)-categories.
\end{cor}
\begin{proof}
We prove the first claim. The dual version is obtained in a similar manner. Applying Corollary~\ref{c:invariance} we deduce that \(g\) is a \(\var\)-limit cone if and only if for every \(x \in \C\) the induced map
\[ (\ovl{\C}^{/g}_{\var})_x \to (\ovl{\C}^{/f}_{\var})_x \]
is an equivalence of \(\infty\)-categories on the underlying simplicial sets. By Proposition~\ref{p:slice-represent} this is equivalent to saying that for every \(x \in \C\) the induced map
\[ \Nat^{\eps}_{K^{\triangleleft}}(\ovl{x},g) \to \Nat^{\eps}_{K}(x,f) \]
is an equivalence of \(\infty\)-categories, where \(\eps=\gr\) if \(\var=\inn\) and \(\eps=\opgr\) if \(\var =\out\). This statement then translates to the desired results by using the identification of Proposition~\ref{p:fiber}.
\end{proof}

We now proceed to discuss the extent to which the notion of 2-(co)limit proposed above is homotopically sound, that is, depends only on the equivalence type of the marked \(\infty\)-bicategory associated to \(K\). For this, we first note that in contrast to the case of Example~\ref{ex:final} above, and to the situation in the realm of \(\infty\)-categories, in general \(\var\)-limit cones are \emph{not} final objects in \(\C^{/f}_{\var}\). A potentially confusing side-effect of this failure is that \(\var\)-limit cones are not internally characterized inside \(\ovl{\C}^{/f}_{\var}\) when the latter is considered as an abstract \(\infty\)-bicategory, but depend also on the fibration \(\ovl{\C}^{/f}_{\var} \to \C\) (or alternatively, on the associated collection of cartesian edges). We will describe this situation using the language of cofinality in \S\ref{s:cofinal}, see Proposition~\ref{p:limit-is-cofinal} below. For the moment, let us just establish the following bi-product of this phenomenon, which will be enough in order to show that the notion of 2-(co)limit proposed above is homotopically sound in the above sense:

\begin{prop}\label{p:pre-cofinal}
Let \(h\colon K \to L\) be a map of marked-scaled simplicial sets and \(f\colon \ovl{L}\to \C\) a diagram. If the map
\[ \ovl{\C}^{f/}_{\var} \to \ovl{\C}^{fh/}_{\var} \]
is a bicategorical equivalence then it also preserves and detects \(\var\)-colimit cones. 
Dually, if the map
\[ \ovl{\C}^{/f}_{\var} \to \ovl{\C}^{/fh}_{\var} \]
is a bicategorical equivalence then it preserves and detects \(\var\)-limit cones.
\end{prop}

\begin{proof}
We prove the case of coinitiality and limits. The dual case is proven in a similar manner. Let \(\wtl{h}\colon L^{\triangleleft}_{\var} \to K^{\triangleleft}_{\var}\) be the induced map on \(\var\)-cones. For \(g\colon \ovl{L}^{\triangleleft}_{\var} \to \C\) extending \(f\) consider the diagram
\[ \xymatrix{
\ovl{\C}_{/g(\ast)} \ar[d]_{\cong} & \ovl{\C}_{/g}\ar[l]_-{\simeq}\ar[r]\ar[d] & \ovl{\C}_{/f} \ar[d] \\
\ovl{\C}_{/g\wtl{h}(\ast)} & \ovl{\C}_{/g\wtl{h}}\ar[l]_-{\simeq}\ar[r]& \ovl{\C}_{/fh} \ ,
}\]
where the left facing horizontal arrows are trivial fibrations by Corollary~\ref{c:cone} and the left vertical arrow is an isomorphism since \(\wtl{h}\colon L^{\triangleleft}_{\var} \to K^{\triangleleft}_{\var}\) sends the cone point to the cone point. It then follows that the middle vertical map is an bicategorical equivalence.
Now by definition \(g\) is a \(\var\)-limit cone if and only if the top right horizontal map is a bicategorical equivalence and that \(g\wtl{h}\) is a \(\var\)-limit cone if and only if the bottom right horizontal map is a bicategorical equivalence. These two conditions thus become equivalent whenever the rightmost vertical map is a bicategorical equivalence.
\end{proof}

\begin{example}\label{ex:homotopically-sound}
Suppose that \(h\colon K \to L\) is a map of marked-scaled simplicial sets whose underlying scaled map \(\ovl{K} \to \ovl{L}\) is a bicategorical equivalence and such that the marked edges in \(L\) are exactly the images of the marked edges in \(K\). Combining Corollary~\ref{c:invariance} and Remark~\ref{r:equivalence} we then get that for every \(\infty\)-bicategory \(\C\) and every diagram \(f\colon \ovl{L} \to \C\) the restriction functor
\[ \ovl{\C}^{/f}_{\var} \to \ovl{\C}^{/fh}_{\var} \]
is a bicategorical equivalence and hence preserves and detects \(\var\)-limit cones by Proposition~\ref{p:pre-cofinal}.
\end{example}

\begin{rem}\label{r:indexed-by-oo-bicats}
In light of Example~\ref{ex:homotopically-sound}, there is no real restriction of generality in considering limits and colimits just for diagrams indexed by marked-scaled simplicial sets whose underlying scaled simplicial sets are \(\infty\)-bicategories. Indeed, if \(K\) is a marked-scaled simplicial set and \(f\colon \ovl{K} \to \C\) is a diagram valued in an \(\infty\)-bicategory \(\C\), then we can factor \(f\) as \(\ovl{K} \to \ovl{\I} \to \C\) where \(\ovl{\I}\) is an \(\infty\)-bicategory and the map \(\ovl{K} \to \ovl{\I}\) is bicategorical trivial cofibration. Setting \(\I\) to be the marked-scaled simplicial set whose underlying scaled simplicial set is \(\ovl{\I}\) and whose marked edges are the images of the marked edges in \(K\), 
Example~\ref{ex:homotopically-sound} tells us that we may as well replace \(K\) with \(\I\) when considering 2-(co)limits of \(f\).
\end{rem}

We finish this subsection by discussing an analogue of Remark~\ref{r:op-limits} where \((-)^{\op}\) is replaced with \((-)^{\co}\). We note that latter does not admit a convenient model with scaled simplicial sets, and so we will need to formulate the claim using the \(\infty\)-bicategory \(\BiCat\). For this, we note that by Remark~\ref{r:diamond-quillen} the functors \((-)^{\flat} \diamond_{\inn} \Del^0\),\((-)^{\flat} \diamond_{\out} \Del^0\), \(\Del^0 \diamond_{\inn} (-)^{\flat}\) and \(\Del^0 \diamond_{\out} (-)^{\flat}\) are left Quillen functors and hence induce functors on the level of \(\infty\)-categories
\[ \BiCat^{\thi} \to (\BiCat^{\thi})_{\Del^0/} ,\]
which we will denote by the same name.

\begin{lemma}\label{l:coposites}
For an \(\infty\)-bicategory \(\I\) there are natural equivalences
\[(\I \diamond_{\inn} \Del^0)^{\co} \simeq \I^{\co} \diamond_{\out} \Del^0\]
and
\[(\Del^0\diamond_{\inn} \I)^{\co} \simeq \Del^0\diamond_{\out} \I^{\co} ,\]
where both sides are considered as functors \(\BiCat^{\thi} \to (\BiCat^{\thi})_{\Del^0/}\) in the input \(\I\).
\end{lemma}
\begin{proof}
This follows from the identification of the right adjoints of these functors given in Corollary~\ref{c:representables}.
\end{proof}

\begin{cor}\label{c:co-limits}
Let \(\C\) be an \(\infty\)-bicategory, \(\I\) a marked-scaled simplicial set whose underlying scaled simplicial set \(\ovl{\I}\) is an \(\infty\)-bicategory and \(g\colon \ovl{\I}^{\triangleleft}_{\inn} \to \C\) an inner cone. Then \(g\) is an inner limit cone if and only if 
\[g^{\co}\colon (\ovl{\I}^{\triangleleft}_{\inn})^{\co} \simeq (\ovl{\I}^{\co})^{\triangleleft}_{\out} \to \C^{\co} \]
is an outer limit cone, where the identification is done using Lemma~\ref{l:coposites}. A similar statement holds for colimits.
\end{cor}
\begin{proof}
Taking into account the compatibility of cones under taking \((-)^{\co}\) described in Lemma~\ref{l:coposites} and the compatibility of slice fibrations with \((-)^{\co}\) appearing in the last part of Corollary~\ref{c:representables}, the desired claim follows directly from the equivalent criterion for limit cones of Corollary~\ref{c:locally}.
\end{proof}

\subsection{Weighted (co)limits}\label{s:weighted}

In this section we will consider a particular case of \(2\)-(co)limits which corresponds to the classical notion of \emph{weighted (co)limits}. We will then show that this particular case is in some sense completely general: every type of \(2\)-colimit can be described as a weighted colimit with respect to a suitable weight.

\begin{define}
Let \(\C\) be an \(\infty\)-bicategory \(\C\) and \(\I\) be an \(\infty\)-bicategory equipped with a map \(f\colon \I \to \C\). Let \(\var \in \{\inn,\out\}\) be a variance parameter. Let \(\I^{\var}=\I\) if \(\var=\inn\) and \(\I^{\var}=\I^{\co}\) if \(\var=\out\). We define weighted \(\var\)-(co)limits as follows:
\begin{enumerate}
\item
For a weight \(w\colon \I^{\var} \to \Catoo\) classified by a \(\var\)-cocartesian fibration \(p\colon \tilde{\I} \to \I\) we define the \ndef{\(w\)-weighted \(\var\)-limit} of \(f\) to be the \(\var\)-limit of \(f\circ p \colon \tilde \I \to \C\). 
\item
For a weight \(w\colon (\I^{\var})^{\coop} \to \Catoo\) classified by a \(\var\)-cartesian fibration \(p\colon \tilde{\I} \to \I\) we define the \ndef{\(w\)-weighted \(\var\)-colimit} of \(f\) to be the \(\var\)-colimit of \(f\circ p \colon \tilde \I \to \C\).
\end{enumerate}
In all the above cases the \(2\)-(co)limit is taken with respect to the marking \(\tilde{\I}^{\natural}\) consisting of the \emph{\(p\)-(co)-cartesian edges}.
\end{define}

\begin{rem}\label{r:all-variances}
It follows from Corollary~\ref{c:co-limits} that the notion of an \(\I\)-indexed inner limit in \(\C\) with respect to a weight \(w\colon \I \to \Catoo\) is the same as the notion of an \(\I^{\co}\)-indexed outer limit in \(\C^{\co}\) with respect to the weight \(w\colon \I = (\I^{\co})^{\co} \to \Catoo\). Similarly, by Remark~\ref{r:op-limits} these are equivalent to the notions of 
\(\I^{\op}\)-indexed inner colimits in \(\C^{\op}\) and \(\I^{\coop}\)-indexed outer colimits in \(\C^{\coop}\) with respect to the weight \((-)^{\op}\circ w\colon \I^{\co} \to \Catoo\).
\end{rem}

\begin{prop}\label{p:weighted-represented}
Let \(\C\) be an \(\infty\)-bicategory and \(f \colon \I \to \C\) and \(w \colon \I \to \Catoo\) two functors of \(\infty\)-bicategories. 
Then the limit \(\ell \in \C\) of \(f\) weighted by \(w\)
is characterized by 
a natural equivalence of \(\infty\)-categories
\[
	\Hom_{\C}(x,\ell)\simeq \Nat_{\I}\bigl(w,\Hom_{\C}(x,f(-))\bigr)
\]
for \(x \in \C\), where the right hand side denotes the mapping category in the \(\infty\)-bicategory \(\Fun(\I,\C)\) between \(w\) and the restriction along \(f\) of the functor represented by \(x\).
\end{prop}

\begin{rem}
Taking into account the behavior of weighted (co)limits under change of variance we may dualize the statement of Proposition~\ref{p:weighted-represented} and obtain that the outer limit \(\ell\) of a diagram \(f\colon \I \to \C\) weighted by \(w\colon \I^{\co} \to \Cat\) is characterized by a natural equivalence
\[
	\Hom_{\C}(x,\ell)\simeq \Nat_{\I}\bigl(w(-)^{\op},\Hom_{\C}(x,f(-))\bigr)
\]
Similarly, the outer colimit \(c\) of \(f\) weighted by \(w\colon \I^{\op} \to \Cat\) and the inner colimit \(c'\) of \(f\) weighted by \(w' \colon \I^{\coop}\to \Cat\) are characterized by equivalences
\[
	\Hom_{\C}(c,x)\simeq \Nat_{\I^{\op}}\bigl(w(-),\Hom_{\C}(f(-),x)\bigr)
\]
and
\[
	\Hom_{\C}(c,x)\simeq \Nat_{\I^{\op}}\bigl(w'(-)^{\op},\Hom_{\C}(f(-),x)\bigr),	
\]
respectively.
\end{rem}

\begin{rem}\label{r:compare-weighted}
When the indexing \(\infty\)-category \(\I\) is an \(\infty\)-category and the \(\infty\)-bicategory \(\C\) is obtained from a presentable \(\infty\)-category tensored over \(\Catoo\), weighted (co)limits were defined by Gepner--Haugseng--Nikolaus~\cite{GepnerHaugsengNikolausLax} in terms of suitable (co)limits over the twisted arrow category of \(\I\). Their definition was later generalized in~\cite{garcia-cofinality,garcia-stern-enhanced} to diagrams taking values in arbitrary \(\infty\)-bicategories (still indexed by \(\infty\)-categories), by using a universal mapping property of the form appearing in Proposition~\ref{p:weighted-represented}. In particular, Proposition~\ref{p:weighted-represented} shows that the present definition of weighted (co)limits indeed generalizes to one of~\cite{garcia-cofinality} (and hence the one of~\cite{GepnerHaugsengNikolausLax}) to arbitrary indexing \(\infty\)-bicategories.
\end{rem}

We will prove Proposition~\ref{p:weighted-represented} below. Before that, let us establish some terminology that will be convenient for handling weighted (co)limits.
Following Rovelli~\cite{RovelliWeighted}, 
we generalize the definition of the thick join to allow for weights to be considered. In what follows, we fix a variance parameter \(\var \in \{\inn,\out\}\).

\begin{define}
Given a \(\var\)-cocartesian fibration \(p\colon\tilde{\I} \to \I\) 
we define the 
\ndef{\(p\)-weighted \(\var\)-cone of \(\I\)}, 
to be the object at the bottom right corner in the pushout diagrams displayed below:
\[
 \begin{tikzcd}
	\tilde{\I} \ar[d,"p"{swap}] \ar[r] & \Del^0 \diamond_{\var} \tilde{\I}^{\natural} \ar[d]\\
	\I\ar[r] & \Del^0\diamond^p_{\var} \I \ .
 \end{tikzcd}
\]
where \(\tilde{\I}^{\natural}\) is the marked-scaled simplicial set whose underlying simplicial set is \(\I\) and whose marked edges are the \(p\)-cocartesian ones. Similarly, for a \(\var\)-cartesian fibration \(p\colon\tilde{\I} \to \I\) we define \(\I \diamond^p_{\var} \Del^0\) using a similar pushout square.
\end{define}

Now suppose that \(\C\) is an \(\infty\)-bicategory, \(f\colon \I \to \C\) is a diagram and \(p\colon \tilde{\I} \to \I\) is a \(\var\)-cocartesian fibration, considered as a weight. By definition, a candidate for the corresponding weighted \(\var\)-limit of \(f\) is given by a cone of the form \(g\colon \Del^0 \diamond_{\var} \tilde{\I}^{\natural} \to \C\) such that \(g_{|\tilde{\I}} = fp\). In particular, such a \(g\) always factors through the weighted cone \(\Del^0 \diamond_{\var}^p \I\). 

\begin{define}
We will say that a map \(g\colon \Del^0 \diamond_{\var}^p\I \to \C\) is a \emph{weighted \(\var\)-limit cone} if its restriction to \(\Del^0 \diamond_{\var} \tilde{\I}^{\natural}\) is a \(\var\)-limit cone with respect to the marking of \(\tilde{\I}^{\natural}\). Similarly, we say that a map \(g\colon \I \diamond_{\var}^p\Del^0 \to \C\) is a \emph{weighted \(\var\)-colimit cone} if its restriction to \(\tilde{\I}^{\natural} \diamond_{\var} \Del^0\) is a \(\var\)-colimit cone with respect to the marking of \(\tilde{\I}^{\natural}\).
\end{define}

We now observe that the weighted inner cone \(\Del^0 \diamond_{\inn}^p \I\) also appears in the formula for the scaled straightening functor recalled in \S\ref{s:straightening}. 
Specifically, the straightening 
\[\St_{\I}(\tilde{\I})\colon \fCs(\I) \to \s^+\] 
is defined as the restriction to \(\fCs(\I)\) of the functor \(\fCs(\Del^0 \diamond_{\inn}^p \I) \to \s^+\) corepresented by the cone point \(\ast\). Let \(\B\) be an \(\infty\)-bicategory equipped with a bicategorical trivial cofibration \(\Del^0 \diamond_{\inn}^p \I \hrar \B\). Since representable functors unstraighten to slice fibrations (see Remark~\ref{r:unst-rep}) and using the compatibility of unstraightening with base change we conclude that the derived unstraightening of the straightening of \(\tilde{\I} \to \I\) can be identified, up to weak equivalence, with the base change \(\B^{\ast/}_{\inn} \times_{\B} \I \to \I\) of the slice fibration \(\B^{\ast/}_{\inn} \to \B\). In particular, there is a natural equivalence
\[
 \begin{tikzcd}
	\tilde{\I} \ar[rr, "\simeq"] \ar[dr] && \B^{\ast/}_{\inn} \times_{\B} \I \ar[dl] \\
	& \I & 
 \end{tikzcd}.
\]
We may then recover the functor \(w\colon \I \to \Catoo\) classifying \(\tilde{\I} \to \I\) as the composite
\[ \I \to \B \xrightarrow{\Hom_{\B}(\ast,-)} \Catoo \]
where the second map is the functor corepresented by \(\ast\).
Any weighted inner cone \(\Del^0 \diamond^{p}_{\inn} \I \to \C\) taking values in an \(\infty\)-bicategory \(\C\) must extend to a map \(\B \to \C\) in an essentially unique fashion. 

\begin{define}\label{d:weighed-B}
We will say that a map \(g\colon \B \to \C\) \emph{exhibits \(g(\ast)\) as the \(w\)-weighted limit of \(f\)} if its restriction to \(\Del^0 \diamond_{\inn} \tilde{\I}^{\natural}\) is an inner limit cone.
\end{define}

\begin{prop}\label{p:weighted-limit}
Keeping the above notations, a map \(g\colon \B \to \C\) extending \(f\colon \I \to \C\) exhibits \(g(\ast)\) as the \(w\)-weighted limit of \(f\) if and only if for every \(x \in \C\) the restriction map
\[\Fun^{\coc}_{\B}(\B^{\ast/}_{\inn},\ovl{\C}^{x/}_{\inn} \times_{\C} \B) 
 \to \Fun^{\coc}_{\I}(\tilde{\I}^{\natural}, \ovl{\C}^{x/}_{\inn} \times_{\C} \I)\]
is an equivalence of \(\infty\)-categories.
\end{prop}
\begin{proof}
Consider the composite \(g'\colon(\tilde{\I}^{\natural})^{\triangleleft}_{\inn} \to \B \to \C\). By Corollary~\ref{c:locally} we have that \(g'\) is an inner colimit cone if and only if the map
\[ \Fun^{\coc}_{\C}((\tilde{\I}^{\natural})^{\triangleleft}\ovl{\C}^{x/}_{\inn})
	\to \Fun^{\coc}_{\C}(\tilde{\I}^{\natural},\ovl{\C}^{x/}_{\inn})\]
is an equivalence of \(\infty\)-categories for every \(x \in \C\). Since \(g'\) factors by construction through \(g\colon \B \to \C\) we may identify this map with the map
\begin{equation}\label{e:weighted-limit}
 \Fun^{\coc}_{\B}((\tilde{\I}^{\natural})^{\triangleleft},\ovl{\C}^{x/}_{\inn} \times_{\C} \B)
 \to \Fun^{\coc}_{\B}(\tilde{\I}^{\natural},\ovl{\C}^{x/}_{\inn} \times_{\C} \B).
\end{equation}
Now as in the proof of~\cite[Proposition 4.1.8]{LurieGoodwillie} (whose outer counterpart was spelled out in the proof of Proposition~\ref{p:rep-inn}), the \(\infty\)-bicategory \(\B^{\ast/}_{\inn}\) admits a lax contraction \((\B^{\ast/}_{\inn})^{\triangleleft}_{\inn} \to \B^{\ast/}_{\inn}\) to the the point \(\id_\ast\). We may thus extend the map 
\[\tilde{\I}^{\natural} \xrightarrow{\simeq} \B^{\ast/}_{\inn} \times_{\B} \I \to \B^{\ast/}_{\inn} \]
to a map
\[(\tilde{\I}^{\natural})^{\triangleleft}_{\inn} \to \B^{\ast/}_{\inn}\]
sending the cone point to \(\id_\ast\). By Lemma~\ref{l:cone},  
Proposition~\ref{p:rep-inn} and 2-out-of-3 property the last map is an inner cocartesian equivalence over \(\B\). We may thus identify the map~\eqref{e:weighted-limit} with the map
\[\Fun^{\coc}_{\B}(\B^{\ast/}_{\inn},\C^{x/}_{\inn} \times_{\C} \B) \to \Fun^{\coc}_{\I}(\tilde{\I}^{\natural},\C^{x/}_{\inn} \times_{\C} \I),\]
and so the desired statement follow.
\end{proof}

\begin{proof}[Proof of Proposition~\ref{p:weighted-represented}]
The map in Proposition~\ref{p:weighted-limit} fits in a zig-zag diagram
\[\begin{tikzcd}[column sep=0em]
& \Fun^{\coc}_{\B}(\B^{\ast/}_{\inn},\C^{x/}_{\inn} \times_{\C} \B) \ar[dr, end anchor = north west]\ar[dl, end anchor = north east, "\simeq"'] & \\
\C^{x/}_{\inn}\times_{\C} \{g(\ast)\} &&\Fun^{\coc}_{\I}(\tilde{\I}^{\natural},\C^{x/}_{\inn} \times_{\C} \I)
\end{tikzcd}\]
where the left diagonal map is an equivalence by Proposition~\ref{p:rep-inn}.
Applying the straightening-unstraightening equivalence we obtain a zig-zag
\[\begin{tikzcd}[column sep=-1em]
& \Nat_{\B}(\Hom_{\B}(\ast,-),\Hom_{\C}(g(-),x)) \ar[dr, end anchor = north west]\ar[dl, end anchor = north east, "\simeq"'] & \\
\Hom_{\C}(x,g(\ast)) && \Nat_{\I}(w,\Hom_{\C}(f(-),x))
\end{tikzcd},\]
where the right diagonal map is induced by the identification \(w \simeq \Hom_{\B}(\ast,-)_{|\I}\). 
Since all these maps are natural in \(x\) we may now conclude that \(g\) exhibits \(g(\ast)\) as the \(w\)-weighted limit of \(f\colon \I \to \C\) if and only if it exhibits \(g(\ast)\) as representing the functor \(x \mapsto \Nat_{\I}(w,\Hom_{\C}(f(-),x))\), as desired.
\end{proof}

Our next goal is to show that any type of \(2\)-(co)limit can be replaced with an equivalent weighted (co)limit. We will argue this first for inner \(2\)-limits, but will explain afterwards how one can obtain the statement for all variances using the \((\ZZ/2)^2\)-symmetry on \(\BiCat\).

Let \(\I\) be an \(\infty\)-bicategory and \(E\) a collection of edges in \(\I\). Denote by \(\I^+\) the marked-scaled simplicial set having \(\I\) as underlying scaled simplicial set and \(E\) as a set of marked edges. Suppose that \(f\colon \I \to \C\) is a diagram in an \(\infty\)-bicategory \(\C\), to which we can associate the corresponding inner \(2\)-limit with respect to the marking \(E\). Let us now consider the underlying marked-simplicial set of \(\I^+\) as an object in the \(\Beta_{\I}\)-fibered model structure on \((\s^+)_{/\I}\). Taking a fibrant replacement with respect to this model structure, we may find an inner cocartesian fibration \(p\colon\tilde{\I} \to \I\) and a map of marked-scaled simplicial sets \(\iota\colon \I^+ \to \tilde{\I}^{\natural}\) whose underlying map of marked simplicial sets is a \(\Beta_{\I}\)-fibered trivial cofibration. Let \(w\colon \I \to \Catoo\) be the functor classifying \(p\). We then have the following:

\begin{prop}\label{p:inner-is-weighted}
A map \(g\colon \Del^0 \diamond_{\inn}^p \I \to \C\) exhibits \(g(\ast)\) as the \(w\)-weighted inner limit of \(f\) if and only if its restriction to \(\Del^0 \diamond_{\inn} \I^+\) is an inner limit cone.
\end{prop}
\begin{proof}
This follows directly from Proposition~\ref{p:pre-cofinal} and Corollary~\ref{c:invariance}
since the map \(\iota\colon \I^+ \to \tilde{\I}^{\natural}\) is an inner cocartesian equivalence over \(\C\).
\end{proof}

We note that while Proposition~\ref{p:inner-is-weighted} is phrased for inner limits, the same idea applies to all types of 2-(co)limits. This might not be clear at first sight, since we have used the \(\Beta_{\I}\)-fibered model structure to construct the map \(\iota\colon \I^+ \to \wtl{\I}^{\natural}\), and this model structure has been elaborated only in the inner cocartesian case. Nonetheless, the only property of \(\iota\) that we actually needed is that it is an inner cocartesian equivalence whose target is an inner cocartesian fibration. The \(\Beta_{\I}\)-fibered model structure was used to show the existence of such a map, for any given marking on \(\I\). However, given that it exists in the inner cocartesian setting implies that it exists in all four variances. Indeed, it follows from Corollary~\ref{c:op-fibration-2} that for a variance parameter \(\var \in \{\inn,\out\}\), the functor \((-)^{\op}\) sends \(\var\)-cocartesian equivalences/fibrations over \(\I\) to \(\var\)-cartesian equivalences/fibrations over \(\I^\op\), and vice versa. Similarly, the functor \((-)^{\co}\) sends inner (co)cartesian equivalences/fibrations over \(\I\) to outer (co)cartesian equivalences/fibrations over \(\I^{\co}\). The existence of a map of the form \(\iota\) in the inner cocartesian context implies its existence for the other variances as well. We may summarize the resulting statement for all four variances as follows:

\begin{cor}\label{c:marked-is-weighted}
Let \(\var \in \{\inn,\out\}\) be a variance parameter. Let \(\I\) be an \(\infty\)-bicategory and \(\I^+\) a marked-scaled simplicial set whose underlying scaled simplicial set is \(\I\). Then there exists a \(\var\)-(co)cartesian fibration \(p\colon\wtl{\I} \to \I\) and a \(\var\)-(co)cartesian equivalence \(\iota\colon\I^+ \to \wtl{\I}^{\natural}\) over \(\I\). Let \(w\) denote the weight associated to \(p\). Then a map \(g\colon \Del^0 \diamond_{\var}^p \I \to \C\) exhibits \(g(\ast)\) as the \(w\)-weighted \(\var\)-limit of \(g_{|\I}\) if and only if its restriction to \(\Del^0 \diamond_{\var} \I^+\) is a \(\var\)-limit cone. Dually, a map \(g\colon \I \diamond^p_{\var} \Del^0 \to \C\) exhibits \(g(\ast)\) as the \(w\)-weighted \(\var\)-colimit of \(g_{|\I}\) if and only if its restriction to \(\I^+ \diamond_{\var} \Del^0\) is a \(\var\)-colimit cone.
\end{cor}

\begin{rem}\label{r:compare-limit}
Combining Corollary~\ref{c:marked-is-weighted} and Remark~\ref{r:compare-weighted} it follows that for diagrams indexed by a marked \emph{\(\infty\)-category} \(\I\), the notion of 2-(co)limit agrees with the one considered by García in~\cite{garcia-cofinality}.
\end{rem}

\begin{cor}
An \(\infty\)-bicategory \(\C\) admits all small \(2\)-(co)limits if and only if it admits all small weighted limits.
\end{cor}

\subsection{Comparison with model categorical weighted limits}
\label{s:model-categories}

In this section we consider the case where the \(\infty\)-bicategory \(\C\) comes from a model category \(\M\) \emph{tensored} over the marked-categorical model structure on \(\s^+\), 
that is, \(\M\) admits a closed action of \(\s^+\) via a left Quillen bifunctor
\[ \s^+ \times \M \to \M \quad\quad (K,X) \mapsto K \otimes X ,\]
whose adjoints in each variable give a cotensor operation \((K,X) \mapsto X^K \in \M\) and an enrichment \((X,Y) \mapsto \M(X,Y) \in \s^+\). In particular, we may consider \(\M\) as a \(\s^+\)-enriched category.
The full subcategory \(\M^\circ \subseteq \M\) 
spanned by the fibrant-cofibrant objects is then fibrant as 
an \(\s^+\)-enriched category,
and we may consider its scaled nerve \(\M_\infty := \Nsc(\M^{\circ})\), which is an \(\infty\)-bicategory.

Our goal in this section is to show that \(\M_{\infty}\) admits small inner and outer (co)limits and that, furthermore, 
these inner and outer (co)limits can be expressed as weighted homotopy (co)limit of a suitable \(\s^+\)-enriched diagram in \(\M\).
To formulate our statement we will need a few preliminaries. Recall that the category \(\s^+\) is cartesian closed, and so in particular for every two marked simplicial sets \(X,Y\) 
we have an internal mapping object \(\Map(X,Y) \in \s^+\) determined by a universal property of the form 
\[
  \Hom_{\s^+}(K, \Map(X,Y)) \cong \Hom_{\s^+}(K \times X,Y).
\]
If \(\J\) is a small \(\s^+\)-enriched category then the category \((\s^+)^{\J}\) of \(\s^+\)-enriched functors
\(\J \lrar \s^+\) inherits a tensor structure over \(\s^+\) given by \((K \times \F)(i) = K \times \F(i)\),
with a compatible enrichment \(\Nat_\J(\F,\G) \in \s^+\) determined by the universal property 
\[\Hom_{\s^+}(K, \Nat_\J(\F,\G)) \cong \Hom_{(\s^+)^{\J}}(K \times \F,\G).\]
We may consider \(\Nat_\J(\F,\G)\) as the marked simplicial set of natural transformations from \(\F\) to \(\G\).
It also admits an explicit description as the equalizer of the following pair of parallel maps
\[
\begin{tikzcd}
\displaystyle\mathop{\prod}_{i \in \J}\Map\bigl(\F(i),\G(i)\bigr) \ar[r,shift left=.75ex]
\ar[r,shift right=.75ex,swap] & \displaystyle\mathop{\prod}_{i,j\in \J}\Map\bigl(\Hom_\J(i,j)\times\F(i),\G(j)\bigr)
\end{tikzcd}.
\]

Let us now recall the definition of weighted limits and colimits in the setting of \(\s^+\)-enriched categories.
\begin{define}\label{d:weighted}
	Let \(\C\) be a \(\s^+\)-enriched category, \(\F\colon \J \lrar \C\) be an enriched functor and \(W\colon \J \lrar \s^+\) an enriched functor. 
	For an object \(X \in \C\) let us denote by \(\F^{X}\colon \J \lrar \s^+\) the functor given by \(\F^{X}(i) = \C(X,\F(i))\).
	 Given an object \(Z \in \C\) we will say that a natural transformation 
	 \(\tau\colon W \Rightarrow \F^Z\) exhibits \(Z\) as the \ndef{\(W\)-weighted limit} of \(\F\)
	 if for every object \(Y \in \C\) the composed map
	\begin{equation}\label{e:composed-2}
	 \begin{tikzcd}
	  \Hom_\C(Y,Z) \ar[r] & \Nat_\J(\F^Z,\F^Y) \ar[r, "{\tau^*}"] & \Nat_\J(W,\F^Y)
	 \end{tikzcd}
	\end{equation}
	is an isomorphism of marked simplicial sets.
	In this case we will also say that \(\tau\) exhibits \(Z\) as the \ndef{\(W\)-weighted colimit} of \(\F^{\op}\colon \J^{\op} \lrar \C^{\op}\).
\end{define}

\begin{rem}\label{r:cone}
	In the setting of Definition~\ref{d:weighted}, let \(\J^{\triangleleft}_W\) denote the \(\s^+\)-enriched category whose objects are \(\{\ast\} \cup \Ob(\J)\), 
	and such that \(\J^{\triangleleft}_W(i,j) = \J(i,j)\), \(\J^{\triangleleft}_W(\ast,i) = W(i)\), \(\J^{\triangleleft}_W(i,\ast) = \emptyset\) and \(\J^{\triangleleft}_W(\ast,\ast) = \Del^0\),
	and where the composition is given by the functorial dependence of \(W(i)\) on \(i\). 
	Then the data of a natural transformation of the form \(\tau\colon W \Rightarrow \F^X\) is equivalent to 
	the data of a functor \(\G\colon\J^{\triangleleft}_W \lrar \C\) which sends \(\ast\) to \(X\). 
	Furthermore, in this case we may identify the map~\eqref{e:composed-2} with the map
	\[
	 \C(Y,X) \cong \Nat_{\J^{\triangleleft}_W}(\J^{\triangleleft}_W(\ast,-),\G^Y) \to \Nat_\J(W,\F^Y)
	\]
	where the isomorphism is given by the Yoneda lemma and the map is obtained by restriction along \(\J \hrar \J^{\triangleleft}_W\). 
\end{rem}

Let us now fix a \(\s^+\)-enriched category \(\J\) such that the projective model structure on \(\M^{\J}\) exists.
In this case, we have a left Quillen bifunctor
\[
 (\s^+)^{\J} \times \M \lrar \M^{\J} \quad\quad (\G,X)(i) \mapsto \G(i)\otimes X 
\]
where \((\s^+)^{\J}\) is endowed as well with the projective model structure (this one always exists since \(\s^+\) is combinatorial).  
Given enriched functors \(\G\colon \J \lrar \s^+\) and \(\F\colon\J \lrar \M\), let us denote by \(\F^\G \in \M\) 
the image of \(\F\) under the right adjoint of
\[\G \otimes (-)\colon \M \lrar \M^\J.\]
Similarly, given an object \(X \in \M\) and a functor \(\F\colon \J \lrar \M\) let us denote by \(\F^X \in (\s^+)^{\J}\) 
the image of \(\F\) under the right adjoint of 
\[(-) \otimes X\colon (\s^+)^{\J} \lrar \M^{\J}.\]
We note that \(\F^X\) is given by \(\F^X(i) = \Hom_\M(X,\F(i))\) and so this notation is consistent with the notation of Definition~\ref{d:weighted}. 

For an object \(W\) of \((\s^+)^{\J}\) and an object \(\F\) of \(\M^{\J}\),
we have a canonical natural transformation \(W \to \F^{\F^W}\) given by the image
of the identity morphism via the natural isomorphisms
\[
	\Hom_\M\bigl(\F^W, \F^W\bigr) \cong \Hom_{\M^{\J}}\bigl(W \otimes \F^W, \F\bigr) \cong \Nat_\J\left(W, \F^{\F^W}\right).
\]
Replacing the first argument of \(\Hom_\M\bigl(\F^W, \F^W\bigr)\) with any object \(Y\) of \(\M\), we get the isomorphism
\[
	\Hom_\M\bigl(Y, \F^W\bigr) \cong \Nat_\J\bigl(W, \F^Y\bigr).
\]
Unwinding the definitions, this shows the following standard result:

\begin{lemma}\label{l:weighted}
	Let \(\F\colon \J \lrar \M\) and \(W\colon \J \lrar \s^+\) be enriched functors. Then the natural transformation \(W \Rightarrow \F^{\F^W}\) exhibits \(\F^W \in \M\) as the \(W\)-weight limit of \(\F\). 
\end{lemma}

When \(W\) is projectively cofibrant the assignment \(\F \mapsto \F^W\) is a right Quillen functor. 
In this case we will refer to \((\F^{\fib})^W\) (where \((-)^{\fib}\) denotes a projective fibrant replacement)
as the \ndef{\(W\)-weighted homotopy limit} of \(\F\). More generally, it will be useful to adopt the following more flexible terminology:

\begin{define}\label{d:weighted-homotopy-limit}
We will say that a natural transformation \(\tau\colon  W \Rightarrow \F^{Z}\) exhibits \(Z\) as the \emph{\(W\)-weighted homotopy limit} of \(\F\) if for every object \(Y \in \M\) the composed map
\[ \M(Y,Z) \to \Nat_\J(W,\F^Y) \to \Nat_{\J}(W,(\F^{\fib})^{Y})\]
is a weak equivalence in \(\s^+\) (that is, a marked categorical equivalence).
\end{define}

\begin{rem}\label{r:F-fibrant}
In the situation of Definition~\ref{d:weighted-homotopy-limit}, it does not matter which fibrant replacement \(\F^{\fib}\) is used. In particular, if \(\F\) is already projectively fibrant then we can take \(\F^{\fib} = \F\), in which case a natural transformation \(\tau\colon W \Rightarrow \F^{Z}\) exhibits \(Z\) as the \(W\)-weighted homotopy limit of \(\F\) if and only if the map \(\M(Y,Z) \to \Nat_{\J}(W,\F^{Y})\) is a marked categorical equivalence.
\end{rem}

We now consider the following setup. Let \(\J\) be a fibrant \(\s^+\)-enriched category such that the projective model structure on \(\M^{\J}\) exists, and \(\I\) an \(\infty\)-bicategory equipped with a bicategorical equivalence 
\[\phi\colon \fCs(\I) \xrightarrow{\simeq} \J.\] 
Let \(W \colon \J \to \s^+\) be a enriched functor which is fibrant and cofibrant with respect to the projective model structure, \(p\colon\tilde{\I} \to \I\) an inner cocartesian fibration and 
\[\psi\colon \St^{\sca}_{\phi}(\tilde{\I}^{\natural}) \to W\] 
a weak equivalence in \((\s^+)^{\J}\). Finally, fix a diagram \(\F\colon \J \to \M^{\circ}\), and let \(f\colon \I \to \M_{\infty}\) be the adjoint of \(\F\phi\colon \fCs(\I)\to \M^{\circ}\). Let \(w\colon \I \to \Catoo\) be the functor classifying \(p\), which in light of the equivalence \(\psi\) we can identify with the restriction to \(\I\) of the functor 
\[\Nsc(W)\colon \Nsc(\J) \to \Nsc(\M^{\circ}) = \M_{\infty}\] 
induced by \(W\). In this situation we would like to obtain a comparison between the \(W\)-weighted limit of \(\F\) and the \(w\)-weighted inner limit of \(f\).

Let \(\J^{\triangleleft}_{W}\) be as in Remark~\ref{r:cone}. Then the weak equivalences \(\phi\colon \fCs(\I) \to \J\) and \(\psi\colon \St^{\sca}_{\phi}(\tilde{\I}^{\natural}) \to W\) determine a commutative square 
\[ \xymatrix{
&\fCs(\I) \ar[r]^{\phi}_{\simeq}\ar[d] & \J \ar[d] \\
\fCs(\Del^0 \diamond^p_{\inn} \I)\ar@{=}[r] & \fCs([\Del^0 \diamond_{\inn} \tilde{\I}^{\natural}] \coprod_{\tilde{\I}}\I) \ar[r]^-{\simeq} & \J^{\triangleleft}_{W} \\
}\]
with horizontal legs weak equivalences. Now since \(W\) is assumed fibrant \(\I^{\triangleleft}_{W}\) is a fibrant \(\s^+\)-enriched category. Let \(\B := \Nsc(\J^{\triangleleft}_W)\) be its scaled nerve, so that \(\B\) is an \(\infty\)-bicategory and the lower horizontal map in the above square gives a bicategorical equivalence
\[\Del^0 \diamond^p_{\inn}  \I = [\Del^0 \diamond_{\inn} \tilde{\I}^{\natural}] \coprod_{\tilde{\I}}\I \xrightarrow{\simeq} \B .\]
Our comparison statement then takes the following form:

\begin{prop}\label{p:weighted-is-weighted}
Keeping the assumptions and notations above, let \(\F\colon \J \to \M^{\circ}\) be a levelwise fibrant functor, \(Z \in \M^{\circ}\) an object and \(\tau\colon W \Rightarrow \F^Z\) a natural transformation, corresponding to an enriched functor \(\G\colon \J^{\triangleleft}_W \to \M^{\circ}\). Then \(\tau\) exhibits \(Z\) as the \(W\)-weighted homotopy limit of \(\F\) if and only if the adjoint map \(g\colon \B \to \M_{\infty}\) exhibits \(Z\) as the \(w\)-weighted inner limit (in the sense of Definition~\ref{d:weighed-B}) of the composite 
\[f\colon \I \longrightarrow \Nsc(\J) \xrightarrow{\Nsc(\F)} \M_{\infty}.\]
\end{prop}
\begin{proof}
Since \(\F\) is levelwise fibrant Remark~\ref{r:F-fibrant} and Remark~\ref{r:cone} tell us that \(\tau\) exhibits \(Z\) as the \(W\)-weighted homotopy limit of \(\F\) if and only if the map
\begin{equation}\label{e:M-J} 
\M(Y,Z) \cong \Nat_{\J^{\triangleleft}_W}(\J^{\triangleleft}_W(\ast,-),\M(Y,\G(-))) \to \Nat_\J(W,\M(Y,\F(-))) 
\end{equation}
is a marked categorical equivalence for every \(Y \in \M^{\circ}\). We note that since \(Y,Z\) and \(W\) are fibrant-cofibrant the marked simplicial sets appearing in~\eqref{e:M-J} are fibrant. Unstraightening along \(\phi\colon \fCs(\I) \to \J\) and the counit map \(\fCs(\B) \to \I^{\triangleleft}_{W}\) we may identify the map of simplicial sets underlying~\eqref{e:M-J}, up to categorical equivalence, with the map
\[ \Fun^{\coc}_{\B}(\B^{\ast/}_{\inn},\M_{\infty}^{Y/} \times_{\M_{\infty}}\B) \to \Fun^{\coc}_{\I}(\tilde{\I}^{\natural},\M_{\infty}^{Y/} \times_{\M_{\infty}}\I) .\]
This last map is a categorical equivalence if and only if \(g\colon \B \to \M_{\infty}\) exhibits \(Z\) as the \(w\)-weighted inner limit of \(f\) by Proposition~\ref{p:weighted-limit}.
\end{proof}

Using Proposition~\ref{p:inner-is-weighted} we would now like to deduce a result for general \(2\)-limits. Let \(\I,\J\) and \(\phi\colon \fCs(\I) \xrightarrow{\simeq} \J\) be as above, and suppose we are given a marking on \(\I\), that is, a marked-scaled simplicial set \(\I^+\) whose underlying scaled simplicial set is \(\I\). As above, we then let \(W\colon \J \to \s^+\) be a fibrant-cofibrant functor equipped with a weak equivalence
\[\psi\colon \St^{\sca}_{\phi}(\I^+) \xrightarrow{\simeq} W.\]
Let \(\F\colon \J \to \M^{\circ}\) be an enriched diagram and \(f\colon \I \to \Nsc(\J) \to \M_{\infty}\) the resulting \(\infty\)-bicategorical diagram. 
We would like to describe inner 2-limits of \(f\) with respect to the marking \(\I^+\) in terms of the \(W\)-weighted homotopy limit of \(\F\). In particular, given an object \(Z \in \M^{\circ}\) and a natural transformation \(\tau\colon W \Rightarrow \F^{Z}\), we may consider the associated functor \(\G\colon \I^{\triangleleft}_W \to \M^{\circ}\). On the other hand, the map \(\psi\) above encodes a weak equivalence of \(\s^+\)-enriched categories
\[\psi^{\triangleleft}\colon\fCs(\Del^0 \diamond_{\inn} \I^+) \xrightarrow{\simeq} \I^{\triangleleft}_W,\]
by which \(\G\) determines an inner cone
\[ \xymatrix{
\Del^0 \diamond_{\inn} \I^+ \ar[r]\ar@/_1pc/[rr]_{g} & \Nsc(\I^{\triangleleft}_{W}) \ar[r]^-{\Nsc(\G)} & \M_{\infty} \\
}\]

\begin{prop}\label{p:rectification-limit}
Keeping the assumptions and notations above,
the natural transformation \(\tau\) exhibits \(Z\) as the \(W\)-weighted homotopy limit of \(\F\) if and only if \(g\) is an inner limit cone.
\end{prop}
\begin{proof}
Proceeding similarly to \S\ref{s:weighted} we may find an inner cocartesian fibration \(p\colon \tilde{\I} \to \I\) equipped with a map \(\iota\colon \I^+ \to \I^{\natural}\) whose underlying map of marked simplicial sets is a \(\Beta_{\I}\)-fibered trivial cofibration. The \(\St^{\sca}_{\phi}(\I^+) \to \St^{\sca}_{\phi}(\I^{\natural})\) is then a trivial cofibration in \((\s^+)^{\J}\), and hence the map \(\psi\) can be factored as
\[\St^{\sca}_{\phi}(\I^+) \xrightarrow{\simeq} \St^{\sca}_{\phi}(\I^{\natural})  \xrightarrow{\psi'} W.\]
This factorization then translates to a factorization of \(\psi^{\triangleleft}\) as
\[ \fCs(\Del^0 \diamond_{\inn} \I^+) \xrightarrow{\simeq} \fCs(\Del^0 \diamond^p_{\inn} \I) \xrightarrow{\simeq} \I^{\triangleleft}_{W}. \]
Setting \(\B := \Nsc(\I^{\triangleleft}_W)\) we then obtain a sequence of maps
\[ \Del^0 \diamond_{\inn} \I^+ \xrightarrow{\simeq} \Del^0 \diamond^p_{\inn} \I \xrightarrow{\simeq} \B \to \M_{\infty}.\]
By Proposition~\ref{p:weighted-is-weighted} we then have that \(\tau\) exhibits \(Z\) as the \(W\)-weighted limit of \(\F\) if and only if the map \(\Del^0 \diamond_{\inn}^p \I \to \M_{\infty}\) is a weighted inner limit cone. On the other hand, by Proposition~\ref{p:inner-is-weighted}, the latter statement is equivalent to \(g\colon \Del^0 \diamond_{\inn} \I^+ \to \M_{\infty}\) being an inner limit cone, and so the proof is complete.
\end{proof}

\begin{rem}\label{r:outer-limits}
If \(\M\) is a model category tensored over \(\s^+\), then we may consider the \(\s^+\)-tensored model category \(\M^{\co}\), whose underlying model category is the same as \(\M\), and whose tensor structure 
\[ \s^+ \times \M^{\co} \to \M^{\co} \]
is given by \((K,M) \mapsto K^{\op} \times M\). Then the fibrant \(\s^+\)-enriched category \((\M^{\co})^{\circ}\) is obtained from \(\M^{\circ}\) by applying \((-)^{\op}\) on all mapping objects, and so
\[ (\M^{\co})_{\infty} \simeq (\M_{\infty})^{\co} .\]
In addition, for every \(\J\)-enriched category we have a natural isomorphism of \(\s^+\)-enriched model categories 
\[(\M^{\co})^{\J^{\co}} \simeq (\M^{\J})^{\co} \]
and so the projective model structure on \((\M^{\co})^{\J^{\co}}\) exists if and only if the projective model structure on \(\M^{\J}\)-exists.
When \(\J\) is fibrant we may in this case apply Proposition~\ref{p:rectification-limit} to \(\M^{\co}\) and \(\J^{\co}\) (and take \(\I \simeq \Nsc(\J^{\co})\)). Interpreting the resulting statement in terms of \(\M_{\infty}\) and using Corollary~\ref{c:co-limits} we get that \emph{outer} limits in \(\M_{\infty}\) can be computed as weighted limits in \(\M^{\co}\).
\end{rem}

\begin{rem}\label{r:inner-colimits}
If \(\M\) is a model category tensored over \(\s^+\) then the model category \(\M^\op\) (equipped with the opposite model structure in the which the fibrations are what used to be cofibrations, and vice versa), is also canonically tensored over \(\s^+\), whose tensor operation 
\[ \s^+ \times \M^{\op} \to \M^{\op} \]
is now given by the cotensor operation \((K,X) \mapsto X^K\) we had before. Similarly, the cotensor operation on \(\M^{\op}\) is given by the tensor operation on \(\M\), and the enrichment in the two cases coincides to the extent that \(\M^{\op}(X,Y) \cong \M(Y,X)\). In particular, \(\M^{\op}\) is also the opposite of \(\M\) as an \(\s^+\)-enriched category.

In the situation of Proposition~\ref{p:rectification-limit}, 
if the projective model structure on \((\M^{\op})^{\J}\) exists (equivalently, if the \emph{injective} model structure on \(\M^{\J}\) exists), then we may apply that proposition to \(\M^{\op}\). Interpreting the resulting statement in terms of \(\M_{\infty}\) and using Remark~\ref{r:op-limits} we get that inner \emph{colimits} in \(\M_{\infty}\) can be computed as \emph{weighted homotopy colimits} in \(\M\).
\end{rem}

\begin{rem}\label{r:outer-colimits}
Combining Remarks~\ref{r:inner-colimits} and~\ref{r:outer-limits} we may similarly identify \emph{outer colimits} in terms of weighted homotopy colimits in \(\M^{\co}\), assuming the relevant injective model structure exists.
\end{rem}

\begin{cor}\label{c:complete}
Let \(\M\) be an \(\s^+\)-tensored model category such that the projective (resp.~injective) model structure exists on \(\M^{\J}\) for any small \(\s^+\)-enriched category \(\J\) (e.g, \(\M\) is a combinatorial model category). Then
the \(\infty\)-bicategory \(\M_\infty\) admits inner and outer limits (resp.~colimits) indexed by arbitrary small marked-scaled simplicial sets, and these are computed by taking weighted homotopy limits (resp.~colimits) in \(\M\).
\end{cor}
\begin{proof}
By Remark~\ref{r:inner-colimits}, \ref{r:outer-limits} and~\ref{r:outer-colimits} it will suffice to prove the case of inner limits.
Let \(f\colon \I \to \M_{\infty}\) be a diagram. We may then find a fibrant \(\s^+\)-enriched category \(\J\) and a trivial cofibration \(\phi\colon \fCs(\I) \to \J\). Since \(\M^{\circ}\) is fibrant as a \(\s^+\)-enriched category the transposed enriched functor \(\fCs(\I) \to \M^{\circ}\) then factors through an enriched functor \(\F\colon \J \to \M^{\circ}\). We now choose a trivial cofibration \(\St^{\sca}_{\phi}(\I^+) \hrar W\) with \(W\) a projectively fibrant (and cofibrant, since \(\St^{\sca}_{\phi}(\I^+)\) is cofibrant) enriched functor \(\J \to \s^+\). Since \(\F\) is levelwise fibrant by construction, its strict \(W\)-weighted limit is also a weighted homotopy limit. Any \(\tau\colon Z \to \F^{W}\) exhibiting such a \(W\)-weighted homotopy limit gives rise to an inner limit cone on \(f\) by Proposition~\ref{p:rectification-limit}, and so the desired result follows.
\end{proof}

\begin{example}[2-limits of \(\infty\)-categories]\label{ex:2-limits-cat}
A fundamental example of an \(\s^+\)-tensored model category is \(\s^+\) itself, which is in particular combinatorial and hence admits all projective and injective model structures. We may then conclude from Corollary~\ref{c:complete} that \(\Catoo \simeq (\s^+)_\infty\) admits all small \(2\)-(co)limits, and that those can be computed as weighted homotopy limits and colimits in \(\s^+\). On the other hand, given that we know that 2-limits exists, an explicit description of them can be deduced from their universal property of Corollary~\ref{c:limit-represents}. Indeed, taking \(x = \Del^0\) in that corollary we deduce that if \(K\) is a marked-scaled simplicial set and \(f\colon \ovl{K} \to \Catoo\) a diagram then
\[ \mathrm{lim}^{\inn}_K f \simeq \RNat_K(\ast,f) \quad\text{and}\quad \mathrm{lim}^{\out}_K f \simeq \LNat_K(\ast,f),\]
where for \(\var \in \{\inn,\out\}\) we use the notation \(\mathrm{lim}^{\var}_K f\) to indicate the image of the cone point under any \(\var\)-limit cone extending \(f\). Alternatively, using the description of Corollary~\ref{c:fiber}, we may write this identification as
\[ \mathrm{lim}^{\inn}_K f \simeq \Fun^{\coc}_{\ovl{K}}(K,\E^{\inn}_f) \quad\text{and}\quad \mathrm{lim}^{\out}_K f \simeq \Fun^{\coc}_{\ovl{K}}(K,\E^{\out}_{f})^{\op},\]
where \(\E^{\inn}_{f} \to \ovl{K}\) is the inner cocartesian fibration classified by \(f\) and \(\E^{\out}_{f} \to \ovl{K}\) is the outer cocartesian fibration classified by \(f\colon \ovl{K}^{\co} \to \Catoo^{\co} \xrightarrow{(-)^{\op}} \Catoo\).
\end{example}

\begin{rem}
When \(\I\) is a marked \(\infty\)-category the description of 2-limits in \(\Catoo\) given in~\ref{ex:2-limits-cat} was established in~\cite{berman-lax}, and before that when \(\I\) has the minimal marking in~\cite{GepnerHaugsengNikolausLax}.
\end{rem}

\subsection{Cofinality}\label{s:cofinal}

Throughout this section \(\var \in \{\out,\inn\}\) is a fixed variance parameter .

\begin{define}
Let \(h\colon K \to L\) be a map of marked-scaled simplicial sets. We will say that \(h\) is \ndef{\(\var\)-cofinal} if it is a \(\var\)-cartesian equivalence over \(\ovl{L}\). We will say that \(h\) is \ndef{\(\var\)-coinitial} if it is a \(\var\)-cocartesian equivalence over \(\ovl{L}\).
\end{define}

\begin{lemma}\label{l:cofinal-equiv}
Let \(h\colon K \to L\) be a map of marked scaled simplicial sets, \(Y\) a scaled simplicial set and \(g\colon \ovl{L} \to Y\) a map. Then the following statements hold:
\begin{enumerate}
\item
If \(h\) is \(\var\)-cofinal then \(h\) is a \(\var\)-cartesian equivalence over \(Y\).
\item
If \(h\) is a \(\var\)-cartesian equivalence over \(Y\) and \(g\colon \ovl{L} \to Y\) is a \(\var\)-cartesian fibration such that the marked arrows in \(L\) are \(g\)-cartesian, then \(h\) is \(\var\)-cofinal.
\end{enumerate}
The same holds if we replace \(\var\)-cofinal by \(\var\)-coinitial and \(\var\)-cartesian equivalence/fibration by \(\var\)-cocartesian equivalence/fibration.
\end{lemma}
\begin{proof}
The first claim is clear since any \(\var\)-cartesian fibration \(X \to Y\) restricts to a \(\var\)-cartesian fibration \(X'= X \times_Y \ovl{L} \to \ovl{L}\), such that functors to \(X'\) over \(\ovl{L}\) are in bijection with functors into \(X\) over \(Y\). To see the second claim let \(X \to \ovl{L}\) be a \(\var\)-cartesian fibration. Since \(\ovl{L} \to Y\) is now assumed a \(\var\)-cartesian fibration the composed map \(X \to \ovl{L} \to Y\) is also a \(\var\)-cartesian fibration. We may then consider the commutative diagram
\[ \xymatrix{
\Fun^{\car}_{\ovl{L}}(L,X)\ar[r]\ar[d] & \Fun^{\car}_{\ovl{L}}(K,X) \ar[d] \\
\Fun^{\car}_{Y}(L,X)\ar[r]\ar[d] & \Fun^{\car}_{Y}(K,X)\ar[d] \\
\Fun^{\car}_Y(L,\ovl{L}) \ar[r] & \Fun^{\car}_Y(K,\ovl{L})
}\] 
in which the top row can be identified with the induced map from the fiber of the bottom left vertical map over \(\id\colon \ovl{L} \to \ovl{L}\) to the fiber of the bottom right vertical map over \(\ovl{h}\colon \ovl{K} \to \ovl{L}\) (where we note that \(\id\) and \(\ovl{h}\) indeed send marked edges to cartesian edges by assumption). In addition, both these vertical arrows are categorical fibrations, and so their fibers are also homotopy fibers. But under the assumption that \(h\) is a \(\var\)-cartesian equivalence and \(\ovl{L} \to Y\) is \(\var\)-cartesian fibration 
we have that the bottom and middle horizontal arrows are equivalence of \(\infty\)-categories, and hence the top horizontal arrow is an equivalence as well, so that \(h\) is \(\var\)-cofinal.
\end{proof}

Applying Lemma~\ref{l:cofinal-equiv}(1) with \(Y=\ast\) we deduce
\begin{cor}\label{c:quillen-A}
If \(f\colon K \to L\) is a \(\var\)-cofinal map then it induces a marked categorical equivalence upon forgetting the scaling.
\end{cor}

\begin{thm}\label{t:cofinal-characterization}
 	Let \(h \colon K \to L\) be a map of marked-scaled simplicial sets. Then the following are equivalent:
 	\begin{enumerate}
 		\item \(h\) is \(\var\)-coinitial.
 		\item For every diagram \(f\colon \ovl{L} \to \C\) with \(\C\) and \(\infty\)-bicategory, the map \(\ovl{\C}^{/f}_{\var} \to \ovl{\C}^{/fh}_{\var}\) is a bicategorical equivalence.
 		\item For every diagram \(f\colon \ovl{L} \to \C\) with \(\C\) and \(\infty\)-bicategory, the map \(\ovl{\C}^{/f}_{\var} \to \ovl{\C}^{/fh}_{\var}\) preserves and detects \(\var\)-limit cones.
 		\item For every \(\infty\)-bicategory \(\C\) and diagram \(\ovl{f}\colon L^{\triangleleft}_{\var} \to \C\) which is a limit of \(f\), the induced map \[K^{\triangleleft}_{\var} \overset{h^{\triangleleft}}{\to}L^{\triangleleft}_{\var}\overset{\ovl{f}}{\to}{\C}\]
 		is a \(\var\)-limit of \(fh\).
 	\end{enumerate}
 	
Similarly, the dual statement obtained by replacing \(\var\)-cofinal by \(\var\)-coinitial and \(\var\)-limits by \(\var\)-colimits holds as well. 
\end{thm}

\begin{rem}
When \(K\) and \(L\) are marked \emph{\(\infty\)-categories} a notion of cofinality was considered in~\cite{garcia-cofinality}, where a version of Corollary~\ref{c:quillen-A} was also proven. The various equivalent characterizations of Theorem~\ref{t:cofinal-characterization} (together with the comparison of Remark~\ref{r:compare-limit}) then imply that the definition of cofinality and Corollary~\ref{c:quillen-A} above indeed generalize their counterparts in~\cite{garcia-cofinality}.
\end{rem}

\begin{proof}[Proof of Theorem~\ref{t:cofinal-characterization}]
We prove the coinitial case. The dual statement follows 
by passing to opposites and using Remarks~\ref{r:op-limits} and~\ref{r:opposites}.

To begin, we note that \((1)\implies (2)\) by
Lemma \ref{l:cofinal-equiv} and Corollary~\ref{c:invariance},
\((2) \implies (3)\) by Proposition~\ref{p:pre-cofinal} and \((3) \implies (4)\) directly from the definitions.
Now suppose that (4) holds. To prove (1) we need to show that it is a \(\var\)-cocartesian equivalence over \(\ovl{L}\), that is, that for every \(\var\)-cocartesian fibration \(X \to \ovl{L}\) the induced map \(\Fun^{\car}_{\ovl{L}}(L,X) \to \Fun^{\car}_{\ovl{L}}(K,X)\) is an equivalence of \(\infty\)-categories. Let us hence fix a \(\var\)-cocartesian fibration \(p\colon X \to \ovl{L}\). Suppose first that \(\var = \inn\) and let \(\vphi\colon \ovl{L} \to \Catoo\) be the map classifying \(p\) via the \(\infty\)-bicategorical Grothendieck-Lurie correspondence. By Example~\ref{ex:2-limits-cat} the \(\infty\)-bicategory \(\Catoo\) admits all 2-limits, and so the diagram \(\vphi\) extends to an inner limit cone \(\psi\colon L^{\triangleleft}_{\inn} \to \Catoo\). By assumption, the restriction of \(\psi\) to \(K^{\triangleleft}_{\inn}\) is an inner limit cone as well. Let \(Y \to \ovl{L}^{\triangleleft}_{\inn}\) be the inner cocartesian fibration classified by \(\psi\) and consider the commutative diagram of \(\infty\)-categories
\[ \xymatrix{
\Fun^{\car}_{\Catoo}(\ast,(\Catoo)^{\ast/}_{\inn}) \ar@{=}[d] & \Fun^{\car}_{\Catoo}(L^{\triangleleft}_{\inn},(\Catoo)^{\ast/}_{\inn})\ar[l]_-{\simeq}\ar[r]^{\simeq}\ar[d] & \Fun^{\car}_{\Catoo}(L,(\Catoo)^{\ast/}_{\inn}) \ar[d] \\
\Fun^{\car}_{\Catoo}(\ast,(\Catoo)^{\ast/}_{\inn}) & \Fun^{\car}_{\Catoo}(K^{\triangleleft}_{\inn},(\Catoo)^{\ast/}_{\inn})\ar[l]_-{\simeq}\ar[r]^{\simeq}& \Fun^{\car}_{\Catoo}(K,(\Catoo)^{\ast/}_{\inn}) \ ,
}\]
where the left pointing horizontal arrows are equivalences by Lemma~\ref{l:cone} and the right pointing horizontal arrows are equivalences by Corollary~\ref{c:locally} (applied with \(\C=\Catoo\) and \(x=\ast \in \Catoo\)). We may thus conclude that the right most vertical map is an equivalence, and so the restriction map
\[ \Fun^{\car}_{\ovl{L}}(L,L \times_{\Catoo} (\Catoo)^{\ast/}_{\inn})
 \to \Fun^{\car}_{\ovl{L}}(K,\ovl{L} \times_{\Catoo} (\Catoo)^{\ast/}_{\inn}) \]
is an equivalence. 
But the inner cocartesian fibration 
\(\ovl{L} \times_{\Catoo} (\Catoo)^{\ast/}_{\inn} \to \ovl{L}\) is just the base change of the universal inner cocartesian fibration along \(\vphi\), and is hence equivalent to \(p\) itself, see Example~\ref{ex:universal}. We therefore get that the induced map 
\[\Fun^{\car}_{\ovl{L}}(L,X) \to \Fun^{\car}_{\ovl{L}}(K,X)\] 
is an equivalence of \(\infty\)-categories, as desired. The claim in the case of \(\var=\out\) proceeds in a similar manner by taking instead the universal \emph{outer} cocartesian fibration, see again Example~\ref{ex:universal}.
\end{proof}

Finally, the notion of cofinality can also be used to obtain an equivalent characterization of 2-(co)limit cones:

\begin{prop}\label{p:limit-is-cofinal}
Let \(\C\) be an \(\infty\)-bicategory, \(K\) a marked-scaled simplicial set and \(f\colon \ovl{K} \to \C\) a diagram. Then for a \(\var\)-cone \(g\colon \ovl{K}^{\triangleleft}_{\var} \to \C\) extending \(f\) the following are equivalent:
\begin{enumerate}
\item
\(g\) is a \(\var\)-limit cone.
\item
\(g\) is a universal object of \(\C^{/f}_{\var}\) with respect to the projection \(\ovl{\C}^{/f}_{\var} \to \C\).
\item
The inclusion \(\{g\} \subseteq \C^{/f}_{\var}\) is \(\var\)-cofinal.
\end{enumerate}
\end{prop}
\begin{proof}
We first note that \((2) \Leftrightarrow (3)\) by Lemma~\ref{l:cofinal-equiv}. We now show that \((1) \Leftrightarrow (2)\).
Let \(\id_g \in \ovl{\C}^{/g}_{\var}\) be the vertex determined, in the \(\var=\inn\) case, by the composed map
\[ \prescript{\flat}{}\Del^1 \otimes K^{\triangleleft}_{\inn} \to \ovl{K}^{\triangleleft}_{\inn} \to \C \]
where the first map is induced by the map \({}^{\flat}\Del^1 \otimes {}^{\flat}\Del^1 \to {}^{\flat}\Del^1\) given on vertices by \((i,j) \mapsto \max(i,j)\) and the second determined by \(g\), and in the \(\var=\out\) case by the composed map
\[ K^{\triangleleft}_{\out} \mgr \prescript{\flat}{}\Del^1 \to \ovl{K}^{\triangleleft}_{\out} \to \C \]
where the first map is induced by the map \({}^{\flat}\Del^1 \otimes {}^{\flat}\Del^1 \to {}^{\flat}\Del^1\) given on vertices by \((i,j) \mapsto \min(i,j)\) and the second determined by \(g\).
Combining Corollary~\ref{c:rep-inn} and Proposition~\ref{p:rep-out} with Corollary~\ref{c:cone} we may conclude that the inclusion \(\{\id_g\} \subseteq \C^{/g}_{\var}\) is a \(\var\)-cartesian equivalence over \(\C\). It then follows that \(g\in \C^{/f}_{\var}\) is universal over \(\C\) if and only if \(\C^{/g}_{\var} \to \C^{/f}_{\var}\) is a \(\var\)-cartesian equivalence over \(\C\). But the latter is a map of \(\var\)-cartesian fibrations over \(\C\) (with marked arrows being the cartesian ones), and is hence a \(\var\)-cartesian equivalence over \(\C\) if and only if it is a bicategorical equivalence on underlying scaled simplicial sets. This shows that (1) and (2) are equivalent. 
\end{proof}

\begin{cor}\label{c:unicity}
Let \(\C\) be an \(\infty\)-bicategory, \(K\) a marked-scaled simplicial set and \(f\colon \ovl{K} \to \C\) a diagram. Then the sub-bicategory of \(\C^{/f}_{\var}\) spanned by \(\var\)-limit cones and cartesian arrows between them is a contractible Kan complex.
\end{cor}
\begin{proof}
Combine Proposition~\ref{p:limit-is-cofinal} and Proposition~\ref{p:rep-unique}.
\end{proof}

\begin{rem}
We consider Proposition~\ref{p:limit-is-cofinal} as an analogue for 2-(co)limit of the fact that in the \(\infty\)-categorical setting, limit cones are final objects in the \(\infty\)-category of cones, and in particular the singleton inclusion they determine is cofinal. We note however two importance differences between that case and the present one:
\begin{enumerate}
\item
In the \(\infty\)-bicategorical setting, 2-limit cones are not necessarily final objects in \(\C^{/f}_{\var}\) (in particular, being final is not implied by the inclusion \(\{g\} \subseteq \C^{/f}_{\var}\) being \(\var\)-cofinal).
\item
The \(\infty\)-bicategorical notion of cofinality is defined for marked-scaled simplicial sets, and hence depends on a choice of marked edges. In particular, while in the \(\infty\)-categorical case the collection of limit cones can be recovered just from the information of the \(\infty\)-category of cones (as the collection of final objects), here one has to take into account the \(\infty\)-bicategory \(\C^{/f}_{\var}\) of cones together with its collection of marked edges.
\end{enumerate}
\end{rem}

\bibliographystyle{amsplain}
\bibliography{biblio}
\end{document}